\documentclass[12pt]{amsart}
\usepackage[a4paper,margin=1in]{geometry}
\usepackage{amsmath,amsthm,amssymb,mathrsfs}
 \usepackage[normalem]{ulem}

 \usepackage{xcolor}

\usepackage{bm}
\usepackage{bbm} 
\usepackage{xcolor}
\usepackage{graphicx}
\usepackage{multirow}
\usepackage{xfrac}
\usepackage{tikz}
\usetikzlibrary{arrows.meta, positioning}

\usepackage{enumitem}

\usepackage[hyphens]{url}
\usepackage[pdftex,breaklinks]{hyperref}

\setlist[enumerate,2]{
  label=(\alph*),
  ref=\theenumi(\alph*)
}

\makeatletter
\newcommand{\safeparagraph}[1]{%
  \par\smallskip
  \noindent\textbf{#1}\,%
}
\makeatother

\newcounter{AssumpA}

\newenvironment{assump}{
  \begin{enumerate}[
    label={\normalfont\textbf{(A\arabic*)}}, 
    ref={\normalfont(A\arabic*)},
    font=\normalfont
  ]
    \setcounter{enumi}{\value{AssumpA}}
}{
    \setcounter{AssumpA}{\value{enumi}}
  \end{enumerate}
}

\newcounter{AssumpB}

\newenvironment{assumpB}{
  \begin{enumerate}[
    label={\normalfont\textbf{(B\arabic*)}}, 
    ref={\normalfont(B\arabic*)},
    font=\normalfont
  ]
    \setcounter{enumi}{\value{AssumpB}}
}{
    \setcounter{AssumpB}{\value{enumi}}
  \end{enumerate}
}

\newcounter{AssumpBprime}
\newenvironment{assumpBprime}{
  \begin{enumerate}[
    label={\normalfont\textbf{(B\arabic*')}}, 
    ref={\normalfont(B\arabic*')},
    font=\normalfont 
  ]
    \setcounter{enumi}{\value{AssumpBprime}}
}{
    \setcounter{AssumpBprime}{\value{enumi}}
  \end{enumerate}
}

\newcounter{AssumpC}

\newenvironment{assumpC}{
  \begin{enumerate}[
    label={\normalfont\textbf{(C\arabic*)}}, 
    ref={\normalfont(C\arabic*)},
    font=\normalfont
  ]
    \setcounter{enumi}{\value{AssumpC}}
}{
    \setcounter{AssumpC}{\value{enumi}}
  \end{enumerate}
}

\newcounter{AssumpI}

\newcounter{AssumpU}

\renewcommand{\epsilon}{\varepsilon}

\newtheorem{theorem}{Theorem}[section]
\newtheorem{proposition}[theorem]{Proposition}
\newtheorem{lemma}[theorem]{Lemma}

\newtheorem{corollary}[theorem]{Corollary}

\newtheorem{maintheorem}{Theorem}

\newtheorem{maincorollary}[maintheorem]{Corollary}   

\theoremstyle{definition}
\newtheorem{example}[theorem]{Example}

\newtheorem{remark}[theorem]{Remark}

\DeclareMathOperator*{\esssup}{ess\,sup}

\DeclareMathOperator{\lip}{Lip}

\DeclareMathOperator{\diam}{diam}
\newcommand{\qand}{\quad\text{and}\quad}

\numberwithin{equation}{section}

\begin{document}

\title[Linear response for skew-product maps with contracting fibres]{Linear response for skew-product maps\\ with contracting fibres}
\author[J. F. Alves]{Jos\'{e} F. Alves}
\address{Jos\'{e} F. Alves\\ Departamento de Matem\'{a}tica\\ Faculdade de Ci\^encias da Universidade do Porto\\ Rua do Campo Alegre 687\\ 4169-007 Porto\\ Portugal}
\email{jfalves@fc.up.pt} \urladdr{http://www.fc.up.pt/cmup/jfalves}

\author[W. Bahsoun]{Wael Bahsoun}
\address{Wael Bahsoun, Department of Mathematical Sciences, Loughborough University, Loughborough, Leicestershire, LE11 3TU, UK}
\email{W.Bahsoun@lboro.ac.uk}
 \urladdr{https://www.lboro.ac.uk/departments/maths/staff/wael-bahsoun}


\date{\today}

\thanks{JFA was partially supported by The Royal Society Wolfson Visiting Fellowship RSWVF$\backslash$25$\backslash$R1$\backslash$1006 at the Department of Mathematical Sciences, Loughborough University hosted by WB; and by CMUP, member of LASI, which is financed by national funds through FCT -- Funda\c{c}\~{a}o para a Ci\^encia e a Tecnologia, I.P., under project UID/00144/2025 (DOI: \url{https://doi.org/10.54499/UID/00144/2025}) and PTDC/MAT-PUR/4048/2021.  
}


\maketitle
\begin{abstract} We study linear response for families of skew-product dynamical systems with contracting fibres.  Our approach is based on a sectional transfer operator acting on families of probability measures along the fibres. The operator allows to describe invariant measures of the skew-product in terms of sample measures over the base dynamics, regardless of invertibility or non-invertibility of the base map.  Under general assumptions, we establish existence and uniqueness of invariant sample measures and prove their differentiability, with respect to system parameters, in suitable topologies. As an application we obtain linear response for \emph{Bernoulli convolutions}, which are of prime importance in the study of  number theoretic problems and fractals. Another application of our results yields linear response for physical measures of \emph{solenoidal attractors with intermittency}, an example of a hyperbolic system which cannot be handled by traditional transfer operator techniques.
\end{abstract}

\setcounter{tocdepth}{3} 
\addtocontents{toc}{\protect\setcounter{tocdepth}{\value{tocdepth}}}

\tableofcontents

\section{Introduction}
Ergodic theory is a field of mathematics that is concerned with studying dynamical systems from probabilistic point of view. It has deep connections with many areas of mathematics such as analysis, geometry and number theory. Starting with the work of Sinai, Ruelle and Bowen   in the 70's, 
significant advances have been made in the study of uniformly hyperbolic systems from statistical point of view  \cite{S72, BR75}. However, unlike their uniformly hyperbolic counterparts, nonuniformly and partially hyperbolic systems are less understood despite several landmark results, including \cite{ABV00,ADL17,BDL18, D04', Y98}. These systems represent a broader and more complicated class of dynamical systems. When analysing probabilistic aspects of hyperbolic dynamical systems, it is important to understand how stable dynamical invariants, such as SRB measures \cite{AV02, CL22, G18}, mixing rates \cite{FMT07} and entropy \cite{KP89} are. In this domain, a far-reaching direction studies differentiability of such quantities with respect to parameters of the system \cite{KP89}. 

More precisely, given a family  of dynamical systems $T_{\alpha}:M\to M$, with $\alpha$ taking values in an interval $J\subset\mathbb R$, it is natural to ask whether dynamical quantities are differentiable, or not, with respect to the parameter $\alpha$ in a suitable topology. In particular, if for each $\alpha$ the system $T_\alpha$ admits a unique physical measure $\mu_\alpha$, it is natural to ask whether $\alpha\mapsto\mu_\alpha$ is differentiable in a suitable topology. When differentiability holds, it is customary to say that the system has linear response. Such a property was first proved by Ruelle in the framework of uniformly hyperbolic systems \cite{R97} followed by Dolgopyat in the setting of partially hyperbolic systems \cite{D04}. Since then many authors studied such a phenomena and higher order differentiability in a smooth uniformly hyperbolic setting, most notably by Gou\"ezel and Liverani \cite{GL06}, Butterley and Liverani \cite{BL07}, among others. However, beyond uniform hyperbolicity, the situation is more subtle. For instance, even for simple piecewise smooth expanding maps if perturbations (or parameters) are transversal to the topological class, linear response can fail to hold, see the work of Baladi \cite{Ba07, Ba10}, Baladi and Smania \cite{BS08} and related work on unimodal maps \cite{BBS15, BS12,BS21, DS18}.  On the other hand, there are other recent examples where linear response holds despite the perturbation being transversal to the topological class \cite{BG24, Can24}. The situation in smooth \emph{partially hyperbolic} setting is not better. Despite the positive results of Dolgopyat \cite{D04}, Zhang \cite{Z18} studied a class of partially hyperbolic skew products where it was shown that linear response fails and that the map $\alpha\mapsto\mu_\alpha$ is only H\"older for certain partially hyperbolic systems. 

Motivated by examples of partially hyperbolic systems that do not fall within the classes considered by \cite{D04} and \cite{Z18}, we consider in this work skew product maps with uniformly contacting fibres. Our approach is based on a sectional transfer operator acting on families of probability measures along the fibres. The operator allows to describe invariant measures of the skew-product in terms of sample measures over the base dynamics, regardless of invertibility or non-invertibility of the base map. Hence, this approach turns out to be not only fruitful to study linear response for certain partially hyperbolic systems but also to study linear response for measures, called Bernoulli convolutions, used in fractal geometry to analyse number theoretic questions  posed by  Erd{\H o}s \cite{E39, S95}. Under general assumptions, we establish existence and uniqueness of invariant sample measures and prove their differentiability, with respect to system parameters, in suitable topologies. In one of our applications we obtain strong linear response for Bernoulli convolutions covering also the biased case addressed in \cite{PS98}. Another application of our results yields linear response for physical measures of solenoidal attractors with intermittency, an example of a hyperbolic system which cannot be handled by traditional transfer operator techniques \cite{BBC24}.\subsection{General setup and  abstract results}\label{se.general}

Let $ \Omega , X $   be compact Riemannian manifolds (possibly with boundary), and $d_\Omega,d_X$ the respective Riemannian distances.  Let also $m$ denote  the Lebesgue (volume) measure defined on the Borel sets of $\Omega$.  Throughout the paper, we use $\diam(X)$ to denote the diameter of $X$.
Given an interval $J\subset\mathbb R$, we consider a family of   skew-product maps 
$T_\alpha:\Omega\times X \to \Omega\times X$, with  $\alpha\in J$, defined for $(\omega,x)\in\Omega\times X$ by
 \begin{equation} \label{eq.skewmap}
T_\alpha(\omega,x) = \big(f_\alpha(\omega), g_\alpha(\omega,x)\big).
\end{equation}
For each $\omega\in \Omega$, we write $g_{\alpha,\omega}:X\to X$ for the fibre map 
\[
g_{\alpha,\omega}(x) = g_\alpha(\omega,x).
\]
We   introduce below a set of general assumptions that provide an abstract framework for our main results.

\begin{assump}
\item \label{as.contraction}  
\em
There exists a constant $0 < \lambda < 1$ such that 
$$d_X(g_{\alpha,\omega}(x),g_{\alpha,\omega}(y)) \le \lambda d_X(x,y),$$ 
for all $x,y\in X$, all $\omega \in \Omega$  and all $\alpha \in J$.
\end{assump}
 
%

\begin{assump}
\item \label{as.base}
\em
For each base map $f_\alpha:\Omega \to \Omega$, the following hold:
\begin{enumerate}
\em
  \item \em \label{as.a1}
  there exists a countable $m$ modulo zero partition of $\Omega$ into open sets on each of which $f_\alpha$ is a $C^1$ diffeomorphism onto its image;
\em
  \item \em  \label{as.a2}
  the map $f_\alpha$ has an invariant probability measure $\eta_\alpha$ that is absolutely continuous with respect to $m$, with density $\rho_\alpha\in L^1(m)$.
\end{enumerate}
\end{assump}

It is well known (see e.g. \cite{A98}) that for every probability measure $\mu$ invariant under the skew-product map $T_\alpha$, whose marginal on $\Omega$ is $\eta_\alpha$, there exists an essentially unique family of sample measures $\{\nu_{\alpha,\omega}\}_{\omega\in \Omega}$ on $X$ such that
\begin{equation}\label{eq.fullsample}
\mu_\alpha = \int_\Omega \nu_{\alpha,\omega} \, d\eta_\alpha(\omega).
\end{equation}

Under assumptions \ref{as.contraction}-\ref{as.base} we are able to construct these sample measures explicitly, 
in a manner that facilitates the analysis of their differentiability in some  families of skew-product maps. 
To this end, let $\mathcal M_1(X)$ denote the space of Borel probability measures on~$X$. Considering in $\mathcal M_1(X)$   the  Wasserstein-1 distance  $W_1$, it becomes a complete metric space, and since we assume $X$ a compact set, then $\mathcal M_1(X)$ is bounded with  the distance~$W_1$; see Appendix~\ref{se.appendix} for details.
Define the space of probability-valued sections on $\Omega$,
\[
\mathbf{P}
= \big\{
      \boldsymbol{\nu} : \Omega \to \mathcal{M}_1(X)
      \;\big|\;
      \boldsymbol{\nu} \text{ is measurable} 
  \big\},
\]
with the identification of sections that coincide $\eta_\alpha$-almost everywhere.
We consider $\mathbf{P}$ equipped with the essential supremum metric. 
Since $(\mathcal{M}_1(X), W_1)$ is bounded, this metric is well defined, and 
$\mathbf{P}$ is in fact a complete metric space.

We now introduce an operator that provides a natural framework to study the evolution 
of the \emph{sample measures} $\nu_\omega$ under the skew-product dynamics. 
For a given $\boldsymbol{\nu} \in \mathbf{P}$ and for $\eta_\alpha$-almost every $\omega \in \Omega$, 
the \emph{sectional transfer operator} $\mathcal{K}_\alpha$ is defined by
\begin{equation}\label{eq.kalpha} 
(\mathcal{K}_\alpha \boldsymbol{\nu})_{\omega} 
= \sum_{\theta \in f_\alpha^{-1}(\{\omega\})} 
\frac{\rho_\alpha(\theta)}{|\det D f_\alpha(\theta)| \, \rho_\alpha(\omega)} \, (g_{\alpha,\theta})_* \nu_{\theta}.
\end{equation}
Since the invariant density of the base map is preserved under its transfer operator, 
the contributions from all preimages of any point $\omega\in\Omega$ sum to one. 
In this sense, $\mathcal K_\alpha$ acts by pushing forward measures along the preimages, 
weighted according to a \emph{reverse Markov kernel}.
By the measurability of all functions involved in its definition, $\mathcal K_\alpha$ defines a map from~$\mathbf P$ into itself.
It  can be interpreted  as the transfer operator 
associated with the skew-product map~$T_\alpha$ and its definition is motivated by the duality relation  
\[
\langle\!\langle \mathcal K_\alpha \bm\nu, \Phi \rangle\!\rangle
=
\langle\!\langle \bm\nu, \Phi \circ T_\alpha \rangle\!\rangle,
\]
which holds for any bounded measurable function $\Phi:\Omega \times X \to \mathbb R$, 
where the pairing $\langle\!\langle \cdot,\cdot \rangle\!\rangle$ integrates the fiberwise action of $\nu_\omega$ 
against $\Phi(\omega,\cdot)$ with respect to $\eta_\alpha$; see Lemma~\ref{le.duality.general}.

Under the assumptions  above, we obtain the following simple result, which serves primarily as a stepping stone and motivation for the more substantial developments that follow.

\begin{maintheorem}\label{th.mainA}
Assume that   \ref{as.contraction}--\ref{as.base} hold. Then $\mathcal K_\alpha$ defines a contraction on $\mathbf P$, and therefore admits a unique fixed point $\boldsymbol{\nu}_\alpha\in\mathbf P$. Moreover,
\[
\mu_\alpha = \int_\Omega \nu_{\alpha,\omega}\, d\eta_\alpha(\omega)
\]
defines a $T_\alpha$-invariant probability measure on $\Omega\times X$.
\end{maintheorem}

We now introduce an additional set of assumptions under which linear response for the invariant measure $\mu_\alpha$   will be established. 
Our strategy is to first prove the differentiability of $\boldsymbol{\nu}_\alpha$  with respect to the parameter $\alpha$ in some suitable Banach space. 
This derivative   cannot be taken directly in  $\mathbf P$, since the space of probability measures lacks a vector space structure.
We therefore make use of   continuous inclusions 
\begin{equation}\label{eq.inclusions}
\mathcal M_1(X) 
\subset \lip(X)^*\subset C^k(X)^*, 
\end{equation}
for $  k\ge 1$, with respect to the natural norms on these dual spaces; see Appendix~\ref{se.fulbranch} for details. 
We also consider the corresponding subspaces
\begin{equation}\label{eq.inclusionsub}
\lip_0(X)^*\subset C^k_0(X)^*,
\end{equation}
consisting of functionals on the respective dual that vanish on constant functions.
In order to study the linear response of the skew-product measure, we begin by
requiring some regularity for the fixed point $\bm\nu_\alpha$ of the sectional
transfer operator. To begin with, we rewuire continuity with respect
to the base point, which naturally leads us to consider the space of continuous
probability-valued sections
\[
\mathbf P_0
= \bigl\{
      \bm{\nu} : \Omega \to \mathcal M_1(X)
      \;\big|\;
      \bm{\nu} \text{ is } C^0
  \bigr\}.
\]
Under very general conditions, this space is invariant under the action of the
sectional transfer operator.
Given $\bm\nu_\alpha, \bm\nu_{\alpha_0} \in \mathbf P_0$, the difference
quotient
\begin{equation}\label{eq.difquotient}
\frac{\bm\nu_\alpha - \bm\nu_{\alpha_0}}{\alpha - \alpha_0}
\end{equation}
 takes values in the vector space
\begin{equation}\label{eq.spaceE0}
\mathbf{E}_{\mathrm L}^0
=
\bigl\{\,\bm{\nu}:\Omega \to \lip_0(X)^*\;\big|\;
      \bm{\nu} \text{ is $C^0$} \bigr\}.
\end{equation}
which, when equipped with the supremum norm $\|\cdot\|_{\mathrm L}$, is a Banach space.
Aiming for differentiability of the map $\alpha\mapsto \bm\nu_\alpha$ in the Banach space $\mathbf{E}_{\mathrm L}^0$ would however be
too ambitious, given the weak regularity of the associated test functions.
For this reason,  we introduce for~$k \ge 1$ the larger spaces
\begin{equation}\label{eq.spaceEk}
\mathbf{E}_{k}^0
=
\bigl\{\,\bm{\nu}:\Omega \to C_0^k(X)^* \;\big|\; \bm{\nu} \text{ is } C^0\bigr\},
\end{equation}
which are   Banach spaces when endowed with the supremum norm $\|\cdot\|_k$.
Note that, due to the continuous embeddings in~\eqref{eq.inclusions}, the sapce~$\mathbf{E}_{\mathrm L}^0$ is continuously embedded  in $\mathbf{E}_{k}^0$, for every
$k \ge 1$. 

The introduction of these   spaces allows for weaker norms, under
which the verification of the next conditions  becomes
feasible. 
We remark that the expression used to define the sectional transfer operator in \eqref{eq.kalpha} is still valid when applied to sections in these spaces; see equation \eqref{eq.push} in Appendix~\ref{se.appendix} for the general definition of the push-forward.
\begin{assump}
\item \label{as.space}  \em  For all $\alpha\in J$, the operator $\mathcal K_\alpha$ acts as a bounded linear endomorphism of  some~$\mathbf{E}_{k}^0$. Moreover,
\begin{enumerate}
 \em
  \item \label{as.spacea}
   \em
there exists $0<\lambda<1$ such that    $\| \mathcal K_\alpha\|\le \lambda$,  for all $\alpha\in J$;  \em
 \item \label{as.spaceb}
   \em
   $\mathcal K_\alpha$ maps the vector subspace $\mathbf{E}_{\mathrm L}^0$ of $\mathbf{E}_{k}^0$ into itself;
 \em
  \item \label{as.spacec}
   \em
  the map $\alpha \mapsto \mathcal K_\alpha   \bm\nu$ is continuous in $\mathbf{E}_{k}^0$ at $\alpha_0$, for any $\bm\nu\in \mathbf{E}_{\mathrm L}^0$;
 \em
  \item \label{as.spaced}
  \em
  the map $\alpha \mapsto \mathcal K_\alpha \bm\nu_{\alpha_0}$ is differentiable in $\mathbf{E}_{k}^0$ at $\alpha_0$.
\end{enumerate}
\end{assump}

As a by-product of our approach, we obtain the following linear response result for the sample measures.
We emphasise that the specific nature of the operator $\mathcal K_\alpha$ and of the spaces~$\mathbf{E}_{\mathrm L}^0,\mathbf{E}_{k}^0$ is immaterial for this theorem.

\begin{maintheorem}\label{th.mainB}
Assume that  \ref{as.contraction}--\ref{as.space} hold. Then,
the map 
$\alpha\mapsto \bm\nu_\alpha$
is differentiable in the space $\mathbf{E}_{k}^0$ at $\alpha_0$.
\end{maintheorem}

For the linear response of the skew-product measure, we need to impose 
additional regularity conditions both on the density $\rho_\alpha$ of the absolutely continuous 
invariant probability of the base map and on the fixed point $\bm\nu_\alpha$ of the operator 
$\mathcal K_\alpha$. 

\begin{assump}
\item \label{as.response}\em
The map 
\(
\alpha \mapsto \rho_\alpha 
\)
is differentiable in $L^1(m)$ at $\alpha_0$.
\end{assump}

\begin{assump}
\item  \label{as.continuous}
\em For every $\alpha \in J$, we have $\bm\nu_\alpha\in\mathbf P_0$ and 
$
\|\bm\nu_\alpha-\bm\nu_{\alpha_0}\|_k \to 0$ as $\alpha \to \alpha_0.
$
\end{assump}

Under these additional assumptions, we obtain our first general linear response result for the skew-product measure provided by Theorem~\ref{th.mainA}. The order of differentiability $C^k$ of the observables appearing in the subsequent results is implicitly tied to the index $k$ of the space $\mathbf{E}_{k}^0$ for which properties~\ref{as.space} and~\ref{as.continuous} hold.

\begin{maintheorem}\label{co.mainC}
Assume that   \ref{as.contraction}--\ref{as.continuous} hold. 
Then, for every $C^k$  function $\Phi : \Omega \times X \to \mathbb{R}$, the map
\[
\alpha \longmapsto \int_{\Omega}\int_X \Phi(\omega,x)d\nu_{\alpha,\omega}(x)\, d\eta_\alpha(\omega)
\]
is differentiable at $\alpha_0$.
\end{maintheorem}

In order to allow for applications   in which the sample measures cannot be directly derived via Theorem~\ref{th.mainB}, we introduce a kind of minimal space in which differentiability of the sample measures implies differentiability of the skew-product measure. This added flexibility will be crucial for the application to systems whose base map has an inducing scheme, particularly, the intermittent solenoid presented in Subsection~\ref{sub.solenoid}.
To this end, we define for   $k\ge 1$
\begin{equation}\label{eq.spaceE*}
\mathbf L_k^0 = \Big\{
\boldsymbol{\nu} : \Omega \to C^k_0(X)^* \;\big\vert\;
\boldsymbol{\nu} \text{ is measurable and }
\int_\Omega \|\nu_\omega\|_{k^*} \, d\eta_{\alpha_0} < \infty
\Big\},
\end{equation}
and set for each $\boldsymbol{\nu}\in\mathbf L_k^0$,  
\begin{equation}\label{eq.norm*}
\|\boldsymbol{\nu}\|_{\mathbf L_k}
= \int_\Omega \|\nu_\omega\|_{k^*} \, d\eta_{\alpha_0}(\omega),
\end{equation}
where $\|\cdot\|_{k^*}$ denotes the standard norm in $C^k(X)^*$; see Appendix~\ref{se.appendix}.
This  defines a norm on~$\mathbf L_k^0$ which makes it a Banach space. In fact, $\mathbf L_k^0$  coincides with the space of measurable 
functions from $\Omega$ to the Banach space $C_0^k(X)^*$ which are Bochner integrable  with
respect to the measure $\eta_{\alpha_0}$. Although we  do not make this explicit in the notation,  it should be noted that the space $\mathbf{L}_k$ is defined using the measure $\eta_{\alpha_0}$ corresponding to the reference  parameter~$\alpha_0$, the one at which we aim to study differentiability. 

\begin{maintheorem}\label{th.mainC}
Assume that \ref{as.contraction}--\ref{as.base} and \ref{as.response}--\ref{as.continuous} hold, and that the map
\(
\alpha \mapsto \bm{\nu}_\alpha
\)
is differentiable in $\mathbf L_k^0$ at $\alpha_0$.
Then, for every $C^k$ observable $\Phi \colon \Omega \times X \to \mathbb{R}$, the map
\[
\alpha \longmapsto \int_{\Omega}\int_X \Phi(\omega,x)\, d\nu_{\alpha,\omega}(x)\, d\eta_\alpha(\omega)
\]
is differentiable at $\alpha_0$.
\end{maintheorem}

Since convergence in the space $\mathbf{E}_{k}^0$      implies convergence in the space~$\mathbf L_k^0$, Theorem~\ref{co.mainC} is a direct consequence of Theorem~\ref{th.mainB} and Theorem~\ref{th.mainC}.
The proof of Theorem~\ref{th.mainB}  will be presented  in Subsection~\ref{sub.sample}.
The proof of Theorem~\ref{th.mainC} will be presented in Subsection~\ref{sub.skewresponse}.

\begin{remark} \label{re.justice}
Theorem~\ref{th.mainC} is  our result of most general applicability  concerning linear response for skew-product maps.   
The assumption of harder verification being in general  the differentiability of the map
\(
\alpha \mapsto \bm{\nu}_\alpha,
\)
even if only  in the space $\mathbf L_k^0$. 
 Recall that~\ref{as.space}  is used in Theorem~\ref{th.mainB} to obtain this conclusion  in the space $\bm{\mathrm E}_{k}^{0}$, and for the general class of systems considered in Subsection~\ref{sub.mapsfull}, this differentiability follows indeed  from Theorem~\ref{th.mainB}. However, this approach cannot be applied to the intermittent solenoid considered in Subsection~\ref{sub.solenoid}, essentially because Theorem~\ref{th.mainB} requires the base dynamics to be at least of class $C^2$, a regularity assumption that is not satisfied by those systems.

As an alternative, in Subsection~\ref{sub.basinducing} we present an approach for skew-product systems whose base map admits an inducing scheme. This  allows us to deduce the differentiability of the map
\(
\alpha \mapsto \bm{\nu}_\alpha
\)
in the space $\mathbf L_k^0$ from the differentiability of the sample invariant measures of the induced skew-product in $\mathbf E_k^0$. For the induced system, this latter property can be established by applying Theorem~\ref{th.mainB}.
We   develop first in Subsection~\ref{sub.mapsfull} a theory for skew-product maps whose base dynamics are given by piecewise expanding full-branch maps.
\end{remark}

In Table~\ref{tab:spaces} below, we summarise  the section spaces and norms used in the analysis throughout this work; see Appendix~\ref{se.appendix} for the norms in the dual spaces used in the fibres.

\begin{table}[h!]
\centering
\renewcommand{\arraystretch}{1.4}
\begin{tabular}{| c | l | l |}
\hline
\multicolumn{1}{|c|}{\textbf{Space}} &
\multicolumn{1}{c|}{\textbf{Definition}} &
\multicolumn{1}{c|}{\textbf{Metric}} \\
\hline

$\mathbf{P}$ 
& $\{\bm\nu: \Omega \to \mathcal{M}_1(X) \mid \bm\nu \text{ is measurable} \}$ 
& \multirow{2}{*}{$d_{\mathrm{P}}(\bm\mu,\bm\nu)
= \displaystyle\esssup_{\omega\in\Omega} \|\mu_\omega-\nu_\omega\|_{\mathrm{Lip}^*}$} \\
\cline{1-2}

$\mathbf{P}_0$ 
& $\{\bm\nu: \Omega \to \mathcal{M}_1(X) \mid \bm\nu \text{ is $C^0$} \}$ 
& \\

\hline

$\mathbf{E}_{\mathrm L}^0$ 
& $\{\bm\nu: \Omega \to \mathrm{Lip}_0(X)^* \mid \bm\nu \text{ is $C^0$} \}$ 
& \multirow{2}{*}{$\|\bm\nu\|_{\mathrm L}
= \displaystyle\sup_{\omega\in\Omega} \|\nu_\omega\|_{\mathrm{Lip}^*}$} \\
\cline{1-2}

$\mathbf{E}_{\mathrm L}$ 
& $\{\bm\nu: \Omega \to \mathrm{Lip}(X)^* \mid \bm\nu \text{ is $C^0$} \}$ 
& \\

\hline

$\mathbf{E}_{k}^0$ 
& $\{\bm\nu: \Omega \to C^k_0(X)^* \mid \bm\nu \text{ is $C^0$} \}$ 
& \multirow{2}{*}{$\|\bm\nu\|_{k}
=\displaystyle \sup_{\omega\in\Omega} \|\nu_\omega\|_{k^*}$} \\
\cline{1-2}

$\mathbf{E}_k$ 
& $\{\bm\nu: \Omega \to C^k(X)^* \mid \bm\nu \text{ is $C^0$} \}$ 
& \\

\hline

$\mathbf{E}_k'$ 
& $\{\bm\nu: \Omega \to C^k(X)^* \mid \bm\nu \text{ is $C^1$} \}$ 
& $\|\bm\nu\|_{k}' = \|\bm\nu\|_{k} + \|\bm\nu'\|_{k}$ \\
\hline

$\mathbf{L}_{k}^0$ 
& $\{\bm\nu: \Omega \to C^k_0(X)^* \mid \bm\nu \in L^1(\eta_{\alpha_0}) \}$ 
& $\|\bm\nu\|_{\mathbf L_k}
=  \int_\Omega\|\nu_\omega\|_{k^*} \, d\eta_{\alpha_0}(\omega)$ \\
\hline

\end{tabular}
\caption{\small Spaces of sections and metrics.}
\label{tab:spaces}
\end{table}

In the two applications  of our results, we  do not consider values of~\(k\) beyond \({k=3}\) for the spaces \(\mathbf{E}_{k}^0\) and beyond \(k=2\) for the spaces \(\mathbf{E}_k\).
The spaces \(\mathbf{E}_{\mathrm L}\),~\(\mathbf{E}_k\), and~\(\mathbf{E}_k'\) have not appeared so far; they will be introduced later as auxiliary spaces in the statements and proofs of some of our results.
The distinction between \(\mathbf{E}_{\mathrm L}\),~\(\mathbf{E}_k\) and \(\mathbf{E}_{\mathrm L}^0\),~\(\mathbf{E}_{k}^0\) is subtle, though essential: the latter consist of those sections whose values are, fibrewise, functionals vanishing on constant functions.
This choice reflects two complementary requirements:
on the one hand, the operator \(\mathcal K_\alpha\) must act as a contraction on the spaces~\(\mathbf{E}_{\mathrm L}^0\),~\(\mathbf{E}_{k}^0\); on the other hand, the spaces \(\mathbf{E}_k\),\(\mathbf{E}_k'\) must be sufficiently large to intersect \(\mathbf P\) and to contain the fixed point of \(\mathcal K_\alpha\) in \(\mathbf P\). Finally, we emphasise that the use of the Lipschitz norm in the definition of the spaces \(\mathbf{E}_{\mathrm L}^0\),\(\mathbf{E}_{\mathrm L}\), rather than the \(C^0\) norm, is essential: the operator~\(\mathcal K_\alpha\) acts fibrewise as a contraction in the corresponding \(\lip_0(X)^*\) norm, a property which generally fails in the \(C^0_0(X)^*\) setting.

\subsection{Base maps with full   branches}\label{sub.mapsfull} 
This section is devoted to skew-product systems whose base dynamics are given by piecewise smooth full branch maps. This setting will be particularly relevant for the applications presented in Section~\ref{se.applications}. Our goal is to introduce verifiable criteria for families of skew-product maps
\[
T_\alpha:\Omega \times X \to \Omega \times X, \qquad \alpha \in J,
\]
of the form~\eqref{eq.skewmap}, under which several of the assumptions of Section~\ref{se.general} can be checked more directly.
For simplicity, we assume throughout that $\Omega$ is a compact interval in $\mathbb{R}$ and that $X$ is a compact, convex subset of $\mathbb{R}^d$ for some $d \ge 1$. 
While this framework admits natural generalisations, the present setting is sufficient for all the applications considered in this paper.

We assume there exists an at most countable family  $\{\Omega_{\alpha,i}\}_{i \in \mathbb I}$  of smoothness domains for each base map $f_\alpha \colon \Omega \to \Omega$. We now introduce a collection of additional assumptions on the family $\{T_\alpha\}_{\alpha \in J}$.

\begin{assumpB}
\item \label{as.full} \em
For each $i \in \mathbb I$ and $\alpha\in J$, the restriction $f_\alpha|_{\Omega_{\alpha,i}} \colon \Omega_{\alpha,i} \to \Omega$ is onto, and its local inverse branch has a $C^1$ extension $\theta_{\alpha,i} \colon \Omega \to \Omega_{\alpha,i}$  satisfying  $|\theta_{\alpha,i}'|\le 1$.
\item
\label{as.theta} 
For each $i\in\mathbb I$, the map $  \alpha \mapsto  \theta_{\alpha,i}\in C^1(\Omega,\Omega)$ is  continuous,     and there exists $C_\theta>0$ such that 
$$  
\sum_{i\in\mathbb I} \sup_{\alpha\in J} \big\| \theta_{\alpha,i} \big\|_{C^1}\le C_\theta.$$
 \em
\end{assumpB}


\begin{assumpB}
\item \label{as.g} \em 
For each $\alpha\in J$ and $\omega\in \Omega$, we have $g_{\alpha,\omega}\in C^3(X)$. Moreover,
\begin{enumerate}
\em \item \em for each $i\in\mathbb I$,  the map $\alpha \mapsto g_{\alpha,\theta_{\alpha,i}} \in C^1( \Omega\times X, X)$ is continuous;
\em \item \em  there exists $\lambda<1$  such that defining for $i\in \mathbb I$,  $\alpha\in J$ and  $k=1,2,3$ \label{as.gb}   
$$L_k=L_k(\alpha,i)=\sup_{\omega\in\Omega}\sup_{x\in X}\|D^kg_{\alpha,\theta_{\alpha,i}(\omega)}(x) \|,
$$
one has
$$\max\left\{ L_1 + L_2 + L_3, \, L_1^2 + 3L_1L_2  \right\}\le \lambda;$$ 
\em \item \em   there exists $C_g>0$ such that 
$$\sum_{i\in\mathbb I} \sup_{ \alpha\in J }  \| g_{\alpha,\theta_{\alpha,i}} \|_{C^1}  \le C_g.$$
\end{enumerate}  
\end{assumpB}

For each $i \in \mathbb{I}$ and $\alpha \in J$, we introduce the \emph{weight function} $p_{\alpha,i}$, defined  for $\eta_\alpha$-almost every $\omega \in \Omega$ by
\begin{equation}\label{eq.pweights}
p_{\alpha,i}(\omega)
=
\frac{\rho_\alpha\bigl(\theta_{\alpha,i}(\omega)\bigr)}
{|  f_\alpha'\bigl(\theta_{\alpha,i}(\omega)\bigr)|\, \rho_\alpha(\omega)}.
\end{equation}
With this notation, the sectional transfer operator $\mathcal K_\alpha$ defined in~\eqref{eq.kalpha} admits the explicit representation
\begin{equation}\label{eq.kalpha2}
(\mathcal K_\alpha \bm{\nu})_\omega
=
\sum_{i \in \mathbb{I}}
p_{\alpha,i}(\omega)\,
\bigl(g_{\alpha,\theta_{\alpha,i}(\omega)}\bigr)_*
\nu_{\alpha,\theta_{\alpha,i}(\omega)} .
\end{equation}
The \emph{reverse Markov kernel} property translates in this setting to the simple relation
\begin{equation}\label{eq.RMK}
\sum_{i \in \mathbb{I}} p_{\alpha,i}(\omega) = 1.
\end{equation}
Our final assumption of this part concerns the continuity and summability properties of these weight functions. Throughout the text, we use a dot over a function to denote its partial derivative with respect to the parameter $\alpha \in J$.
  \begin{assumpB}
\item \label{as.p} \em For each $i\in\mathbb I$,  the map $  \alpha \mapsto   p_{\alpha,i}\in C^1(  \Omega,\mathbb R)$ is   continuous,   and there exists $C_p>0$ such that 
$$ 
\sum_{i\in\mathbb I} \sup_{\alpha\in J}\|  p_{\alpha,i} \|_{C^1}\le C_p\qand \sum_{i\in\mathbb I} \sup_{\alpha\in J}\| \dot p_{\alpha,i} \|_{C^0}\le C_p.
$$
 \end{assumpB}

We refer the reader to Remark~\ref{re.bell} for a detailed explanation of the role played by the constant
$0<\lambda<1$   in~\ref{as.gb}, and for the motivation behind its   choice. 
Note that if the maps $g_{\alpha,\omega}$  are affine then \ref{as.gb} means that $g_{\alpha,\omega}$ are uniform smooth contractions  in $X$.
Clearly \ref{as.gb}    implies \ref{as.contraction} in this setting.
The constant $\lambda<1$ is crucial in establishing the contraction properties of the sectional transfer operator
on suitably chosen Banach spaces.
In the present setting, we take $\mathbf{E}_3$ as the main reference space where the linear response of the sample measures is obtained.

 \begin{maintheorem}\label{th.mainD}
Assume that  \ref{as.base} and \ref{as.full}--\ref{as.p} hold. Then,   
\begin{enumerate}
\item $\mathcal K_\alpha$
has a  unique fixed point $\bm\nu_\alpha$ in $\mathbf P$;
\item   $\bm\nu_\alpha\in\mathbf P_0$ and 
$
\|\bm\nu_\alpha-\bm\nu_{\alpha_0}\|_3 \to 0$ as $\alpha \to \alpha_0;
$
\item $\bm\nu_\alpha:\Omega\to C^2(X)^*$ is a $C^1$ function;
\item the map 
 $\alpha\mapsto\bm\nu_\alpha$
 is differentiable in $\mathbf{E}_3^0$ at $\alpha_0$.
\end{enumerate}
\end{maintheorem}

Since condition~\ref{as.gb} implies~\ref{as.contraction}, the existence and uniqueness of a fixed point for~$\mathcal K_\alpha$ in~$\mathbf P$ follow from Theorem~\ref{th.mainA}.
The remaining conclusions are proved in Section~\ref{se.fulbranch}. 
Specifically, the second conclusion follows from Corollary~\ref{co.KP0} and Corollary~\ref{co.contfp}. 
The differentiability of~$\bm\nu_\alpha$ is established in Proposition~\ref{pr.differfixed} as a by-product of the argument used to show that the assumptions of~\ref{as.space} are satisfied, allowing Theorem~\ref{th.mainB} to be applied for the final conclusion.

\begin{maincorollary}\label{th.piecewise2}
Assume that   \ref{as.base}, \ref{as.response} and \ref{as.full}--\ref{as.p} hold and $\bm\nu_\alpha$ is the unique fixed point of $\mathcal K_\alpha$ in $\mathbf P$. Then, for every $C^3$ function $\Phi : \Omega \times X \to \mathbb{R}$, the map
\[
\alpha \longmapsto \int_{\Omega}\int_X \Phi(\omega,\cdot)d\nu_{\alpha,\omega}\, d\eta_\alpha(\omega)
\]
is differentiable at $\alpha_0$.
\end{maincorollary}

Since convergence in $\mathbf{E}_3^0$ implies convergence in $\mathbf L_3$, 
Corollary~\ref{th.piecewise2} follows from Theorems~\ref{th.mainC} and~\ref{th.mainD} once assumption~\ref{as.continuous} has been verified for $k=3$. This verification will be carried out in Section~\ref{se.fulbranch}, specifically in Corollaries~\ref{co.KP0} and~\ref{co.contfp}.
We conclude this part with a remark concerning the choice of the spaces $\mathbf{E}_3^0$ and $\mathbf{E}_{\mathrm L}^0$. 

\begin{remark}
In the present setting, to obtain the differentiability with respect to the parameter required by \ref{as.spaced}, there are two points where the $C^2$ regularity of the test functions is essential:
\begin{enumerate}
    \item the differentiability of the invariant section (Proposition~\ref{pr.differfixed});
    \item the   continuity of the push-forward operator (Lemma~\ref{le.salvation}).
\end{enumerate}
In both cases, one cannot expect to improve the conclusions if the test functions are only assumed to be $C^1$; see Remark~\ref{re.independent1} and Remark~\ref{re.c2}.

However, since the first condition above is trivially satisfied when the invariant section is constant, and the second already provides the continuity required in~\ref{as.spacec} in the space $\mathbf E_2$, it is natural to ask whether differentiability of the invariant section could be obtained in $\mathbf E_2$, thereby strengthening the conclusion in the case where the section does not depend on the base points.

Even in this simplified scenario, this is not possible. Indeed, when verifying \ref{as.spaced}, Lemma~\ref{le.salvation} is applied in a crucial way to the gradient of a test function, which requires the test function to be of class $C^3$; see the estimates for the terms $C_\epsilon$ and $D_\epsilon$ in the proof of Proposition~\ref{pr.diffKi}. Assuming the invariant section is constant has absolutely no effect on the derivation of these estimates.

Finally, to avoid exceeding $C^2$ regularity for the test functions, it is essential that the continuity in~\ref{as.spacec} is only required for elements in the space $\mathbf{E}_{\mathrm L}^0$. This allows the application of Lemma~\ref{le.salvation} while keeping the fibres at the level of $\lip(X)^*$, ensuring continuity of the push-forward operator when tested against $C^2$ observables.
\end{remark}

\subsection{Base maps with inducing schemes}\label{sub.basinducing}

This section is largely motivated by Remark~\ref{re.justice}, and its purpose is to obtain the conclusion of Theorem~\ref{th.mainB} under milder assumptions, by indirectly applying Theorem~\ref{th.mainB} to an appropriate inducing scheme.
As before, we consider a family of skew-product maps
\[
T_\alpha:\Omega \times X \to \Omega \times X, \qquad \alpha \in J,
\]
%
For many  base maps $f_\alpha$, one can construct an \emph{inducing scheme}, that is, there exist a subset $\bar\Omega \subset \Omega$ and a \emph{first return time}  function $\tau_\alpha:\bar\Omega \to \mathbb{N}$ such that the induced base map
\[
\bar f_\alpha : \bar\Omega \to \bar\Omega
\]
is defined for each $\omega \in \bar\Omega$ by
\[
\bar f_\alpha(\omega) = f_\alpha^{\tau_\alpha(\omega)}(\omega).
\]

 \begin{assumpC}\item 
\label{as.induce}
\em For each $\alpha\in J$, the maps $\bar f_{\alpha}$ satisfies:
  \em
\begin{enumerate}
 \item\em  there exists a countable $m$ modulo zero partition $\{\bar\Omega_{\alpha,i}\}_{i\in \mathbb I}$ of $\bar\Omega$ such that ${\bar f_\alpha|_{\bar\Omega_{\alpha,i}} \colon \bar\Omega_{\alpha,i} \to \bar\Omega}$ is  a $C^1$ diffeomorphism for all $i\in \mathbb I$; 
 \em \item \em there exists $\sigma>1$ such that $|\bar f_\alpha' |_{\bar\Omega_{\alpha,i}}|\ge \sigma$ for all $i\in \mathbb I$;
 
\em   \item\em  there exist $C,\gamma>0$ such that, for all    $i\in \mathbb I$ and all $x,y\in\bar\Omega_{\alpha,i}$
 $$
 \log\frac{\bar f'_\alpha(x)}{\bar f_\alpha' (y)}\le
 C \left|\bar f_\alpha(x)-\bar f_\alpha(y)\right|^\gamma.
 $$
 \em \item\em  $\tau_\alpha$ is integrable with respect to Lebesgue measure on $\bar\Omega$.
  \end{enumerate}
\end{assumpC}

Under these conditions, the induced base map $\bar{f}_\alpha$ has a unique absolutely continuous invariant probability measure $\bar{\eta}_\alpha$ with a density $\bar\rho_\alpha\in L^1(\bar\Omega)$. Moreover, an absolutely continuous invariant probability measure $\eta_\alpha$ for the original base map $f_\alpha$ can be obtained via the \emph{unfolding formula}
 \begin{equation}\label{eq.measures}
\eta_\alpha = \frac{1}{E} \sum_{k=0}^\infty (f_\alpha^{k})_* \big(\bar{\eta}_\alpha\vert \{\tau_\alpha > k\}  \big),\qquad E=\int_{\bar\Omega} \tau_\alpha \, d\bar{\eta}_\alpha.
\end{equation}
This classical  procedure allows one to transfer many statistical properties  from the induced map to the original system.
 The induced skew-product map is
\begin{equation}\label{eq.induceT}
\bar T_\alpha : \bar\Omega \times X \to \bar\Omega \times X, \qquad
\bar T_\alpha(\omega,x) = (\bar{f}_\alpha(\omega), \bar{g}_\alpha(\omega,x)),
\end{equation}
with  the induced fibre map
\[
\bar g_\alpha:\bar\Omega\times X\to X
\]
given for any $(\omega,x)\in \bar\Omega\times X$ by
\begin{equation*}
\bar g_\alpha(\omega,x)
=
g_{\alpha,f_\alpha^{\tau_\alpha(\omega)-1}(\omega)}
\circ \cdots \circ
g_{\alpha,f_\alpha(\omega)}
\circ
g_{\alpha,\omega}(x).
\end{equation*}
Equivalently, $\bar g_\alpha(\omega,\cdot)
=
g_{\alpha,\omega}^{(\tau_\alpha(\omega))}$ with
\begin{equation}\label{eq:g(k)-def}
g_{\alpha,\omega}^{(k)}
=
g_{\alpha,f_\alpha^{k-1}(\omega)}
\circ \cdots \circ
g_{\alpha,\omega}.
\end{equation}

Our goal is now to provide conditions under which some of the results of the previous sections can be applied, yielding linear response for the induced skew-product maps $\bar T_\alpha$ and, from this, deriving the linear response of the original skew-product maps $T_\alpha$. 
To implement this inducing  strategy, we will work with certain Banach spaces from Table~\ref{tab:spaces}, both on the base space $\Omega$ and on the inducing domain $\bar\Omega$. The distinction between these two settings will be indicated by overlining the corresponding space for the induced system and keeping it not overlined for the original one.

Whenever needed, we regard a section $\bar{\bm\nu}$  on $\bar\Omega$ as a section on $\Omega$ by extending it by zero outside $\bar\Omega$. We define the \emph{sectional unfolding operator} $\mathcal{U}_\alpha$ which constructs a full section   on~$\Omega$ from a section~$\bar{\boldsymbol{\nu}}$ on the induced domain $\bar\Omega$ by
 \begin{equation}\label{eq.unfoldopera}
 \mathcal{U}_\alpha(\bar{\boldsymbol{\nu}}) = \sum_{k=0}^\infty \mathcal{K}_\alpha^k \bigl( \bar{\boldsymbol{\nu}} \cdot \mathbf{1}_{\{\tau_\alpha > k\}} \bigr),
\end{equation}
where $\mathcal{K}_\alpha$ is the sectional transfer operator associated with   $T_\alpha$ and~$\mathbf{1}_A$ stands for   the \emph{indicator function} of any set $A \subset \Omega$. Note that the sum in~\eqref{eq.unfoldopera} is pointwise finite, and therefore   well defined as a section on $\Omega$.

The main result of this section will be obtained under conditions that, in particular, ensure the existence of a unique fixed point $\bar{\bm{\nu}}_\alpha$ in $ \bar{\mathbf P}$ for the sectional transfer operator~$\bar{\mathcal K}_\alpha$ associated with the induced system $\bar T_\alpha$. We  require certain regularity properties of the unfolding operator, which we state next. As has been customary, these include conditions imposed both for generic parameter values and at a distinguished parameter $\alpha_0$.

\begin{assumpC}\item 
\label{as.unfold}
\em The unfolding operator $\mathcal U_\alpha$  defines a   bounded operator from $\bar{\mathbf E}_3^0$ to  $\mathbf{L}_3^0$. Moreover,  
  \em
\begin{enumerate}
\item\em \label{as.unfolda} there exists $C>0$ such that $\|\mathcal U_\alpha\|\le C$, for all $\alpha\in J$;
 \em \item\em  \label{as.unfoldb}  the map $\alpha \mapsto \mathcal{U}_\alpha( \bar{\bm{\nu}} )$ is continuous in $\mathbf{L}_3^0 $ at $\alpha_0$, for every $\bar{\bm\nu}\in \bar{\mathbf E}_3^0$;
  \em\item\em  \label{as.unfoldc}  the map $\alpha \mapsto \mathcal{U}_\alpha(\bar{\bm{\nu}}_{\alpha_0})$ is differentiable in $\mathbf{L}_3^0 $ at $\alpha_0$.
  \end{enumerate}
\end{assumpC}

We are now in a position to state our main result for skew-product systems whose base maps admits an inducing scheme as described above.

\begin{maintheorem}\label{thm:inducedLR}
Assume that  \ref{as.induce}--\ref{as.unfold} hold and   that \ref{as.full}--\ref{as.p} hold for 
$\bar T_\alpha$.
Then, 
\begin{enumerate}
\item $\bar{\mathcal K}_\alpha$  has a unique fixed point $\bar{\bm\nu}_\alpha$ in $ \bar{\mathbf P}$;
\item $\bm\nu_\alpha=\mathcal U_\alpha\bar{\bm\nu}_\alpha$ is a fixed point of $\mathcal K_\alpha$ in $\mathbf P$; 
\item the map 
 $\alpha\mapsto\bm\nu_\alpha$
 is differentiable in $\mathbf L_3^0$.
\end{enumerate}
\end{maintheorem}

This result will be proved in Section~\ref{se.inducing}. We remark that the existence of a unique fixed point $\bar{\bm{\nu}}_\alpha$ in $\bar{\mathbf P}$ (in fact, in $\mathbf P_0$) for 
$\bar{\mathcal K}_\alpha$ is a direct consequence of Theorem~\ref{th.mainD}, since assumptions~\ref{as.full}--\ref{as.p} are assumed for $\bar T_\alpha$, and assumption~\ref{as.base} for $\bar T_\alpha$ follows directly from~\ref{as.induce}. 

\subsection{Applications}\label{se.applications}

In this section, we apply our abstract framework to two landmark problems that serve as primary motivations for this work: the \emph{Bernoulli convolutions} and the \emph{intermittent solenoid}. These applications illustrate the power of our approach in handling two distinct but fundamental challenges in linear response theory. 

 \subsubsection{Bernoulli convolutions}\label{sub.convolutions}

Bernoulli convolutions are a classical object at the interface of
probability theory, harmonic analysis, and fractal geometry.
They arise as probability measures on the real line generated by
random power series of the form
\[
Y_{\alpha } = \sum_{n=0}^{\infty} a_n \alpha^n,
\]
where $(a_n)_{n\ge0}$ are i.i.d.\ random variables with
$\mathbb P(a_n=1)=\mathbb P(a_n=-1)=\beta$, and $0<\alpha, \beta<1$.
We denote by $\nu_{\alpha,\beta}$ the distribution of $Y_{\alpha }$.

For the classical unbiased case $\beta=1/2$, a result of Jessen and Wintner~\cite{JW35} shows that
$\nu_{\alpha,1/2}$ is always \emph{pure}, that is, either absolutely continuous
or   singular with respect to Lebesgue measure.
More precisely, for $\alpha<1/2$, the measure $\nu_\alpha$ is singular
and supported on a Cantor set of zero Lebesgue measure.
The regime $\alpha>1/2$ is considerably more subtle:
Solomyak~\cite{S95} proved that $\nu_{\alpha,1/2}$ is absolutely continuous for
Lebesgue almost every \ $\alpha\in(1/2,1)$, whereas Erd\H{o}s~\cite{E39}
showed that $\nu_{\alpha,1/2}$ is singular whenever $1/\alpha$ is a
Pisot  number. The connection of this unbiased case to deterministic dynamics is provided by the
 baker map; see e.g.~\cite{AY84}.

Here we address the general biased case with $\beta\in(0,1)$ studied in \cite{PS98} by Peres and Solomyak. Consider the family of skew-product
map $T_{\alpha,\beta}\colon [-1,1]^2\to[-1,1]^2$ defined by
\[
T_{\alpha,\beta}(\omega,x)=(f_{\beta}(\omega),g_\alpha(\omega,x)),
\]
where the base map $f_\beta\colon[-1,1]\to[-1,1]$ is given by
\[
f_\beta(\omega)=
\begin{cases}
\frac1\beta\omega-\frac{\beta-1}{\beta}, & \omega\in[-1,2\beta-1),\\[2mm]
\frac1{1-\beta}\omega-\frac{\beta-1}{\beta}, & \omega\in[2\beta-1,1];
\end{cases}
\]
and the fibre maps $g_\alpha\colon[-1,1]\times[-1,1]\to[-1,1]$,
for $\alpha\in(0,1)$, are given by
\[
g_\alpha(\omega,x)=
\begin{cases}
\alpha x-(1-\alpha), & \omega\in[-1,2\beta-1),\\[2mm]
\alpha x+(1-\alpha), & \omega\in[2\beta-1,1].
\end{cases}
\]
This system admits a unique physical measure
$\mu_{\alpha,\beta}=m\times\nu_{\alpha,\beta}$, where $m$ denotes Lebesgue measure on the
base.
The vertical marginal $\nu_{\alpha,\beta}$ is the unique invariant probability measure of the  iterated function system $\{g_{\alpha,1}, g_{\alpha,2}; \beta,1-\beta\}$
satisfying the relation
\[
\nu_{\alpha,\beta}
=\tfrac1\beta (g_{\alpha,1})_*\nu_{\alpha,\beta}
+\tfrac1{1-\beta} (g_{\alpha,2})_*\nu_{\alpha,\beta},
\]
where $g_{\alpha,1}=g_{\alpha}(\omega,\cdot)$ for $\omega<2\beta-1$ and
$g_{\alpha,2}=g_{\alpha}(\omega,\cdot)$ for $\omega\ge 2\beta-1$.
Equivalently, $\nu_{\alpha,\beta}$ is the distribution of the random series
\[
X_{\alpha}=(1-\alpha)\sum_{n=0}^{\infty} a_n \alpha^n,
\]
where $(a_n)_{n\ge0}$  are i.i.d.\ random variables with
$\mathbb P(a_n=1)=\mathbb P(a_n=-1)=\beta$; this corresponds to a simple rescaling
of the original Bernoulli convolution.

Kloeckner~\cite{K22} proved that in the unbiased case the map
$\alpha\mapsto\nu_{\alpha,1/2}$ is differentiable for \emph{almost every}
$\alpha\in(1/2,1)$ -- more precisely, for a full-measure subset of the
parameters considered in~\cite{S95} -- when tested against smooth observables,
a result usually referred to as \emph{weak linear response}.
Recently we were made aware that in the upcoming paper \cite{F26} by Fu, the author studies weak linear response for unbiased Bernoulli convolutions against H\"older observables. In particular, it is shown that whenever the Hausdorff dimension of $\nu_{\alpha,1/2}$ is smaller than one, weak differentiability against H\"older observables fails to hold, while for $C^1$ observables weak linear response always hold.  By contrast, the framework developed in this paper allows us to obtain linear response in a \emph{strong sense}, valid \emph{for all parameters} $\alpha,\beta\in(0,1)$.

\begin{maintheorem}\label{th.mainconv}
The probability measure $\nu_{\alpha,\beta}$ depends differentiably on $\alpha,\beta\in(0,1)$ with values in $C^3([-1,1])^*$.
\end{maintheorem}

Theorem~\ref{th.mainconv} follows as a direct application of the final item of
Theorem~\ref{th.mainD}. For that, we are going to consider the dependence on the parameters $\alpha,\beta$ separately. Nevertheless, we use the sectional transfer operator in this setting labelled by both parameters
\[
(\mathcal K_{\alpha,\beta} \bm\nu)_\omega
=\tfrac1\beta (g_{\alpha,1})_*\nu_{\theta_{\alpha,\beta,1}(\omega)}
+\tfrac1{1-\beta} (g_{\alpha,2})_*\nu_{\theta_{\alpha,\beta,2}(\omega)}.
\]
Since $\mathcal K_{\alpha,\beta}$ preserves the closed subspace of constant
sections of $\mathbf P$, its fixed point $\bm\nu_{\alpha,\beta}$ is necessarily a
constant section.
Consequently, the space $\mathbf E_3$ in which differentiability is
obtained can naturally be identified with the dual space
$C^3([-1,1])^*$.
It therefore suffices to verify the assumptions of
Theorem~\ref{th.mainD}.

\medskip

\noindent\ref{as.base}, \ref{as.full}, \ref{as.theta}:
These assumptions are trivially satisfied: the base map is  $f_\beta$ for all $\alpha,\beta$, with two inverse branches taking values  in $[-1,2\beta-1)$ and
$[2\beta-1,1]$, and its absolutely continuous invariant probability measure is
Lebesgue measure.

\medskip

\noindent\ref{as.g}:
Since both fibre maps are affine contractions, all the required
conditions are immediately verified.
The only subtlety concerns the contraction constant $\lambda=\alpha$, which
cannot be chosen uniformly for all $\alpha\in(0,1)$.
However, as differentiability is a local property in parameter space,
$\lambda$ can always be chosen locally uniformly.

\medskip

\noindent\ref{as.p}:
This assumption is trivially verified, as the weight
functions are constant on each branch with values $\beta$ and $1-\beta$.

\subsubsection{Solenoid with intermittency}\label{sub.solenoid}

We consider a class of skew-product systems that are partially hyperbolic
but not uniformly hyperbolic, due to the presence of a neutral fixed point
in the base dynamics.
Such examples arise naturally in the study of intermittent behaviour and
provide a benchmark for statistical stability beyond the uniformly
hyperbolic setting.
 This example differs from the classical construction of the  hyperbolic solenoidal map as a skew-product in that we replace the doubling map in the base with an intermittent LSV (Liverani-Saussol-Vaienti) map~\cite{LSV99}. Linear response for LSV maps was studied in \cite{BS16, BT16, K16} and in a random and sequential setting in \cite{BRS20,DGS25}.

Let $D\subset\mathbb R^2$ denote the unit disk and fix $\lambda\in(0,1/2)$.
For each $\alpha\in(0,1)$, consider the skew-product map
\[
T_\alpha \colon [0,1]\times D \longrightarrow [0,1]\times D,
\qquad
T_\alpha(\omega,x)=(f_\alpha(\omega), g(\omega,x),
\]
where the base map $f_\alpha\colon[0,1]\to[0,1]$ is  given by
\[
f_{\alpha}(\omega)=
\begin{cases}
\omega\bigl(1+2^{\alpha}\omega^{\alpha}\bigr),
& \omega\in[0,\tfrac12],\\[0.2cm]
2\omega-1,
& \omega\in(\tfrac12,1],
\end{cases}
\]
and the fibre maps   are defined for $(\omega,x)\in [0,1]\times D$ by
\[
g(\omega,x)
=
\Bigl(\tfrac12\cos(2\pi\omega),\,\tfrac12\sin(2\pi\omega)\Bigr)
+\lambda x.
\]

The map $T_\alpha$ admits a unique physical measure $\mu_\alpha$ whose basin
covers Lebesgue-almost every point in $[0,1]\times D$; see
\cite[Theorem~4.28]{A20}.
The dynamics is partially hyperbolic, with uniform contraction along the
fibres and non-uniform expansion in the base direction, where a neutral
fixed point gives rise to intermittent behaviour.
This lack of uniform hyperbolicity is precisely the mechanism responsible
for the polynomial decay of correlations established in~\cite{AP08}.
The statistical properties of this system fit naturally into the induced
framework developed in this paper.

\begin{maintheorem}\label{th.mainsol}
For   every
$C^3$ observable $\Phi\colon[0,1]\times D\to\mathbb R$, the map
\[
\alpha \longmapsto \int_{[0,1]\times D}
\Phi \,d\mu_{\alpha} 
\]
is differentiable at every parameter $\alpha_0\in(0,1)$.
\end{maintheorem}

This follows as an application of our previous results in the following way. 
First, note that the system satisfies the assumptions of Theorem~\ref{th.mainA}, 
which guarantees the existence of a skew-product measure as in the conclusion of that theorem, 
with marginal $\eta_\alpha$ on $\Omega$. 
Since $\eta_\alpha$ is a physical measure for the base map, 
it follows from \cite[Theorem A]{K21} that this skew-product measure is a physical measure for $T_\alpha$, 
and thus coincides with $\mu_\alpha$ by uniqueness. 
The  conclusion of Thoerem~\ref{th.mainsol}   then follows from Theorem~\ref{th.mainC} applied with $k=3$, 
together with Theorem~\ref{thm:inducedLR}. 
Recall that the applicability of Theorem~\ref{th.mainC} relies on the final conclusion of Theorem~\ref{thm:inducedLR} 
as well as on assumptions~\ref{as.contraction}--\ref{as.base} and \ref{as.response}--\ref{as.continuous}, which we now proceed to verify.

\medskip
\noindent\ref{as.contraction}:
Follows trivially from the definition of the fibre maps.

\medskip
\noindent\ref{as.base}:
Is well known in the literature; see, for instance,
\cite[Section~3.5.2]{A20}.

\medskip
\noindent\ref{as.response}:
Is established in~\cite[Theorem~3.1]{BS16}.

\medskip
\noindent\ref{as.continuous}:
The first part follows from Corollary~\ref{co.KP0}, while the second part is
a consequence of Corollary~\ref{co.contfp}, both under
assumptions~\ref{as.full}--\ref{as.theta} and a weaker version~\ref{as.p'} of assumption~\ref{as.p} that we state in the beginning of Subsection~\ref{sub.invariance}, which we address
next.

\medskip
\noindent\ref{as.full}\ref{as.theta}: 
These properties are immediate, since the base maps $f_\alpha$ satisfy
$f'_\alpha \ge 1$ and admit only two inverse branches, which can easily be seen to depend
continuously on the parameter in the $C^1$ topology.

\medskip

\noindent\ref{as.p'}: 
In this case, we have two weight functions
\begin{equation*} 
p_{\alpha,i}(\omega)
=
\frac{\rho_\alpha\bigl(\theta_{\alpha,i}(\omega)\bigr)}
{|  f_\alpha'\bigl(\theta_{\alpha,i}(\omega)\bigr)|\, \rho_\alpha(\omega)},\quad i=1,2.
\end{equation*}
It suffices to show that these quantities depend continuously on $\alpha$ in the $C^0$ norm. 
It is straightforward to verify that both the inverse branches $\theta_{\alpha,i}$ and the derivative $f_\alpha'$ depend continuously on $\alpha$. 
Assume that $p_{\alpha,1}$ corresponds to the right-hand branch, which is uniformly expanding. 
It follows from~\eqref{eq.measures} that, up to a normalising constant, $\rho_\alpha \circ \theta_{\alpha,1}$ coincides with $\bar\rho_\alpha \circ \theta_{\alpha,1}$. 
This density is bounded away from zero and infinity and depends continuously (in fact, smoothly) on the parameter $\alpha$, by~\cite[Theorem~3.2]{BS16}. 
On the other hand, by the results of~\cite[Section~6.3]{Y99}, there exists a constant $C>0$ such that
\begin{equation}\label{eq:behave}
C^{-1}\omega^{-\alpha} \le \rho_\alpha(\omega) \le C\,\omega^{-\alpha}.
\end{equation}
Combining this estimate with~\cite[Theorem~3.1]{BS16}, we deduce that $1/\rho_\alpha$ depends continuously on $\alpha$ in the $C^0$ norm. 
Consequently, $p_{\alpha,1}$ depends continuously on $\alpha$ in the $C^0$ norm. 
Finally, since $p_{\alpha,0}+p_{\alpha,1}=1$ by~\eqref{eq.RMK}, the same continuity property holds for $p_{\alpha,0}$.

\medskip

The verification of the existence of an inducing scheme for which \ref{as.full}--\ref{as.p} (for the induced map) and
\ref{as.induce}--\ref{as.unfold} hold
is deferred to Subsection~\ref{sub.solenoidcheck}, thereby ensuring the
applicability of Theorem~\ref{thm:inducedLR}.

\section{Invariant measures}

In this section, we establish the conclusions of Theorem~\ref{th.mainA} under assumptions~\ref{as.contraction}--\ref{as.base}. 
We proceed in two main steps. In Subsection~\ref{se.wasserstein}, we show that the operator $\mathcal K_\alpha$ is a contraction on the complete metric space $\mathbf P$, which guarantees the existence of a unique fixed point~$\boldsymbol{\nu}_\alpha \in \mathbf P$. 
In Subsection~\ref{subskewinvariant}, we use this fixed point to construct the corresponding skew-product measure and prove its invariance under the dynamics of $T_\alpha$.

\subsection{Sample measures} \label{se.wasserstein}

We consider the space $\mathcal M_1(X)$ of Borel probability measures on \(X\), equipped with the
\emph{Wasserstein--\(1\)} distance, induced by the  norm on \(\mathrm{Lip}(X)^*\), given for \(\mu,\nu \in \mathcal M_1(X)\)  by
\[
W_1(\mu,\nu)
=
\sup_{\mathrm{Lip}(\phi)\le 1}
\left\{
\int \phi \, d\mu - \int \phi \, d\nu
\right\}.
\]
see the Appendix~\ref{se.appendix} for details.

We start by recording the standard fact that the push-forward of probability
measures by a Lipschitz map contracts distances in the Wasserstein--\(1\) metric
by at most the Lipschitz constant.

\begin{lemma} 
\label{lem:push-contract}
 If $h:X\to X$ has $\lip(h)\le L$, then for any   $\mu,\nu\in\mathcal M_1(X)$,
\[
W_1(h_*\nu,\; h_*\mu) \le L\, W_1(\nu,\mu).
\]
\end{lemma}

\begin{proof}
Let $h:X\to X$ satisfy $\mathrm{Lip}(h)\le L$. For  any  $\mu,\nu\in\mathcal M_1(X)$, we have
 \[
W_1(\nu,\mu)
= \sup_{\mathrm{Lip}(\phi)\le 1}
\left| \int_X \phi\, \nu- \int_X \phi\, d\mu \right|.
\]
Also, 
\begin{align*}
W_1(h_*\nu,\; h_*\mu)
&= \sup_{\mathrm{Lip}(\phi)\le 1}
\left| \int_X \phi\, d(h_*\nu)
       - \int_X \phi\, d(h_*\mu) \right|
\\
&= \sup_{\mathrm{Lip}(\phi)\le 1}
\left| \int_X \phi\circ h\, \nu- \int_X \phi\circ h \, d\mu \right|.
\end{align*}
Since $\mathrm{Lip}(\phi\circ h)\le \mathrm{Lip}(\phi)\,\mathrm{Lip}(h)\le L$, we have
\[
W_1(h_*\nu,\; h_*\mu)\le 
\sup_{\mathrm{Lip}(\psi)\le L}
\left| \int_X \psi\, \nu- \int_X \psi\, d\mu \right|.
\]
Rescaling $\phi = \psi/L$, we have $\mathrm{Lip}(\phi)\le 1$, which gives
\[
\sup_{\mathrm{Lip}(\psi)\le L}
\left| \int_X \psi\,d \nu- \int_X \psi\, d\mu \right|
= L \sup_{\mathrm{Lip}(\phi)\le 1}
\left| \int_X \phi\, d\nu - \int_X \phi\, d\mu \right|.
\]
Thus,
\[
W_1(h_*\nu,\; h_*\mu)
\le L\,W_1(\nu,\mu),
\]
 which completes the proof.
\end{proof}

We equip $\mathbf P$ with the supremum metric, defined for any $\bm\mu, \bm\nu \in \mathbf P$ by
\[
d_{\mathrm P}(\bm\mu, \bm\nu) = \esssup_{\omega \in \Omega} W_1(\mu_\omega, \nu_\omega).
\]
Since $(\Omega, d_\Omega)$ is a compact metric space and $(\mathcal M_1(X), W_1)$ is complete, it follows that $d_{\mathrm P}$ is a well-defined metric on $\mathbf P$ and that $(\mathbf P, d_{\mathrm P})$ is complete.  
Moreover, the measurability of the functions involved in the definition of $\mathcal K_\alpha$ ensures that $\mathcal K_\alpha$ maps $\mathbf P$ into itself.  

\begin{lemma}\label{le.contractinp} Let $0< \lambda<1$ be the constant given by~\ref{as.contraction}. For all $\bm\mu,\bm\nu\in\mathbf P$, we have
\begin{equation*}\label{eq:lambdacontract}
d_{\mathrm P}({\mathcal K_\alpha} \bm\nu,{\mathcal K_\alpha} \bm\mu)\le \lambda d_{\mathrm P}(\bm\nu,\bm\mu).
\end{equation*}
\end{lemma}

\begin{proof}
Given $\bm\mu,\bm\nu\in\mathbf P$, by  
convexity of \(W_1\) (recall   it is induced by the norm in $\lip(X)^*$),  
\begin{align}
&W_1\big(({\mathcal K_\alpha}\boldsymbol\nu)_\omega,({\mathcal K_\alpha}\bm\mu)_\omega\big)
 \le \sum_{\theta\in f_\alpha^{-1}(\{\omega\})} \frac{\rho_\alpha(\theta)}{|\det Df_\alpha(\theta)|\rho_\alpha(\omega)} W_1((g_{\alpha,\theta})_*\nu_{ \theta},(g_{\alpha,\theta})_*\mu_{ \theta}).\label{eq.convex}
\end{align}
Using assumption \ref{as.contraction}  and Lemma~\ref{lem:push-contract}, we   obtain for each $\theta$ 
 \[
\begin{aligned}
 W_1 \left((g_{\alpha,\theta})_*\nu_{ \theta} ,  (g_{\alpha,\theta})_* {\mu}_{ \theta}\right) 
\le  \lambda \, W_1(\nu_{ \theta},\mu_{ \theta}),
\end{aligned}
\]
which together with~\eqref{eq.convex} and~\eqref{eq.convex},   yields
\begin{align}
W_1\big(({\mathcal K}\boldsymbol\nu)_\omega,\,({\mathcal K}\bm\mu)_\omega\big)
&\le \lambda \sum_{\theta\in f_\alpha^{-1}(\{\omega\})} \frac{\rho_\alpha(\theta)}{|\det Df_\alpha(\theta)|\rho_\alpha(\omega)} W_1( \nu_\theta, \mu_\theta)\nonumber\\
&
\le \lambda \sum_{\theta\in f_\alpha^{-1}(\{\omega\})} \frac{\rho_\alpha(\theta)}{|\det Df_\alpha(\theta)|\rho_\alpha(\omega)} d_{\mathrm P} (\bm\nu,\bm\mu) \label{eq.seila3}
\end{align}
Now observe  that the density $\rho_\alpha$ of the invariant measure $\eta_\alpha$ satisfies the Perron--Frobenius equation
\begin{equation*}
\rho_\alpha(\omega) = \sum_{\theta \in f_\alpha^{-1}(\omega)} \frac{\rho_\alpha(\theta)}{|\det Df_\alpha(\theta)|}.
\end{equation*}
which together with~\eqref{eq.seila3} yields 
$$
W_1\big(({\mathcal K}\boldsymbol\nu)_\omega,\,({\mathcal K}\bm\mu)_\omega\big)
\le \lambda\, d_{\mathrm P}(\bm\nu,\bm\mu).
$$
Taking the essential supremum over $\omega\in\Omega$, we get
 \begin{equation*}
d_{\mathrm P}\big({\mathcal K_\alpha}\boldsymbol\nu ,\,{\mathcal K_\alpha}\bm\mu\big)\le d_{\mathrm P}(\bm\nu,\bm\mu),
\end{equation*}
thus completing  the proof.
\end{proof}

Since $(\mathbf P,d_{\mathrm P})$ is  a complete metric space, Banach fixed point theorem and the previous lemma give  the first conclusion  of Theorem~\ref{th.mainA}.
\begin{corollary}\label{co.fixedpoint}
$\mathcal K_\alpha$ has a unique fixed point  in $\mathbf P$.
\end{corollary}

\subsection{Skew-product measure}\label{subskewinvariant}

Our goal here is to deduce the second   conclusion of  Theorem~\ref{th.mainA}. Given any $\bm\nu\in\mathbf P$ and any bounded measurable function
$\Phi:\Omega\times X\to\mathbb R$, we define the pairing
\[
\langle\!\langle \bm\nu,\Phi\rangle\!\rangle
=
\int_\Omega
\langle \nu_\omega,\Phi(\omega,\cdot)\rangle
\,d\eta_\alpha(\omega).
\]

\begin{lemma} \label{le.duality.general}
For any $\bm\nu\in\mathbf P$ and any
bounded measurable function $\Phi:\Omega\times X\to\mathbb R$, one has
\[
\langle\!\langle \mathcal K_\alpha\bm\nu,\Phi\rangle\!\rangle
=
\langle\!\langle \bm\nu,\Phi\circ T_\alpha\rangle\!\rangle .
\]
\end{lemma}

\begin{proof}
By definition of the pairing and of $\mathcal K_\alpha$, and using that
$d\eta_\alpha =\rho_\alpha \,dm $, we compute
\begin{align*}
\langle\!\langle \mathcal K_\alpha\bm\nu,\Phi\rangle\!\rangle
&=
\int_\Omega
\langle (\mathcal K_\alpha\bm\nu)_\omega,\Phi(\omega,\cdot)\rangle
\,\rho_\alpha(\omega)\,dm(\omega) \\
&=
\int_\Omega
\sum_{\theta\in f_\alpha^{-1}(\omega)}
\frac{\rho_\alpha(\theta)}{|\det Df_\alpha(\theta)|}
\,
\langle (g_{\alpha,\theta})_*\nu_\theta,\Phi(\omega,\cdot)\rangle
\,dm(\omega) \\
&=
\int_\Omega
\sum_{\theta\in f_\alpha^{-1}(\omega)}
\frac{\rho_\alpha(\theta)}{|\det Df_\alpha(\theta)|}
\,
\langle \nu_\theta,\Phi(\omega,g_{\alpha,\theta}(\cdot))\rangle
\,dm(\omega).
\end{align*}
The absolute convergence of the series allows us to exchange summation and
integration expressed in the duality.
By \ref{as.base},  there exists a countable family
$\{\Omega_i\}_{i\in\mathbb I}$ of measurable subsets of $\Omega$, covering $\Omega$
up to an $m$-null set, such that each restriction
$f_\alpha|_{\Omega_i}:\Omega_i\to f_\alpha(\Omega_i)
$
is a $C^1$ diffeomorphism onto its image and
$\det Df_\alpha(\theta)\neq 0$ for all $\theta\in\Omega_i$.
We now perform a change of variables on each inverse branch 
\[
dm(\omega)=|\det Df_\alpha(\theta)|\,dm(\theta).
\]
Therefore,
\begin{align*}
\int_\Omega
\sum_{\theta\in f_\alpha^{-1}(\omega)}
\frac{\rho_\alpha(\theta)}{|\det Df_\alpha(\theta)|}
\, &
\langle \nu_\theta,\Phi(\omega,g_{\alpha,\theta}(\cdot))\rangle
\,dm(\omega)  \\
&  =
\sum_{i\in\mathbb I}
\int_{\Omega_i}
\rho_\alpha(\theta)\,
\langle \nu_\theta,
\Phi(f_\alpha(\theta),g_{\alpha,\theta}(\cdot))\rangle
\,dm(\theta) \\
&=
\int_\Omega
\langle \nu_\theta,
\Phi(f_\alpha(\theta),g_{\alpha,\theta}(\cdot))\rangle
\,d\eta_\alpha(\theta).
\end{align*}
By definition of the skew-product map $T_\alpha$, this equals
\[
\langle\!\langle \bm\nu,\Phi\circ T_\alpha\rangle\!\rangle,
\]
which concludes the proof.
\end{proof}

\begin{corollary}\label{co.invariantmeasure}
If  $\bm\nu_\alpha\in \mathbf P$ is a fixed point of the   operator $\mathcal K_\alpha$, then
\[
\mu_\alpha = \int_\Omega \nu_{\alpha,\omega}\, d\eta_\alpha(\omega)
\]
is  a $T_\alpha$-invariant probability measure.
\end{corollary}

\begin{proof}
Let $\Phi:\Omega\times X\to \mathbb R$ be an arbitrary  bounded measurable function. 
We have
\[
\int_{\Omega\times X} \Phi \circ T_\alpha \, d\mu_\alpha
=
\langle\!\langle \bm\nu_\alpha, \Phi\circ T_\alpha \rangle\!\rangle
=
\langle\!\langle \mathcal K_\alpha \bm\nu_\alpha, \Phi \rangle\!\rangle
=
\langle\!\langle \bm\nu_\alpha, \Phi \rangle\!\rangle
=
\int_{\Omega\times X} \Phi \, d\mu_\alpha,
\]
where the first and fourth equalities follow from the definition of the pairing, the second equality uses the duality given by Lemma~\ref{le.duality.general},  and the third uses that
$\bm\nu_\alpha$ is a fixed point of $\mathcal K_\alpha$. 
Hence,
\((T_\alpha)_* \mu_\alpha = \mu_\alpha\), showing that $\mu_\alpha$ is
$T_\alpha$-invariant.
\end{proof}
 
\section{Linear response}

Here we consider  a family of   skew-product maps 
$T_\alpha:\Omega\times X \to \Omega\times X$, with $\alpha\in J$. 
Our   goal in the next two subsections   is to prove Theorem~\ref{th.mainB} and Theorem~\ref{th.mainC}.

\subsection{Sample measures}\label{sub.sample}
Here we prove Theorem~\ref{th.mainB}. Assume that  $\mathcal K_\alpha$ defines a bounded linear operator on a Banach space   $\mathbf{E}_{k}^0$, for some $k\ge 0$, and  \ref{as.contraction}--\ref{as.space} hold.  
Let $\bm\nu_\alpha$  the unique fixed point of $\mathcal K_\alpha$ in~$\mathbf P$.
We need to show that there exists some $\dot{\bm\nu}_{\alpha_0}\in\mathbf{E}_{k}^0$ such that 
\begin{equation}\label{eq.limitsample}
\lim_{\alpha\to\alpha_0}\left\|\frac{\bm\nu_{\alpha}-\bm\nu_{\alpha_0}}{\alpha-\alpha_0}- \dot{\bm\nu}_{\alpha_0}\right\|_k=0.
\end{equation}
Since  $ {\mathcal K}_{\alpha}\bm\nu_{\alpha}=\bm\nu_{\alpha}$ and 
$ {\mathcal K}_{\alpha_0}\bm\nu_{\alpha_0}=\bm\nu_{\alpha_0}$, we may write
 $$
 (I- {\mathcal K}_{\alpha})(\bm\nu_{\alpha}-\bm\nu_{\alpha_0})=( {\mathcal K}_{\alpha}- {\mathcal K}_{\alpha_0})\bm\nu_{\alpha_0}.
 $$
By assumption~\ref{as.spacea} and a  standard result on Banach spaces, we have   that $I-\mathcal K_{\alpha}$ is invertible. Moreover, the  inverse is given by the von Neumann series
 \begin{equation}\label{eq.series}
(I - \mathcal K_{\alpha})^{-1} = \sum_{n=0}^\infty \mathcal K_{\alpha}^n.
\end{equation}
Hence,
 \begin{equation}\label{eq.quotient}
\frac{\bm\nu_{\alpha}-\bm\nu_{\alpha_0}}{\alpha-\alpha_0}= {(I- {\mathcal K}_{\alpha})^{-1}} 
{\frac{ {\mathcal K}_{\alpha}- {\mathcal K}_{\alpha_0}}{\alpha-\alpha_0}} \bm\nu_{\alpha_0}.
\end{equation}
Define
\begin{equation}\label{eq.tau}
\bm\tau_\alpha={\frac{ {\mathcal K}_{\alpha}- {\mathcal K}_{\alpha_0}}{\alpha-\alpha_0}} \bm\nu_{\alpha_0}.
\end{equation}
It follows from assumption~\ref{as.spaceb}
that $\bm\tau_\alpha\in\mathbf{E}_{\mathrm L}^0$ and by assumption~\ref{as.spaced} there exists some
$\bm\tau_{\alpha_0}\in\mathbf{E}_{k}^0$ such that
\begin{equation}\label{eq.limit}
\lim_{\alpha\to\alpha_0} \left\|  \bm\tau_\alpha-  {\bm\tau}_{\alpha_0}\right\|_k=0.
\end{equation}
Combining~\eqref{eq.quotient} and~\eqref{eq.tau}, it remains to prove that
$(I-\mathcal K_{\alpha})^{-1}\bm\tau_\alpha$ converges in $\mathbf{E}_{k}^0$, when 
$\alpha\to\alpha_0$. In fact, we will  show that   it converges to $(I- {\mathcal K}_{\alpha_0})^{-1}\bm\tau_{\alpha_0}$. Writing
 \begin{align}
\|(I- {\mathcal K}_{\alpha})^{-1}\bm\tau_\alpha- (I- {\mathcal K}_{\alpha_0})^{-1}\bm\tau_{\alpha_0}\|_k
\,\le\; &
\|(I- {\mathcal K}_{\alpha})^{-1}\bm\tau_\alpha- (I- {\mathcal K}_{\alpha})^{-1}\bm\tau_{\alpha_0}\|_k\nonumber \\
&+
\|(I- {\mathcal K}_{\alpha})^{-1}\bm\tau_{\alpha_0}- (I- {\mathcal K}_{\alpha_0})^{-1}\bm\tau_{\alpha_0}\|_k, \nonumber
\end{align}
we are left to prove that
\begin{enumerate}
\item[(i)] $\|(I- {\mathcal K}_{\alpha})^{-1}\bm\tau_\alpha- (I- {\mathcal K}_{\alpha})^{-1}\bm\tau_{\alpha_0}\|_k\to0$, as $\alpha\to\alpha_0$;
\item[(ii)]  $\|(I- {\mathcal K}_{\alpha})^{-1}\bm\tau_{\alpha_0}- (I- {\mathcal K}_{\alpha_0})^{-1}\bm\tau_{\alpha_0}\|_k\to0$, as $\alpha\to\alpha_0$.
\end{enumerate}
We first prove (i). We have
\begin{equation}\label{eq.bound1}
\|(I- {\mathcal K}_{\alpha})^{-1}\bm\tau_\alpha- (I- {\mathcal K}_{\alpha})^{-1}\bm\tau_{\alpha_0}\|_k\le \|(I- {\mathcal K}_{\alpha})^{-1}\|\cdot\| \bm\tau_\alpha-  {\bm\tau}_{\alpha_0} \|_k.
\end{equation}
It follows from  \ref{as.spacea} and~\eqref{eq.series} that
\begin{equation}\label{unifbound}
\|(I - \mathcal K_{\alpha})^{-1} \| \le \sum_{n=0}^\infty \|\mathcal K_{\alpha}\|^n \le   \sum_{n=0}^\infty\lambda^n=\frac1{1-\lambda}.
\end{equation}
Using~\eqref{eq.limit}, \eqref{eq.bound1} and \eqref{unifbound}, we obtain (i).

\medskip
\noindent
To prove (ii), we  use   the \emph{second resolvent identity}
\begin{align*}
 (I- {\mathcal K}_{\alpha})^{-1} - (I- {\mathcal K}_{\alpha_0})^{-1} 
= (I - \mathcal{K}_\alpha)^{-1} (\mathcal{K}_\alpha - \mathcal{K}_{\alpha_0}) (I - \mathcal{K}_{\alpha_0})^{-1},
\end{align*}
which combined  with~\eqref{unifbound} yields
\begin{equation} \label{eq:boundid}
\|(I- \mathcal{K}_{\alpha})^{-1}\bm\tau_{\alpha_0} - (I- \mathcal{K}_{\alpha_0})^{-1}\bm\tau_{\alpha_0}\|_k\le \frac1{1-\lambda} \|(\mathcal{K}_\alpha - \mathcal{K}_{\alpha_0}) (I- \mathcal{K}_{\alpha_0})^{-1}\bm\tau_{\alpha_0}\|_k.
\end{equation}
Set 
$$\bm\tau_0=(I- \mathcal{K}_{\alpha_0})^{-1}\bm\tau_{\alpha_0}.
$$ 
By \eqref{eq.limit}, the vector $\bm\tau_{\alpha_0}$ is   the limit in $\mathbf{E}_{k}^0$ of $\bm\tau_\alpha \in \mathbf{E}_{\mathrm L}^0$, as $\alpha\to\alpha_0$. Thus, $\bm\tau_{\alpha_0}$ belongs to the closure $\overline{ \mathbf{E}}_{\mathrm L}^0$ of $ \mathbf{E}_{\mathrm L}^0$ in the  topology induced by the norm in $\mathbf{E}_{k}^0$. 
Moreover, assumption~\ref{as.spaceb} gives that 
$\mathcal{K}_{\alpha_0}$ maps $\mathbf{E}_{\mathrm L}^0$ into itself.  
Since we assume  $\mathcal{K}_{\alpha_0}$ a bounded operator on $\mathbf{E}_{k}^0$, then~$\mathcal{K}_{\alpha_0}$   maps  $\overline{ \mathbf{E}}_{\mathrm L}^0$ into itself. Consequently, the series in~\eqref{eq.series} defining $(I - \mathcal{K}_{\alpha_0})^{-1} $ also maps $\overline{ \mathbf{E}}_{\mathrm L}^0$ into itself. This ensures that~$\bm\tau_0 \in \overline{ \mathbf{E}}_{\mathrm L}^0$.

Now, we fix an arbitrary  $\epsilon > 0$. 
By  definition of the closure, there exists an approximating element $\bm\tau_\epsilon \in \mathbf{E}_{\mathrm L}^0$ such that
\begin{equation}\label{eq.approc}
\|\boldsymbol{\tau}_0 - \boldsymbol{\tau}_\epsilon\|_k< \frac{(1-\lambda)\epsilon}{3 \lambda}.
\end{equation}
We now bound the right-hand side of \eqref{eq:boundid} using the triangle inequality:
\begin{align}\label{eq.haja}
\|(\mathcal{K}_\alpha - \mathcal{K}_{\alpha_0}) \boldsymbol{\tau}_0\|_k
\le 
\| \mathcal{K}_\alpha(\boldsymbol{\tau}_0 - \boldsymbol{\tau}_\epsilon) \|_k+ \| (\mathcal{K}_\alpha - \mathcal{K}_{\alpha_0}) \boldsymbol{\tau}_\epsilon \|_k+ \| \mathcal{K}_{\alpha_0}(\boldsymbol{\tau}_\epsilon - \boldsymbol{\tau}_0)\|_k.
\end{align}
By assumption \ref{as.spacea} and the choice of $\bm\tau_\epsilon$ in \eqref{eq.approc}, the first and third terms are bounded as
\begin{equation}\label{eq.both1}
\| \mathcal{K}_\alpha(\boldsymbol{\tau}_0 - \boldsymbol{\tau}_\epsilon) \|_k\le
\lambda \|\boldsymbol{\tau}_0 - \boldsymbol{\tau}_\epsilon\|_k< \frac{(1-\lambda)\epsilon}{3},
\end{equation}
and
\begin{equation}\label{eq.both2}
\| \mathcal{K}_{\alpha_0}(\boldsymbol{\tau}_\epsilon - \boldsymbol{\tau}_0) \|_k\le
\lambda \|\boldsymbol{\tau}_\epsilon - \boldsymbol{\tau}_0\|_k< \frac{(1-\lambda)\epsilon}{3}.
\end{equation}
Since $\boldsymbol{\tau}_\epsilon \in \mathbf{E}_{\mathrm L}^0$, we invoke assumption~\ref{as.spacec}, which provides a $\delta > 0$ such that if $|\alpha - \alpha_0| < \delta$, then 
\begin{equation}\label{eq.lasts}
\|(\mathcal{K}_\alpha - \mathcal{K}_{\alpha_0}) \boldsymbol{\tau}_\epsilon\|_k< \frac{(1-\lambda)\epsilon}{3}.
\end{equation}
Substituting \eqref{eq.both1}, \eqref{eq.both2}, and \eqref{eq.lasts} into \eqref{eq.haja}, we obtain
\begin{equation*}
\|(\mathcal{K}_\alpha - \mathcal{K}_{\alpha_0}) \boldsymbol{\tau}_0\|_k< (1-\lambda)\epsilon.
\end{equation*}
Finally, substituting this back into  \eqref{eq:boundid}, we have for $|\alpha - \alpha_0| < \delta$
\begin{equation*}
\|(I- \mathcal{K}_{\alpha})^{-1}\bm\tau_{\alpha_0} - (I- \mathcal{K}_{\alpha_0})^{-1}\bm\tau_{\alpha_0}\|_k\le \frac{1}{1-\lambda} \cdot (1-\lambda)\epsilon = \epsilon.
\end{equation*}
Since $\epsilon > 0$ was arbitrary, we obtain (ii), thereby completing the proof of Theorem~\ref{th.mainB}.

 \subsection{Skew-product measure}\label{sub.skewresponse}
 Here we prove Theorem~\ref{th.mainC}.
Let 
$\boldsymbol{\nu}_\alpha $ be the unique fixed point of~$\mathcal K_\alpha$ in $\mathbf P_0$ and
\[
\mu_\alpha = \int_\Omega \nu_{\alpha,\omega}\, d\eta_\alpha(\omega)
\]
be the measure given by Theorem~\ref{th.mainA}.
Assume that \ref{as.response}--\ref{as.continuous} hold and the map 
$\alpha\mapsto \bm\nu_\alpha$
is differentiable in the space $\mathbf L_k^0$ at $\alpha_0$, for some $k\ge1$.
 Given a $C^k$  observable $\Phi:\Omega\times  X\to\mathbb R$, define for each~$\alpha\in J$
\[
\psi_\alpha(\omega)=\int_X \Phi(\omega,x)\, d\nu_{\alpha,\omega}(x),
\]
so that $$\int _{\Omega\times X}\Phi\, d\mu_\alpha = \int_\Omega \psi_\alpha\, d\eta_\alpha.$$

\begin{lemma}\label{le.varphicont}
Each $\psi_\alpha:\Omega\to\mathbb R$  is a continuous function; moreover, $\|\psi_\alpha-\psi_{\alpha_0}\|_{C^0}\to 0$, as $\alpha\to\alpha_0$.
\end{lemma}
\begin{proof}

First we show that $\psi_\alpha$ is a continuous function. 
For all $\omega,\omega'\in\Omega$, we have
\begin{align*}
|\psi_\alpha(\omega)-\psi_\alpha(\omega')|&=\left|\int_X\Phi(\omega,\cdot)\, d\nu_{\alpha,\omega}-\int_X \Phi(\omega',\cdot)\, d\nu_{\alpha,\omega}\right|\\
&\le\int_X |\Phi(\omega,\cdot)- \Phi(\omega',\cdot)|\, d\nu_{\alpha,\omega}+\left|\int_X \Phi(\omega' ,\cdot)\, d\nu_{\alpha,\omega}-\int_X \Phi(\omega',\cdot)\, d\nu_{\alpha,\omega'} \right| .
\end{align*}
Let $L=\|\Phi\|_{C^k}$. Since $L$ works as a Lipschitz constant for the map $\omega\mapsto\Phi(\omega,\cdot)$ and each $\nu_{\alpha,\omega}$ is a probability measure, we may write
\begin{align*}
|\psi_\alpha(\omega)-\psi_\alpha(\omega')|
&\le \int_X L\, d_\Omega(\omega,\omega') \, d\nu_{\alpha,\omega}
+ L\left|\int_X \frac{\Phi(\omega',\cdot)}{L}\, d\nu_{\alpha,\omega}
      - \int_X \frac{\Phi(\omega',\cdot)}{L}\, d\nu_{\alpha,\omega'}\right| \\
&= L d_\Omega(\omega,\omega') + L\Big\langle\nu_{\alpha,\omega}-\nu_{\alpha,\omega'},     \frac{\Phi(\omega',\cdot)}{L}\Big\rangle \\
&\le L d_\Omega(\omega,\omega') + L\|\nu_{\alpha,\omega}-\nu_{\alpha,\omega'}\|_{k^*}. 
\end{align*}
Since the first part of~\ref{as.continuous} implies that 
$\bm\nu_{\alpha}$ is a continuous map taking values in $C^k(X)^*$, 
it follows   that $\psi_\alpha$ is continuous.

\smallskip
\noindent
Now we prove that $\psi_\alpha$ converges to $\psi_{\alpha_0}$  in the supremum norm. Given any   $\omega\in\Omega$,  
\begin{align*}
|\psi_\alpha(\omega)-\psi_{\alpha_0}(\omega)|&=\left|\int_X \Phi(\omega,\cdot)\, d\nu_{\alpha,\omega}-\int_X \Phi(\omega,\cdot)\, d\nu_{\alpha_0,\omega}\right| \\
&=L\Big| \Big\langle\nu_{\alpha,\omega}-\nu_{\alpha_0,\omega},     \frac{\Phi(\omega,\cdot)}{L}\Big\rangle\Big|\\
&\le L\|\nu_{\alpha,\omega}-\nu_{\alpha_0,\omega}\|_{k^*}.
\end{align*}
Taking supremum in $\omega\in\Omega$ on both sides and using the second part of assumption~\ref{as.continuous}, the second conclusion follows.
\end{proof}

This lemma will be used to complete the proof of Theorem~\ref{th.mainA}.
By Assumption~\ref{as.response}, there exists ${\dot\rho_{\alpha_0}\in L^1(m)}$ such that
\begin{equation}\label{eq.convB}
\lim_{\alpha\to\alpha_0}
\int_\Omega 
\Big|\frac{\rho_\alpha-\rho_{\alpha_0}}{\alpha-\alpha_0}-\dot\rho_{\alpha_0}\Big|
\, dm
=0.
\end{equation}
Define the signed measure $\dot\eta_{\alpha_0} = \dot\rho_{\alpha_0}\, m$.
We complete the proof of Theorem~\ref{th.mainC} by showing that
\begin{equation}\label{eq:LR-split}
\frac{d}{d\alpha}\Big|_{\alpha_0}\int_{\Omega\times X} \Phi\,d\mu_\alpha
\;=\; 
{\int_\Omega \langle \dot\nu_{\alpha_0,\omega}, \Phi(\omega,\cdot)\rangle   \, d\eta_{{\alpha_0}}(\omega)}
\;+\;
{ \int_\Omega\psi_{\alpha_0}\, d\dot\eta_{\alpha_0}}.
\end{equation}
%
We may write 
\[
\int_{\Omega\times X}\Phi\,d\mu_{\alpha} - \int_{\Omega\times X} \Phi\,d\mu_{\alpha_0}
= \int_\Omega \psi_{\alpha}\, d\eta_{{\alpha}} - \int_\Omega\psi_{\alpha_0}\, d\eta_{\alpha_0}
= A_\alpha + B_\alpha,
\]
with
\[
A_\alpha = \int_\Omega\big(\psi_{\alpha}-\psi_{\alpha_0}\big)\, d\eta_{ {\alpha_0}}
\qand
B_\alpha = \int_\Omega \psi_{\alpha}\, d\big(\eta_{{\alpha}}-\eta_{\alpha_0}\big).
\]
We first show that
\begin{equation}\label{eq.limA}
\lim_{\alpha\to \alpha_0}\frac{A_\alpha}{\alpha-\alpha_0} 
= {\int_\Omega \langle \dot\nu_{\alpha_0,\omega}, \Phi(\omega,\cdot)\rangle   \, d\eta_{{\alpha_0}}(\omega)}.
\end{equation}
We have
\[
\frac{A_\alpha}{{\alpha-\alpha_0}}=\int_\Omega\frac{\psi_{\alpha}(\omega)-\psi_{\alpha_0}(\omega)}{{\alpha-\alpha_0}}\, d\eta_{ {\alpha_0}}(\omega)
= \int_\Omega\int_X \Phi(\omega,\cdot) \,d \left(\frac{\nu_{\alpha,\omega}- \nu_{\alpha_0,\omega}}{{\alpha-\alpha_0}}\right)   d\eta_{{\alpha_0}}(\omega).
\]
Letting $C=\|\Phi\|_{C^k}$, we may write for all $\omega\in\Omega$,
\begin{align*}
\left|\frac{A_\alpha}{{\alpha-\alpha_0}} 
- \int_\Omega \langle \dot\nu_{\alpha_0,\omega}, \Phi(\omega,\cdot)\rangle   d\eta_{{\alpha_0}}(\omega)\right|
&=C\left|
\int_\Omega\Big\langle \frac{\nu_{\alpha,\omega}- \nu_{\alpha_0,\omega}}{{\alpha-\alpha_0}}-\dot\nu_{\alpha_0,\omega}\,, \frac{\Phi(\omega,\cdot)}C \Big\rangle    d\eta_{{\alpha_0}}(\omega)\right|\\
&\le  C \int_\Omega\Big\| \frac{\nu_{\alpha,\omega}- \nu_{\alpha_0,\omega}}{{\alpha-\alpha_0}}-\dot \nu_{\alpha_0,\omega}\Big\|_{k^*} \, d\eta_{{\alpha_0}}(\omega)
\\
&\le C\Big\| \frac{\bm\nu_{\alpha }- \bm\nu_{\alpha_0 }}{{\alpha-\alpha_0}}- \bm{\dot\nu}_{\alpha_0}\Big\|_{\mathbf L_k}. 
\end{align*}
Since we assume that the map 
$
\alpha \mapsto \bm{\nu}_\alpha
$
is differentiable in $\mathbf L_k^0$ at $\alpha_0$, the latter expression converges to zero as $\alpha \to \alpha_0$, and hence~\eqref{eq.limA} holds.
It remains to show that
\begin{equation}\label{eq.limB}
\lim_{\alpha \to \alpha_0}
\frac{B_{\alpha}}{\alpha-\alpha_0}
=
\int_\Omega \psi_{\alpha_0}\, d\dot\eta_{\alpha_0}.
\end{equation}
We have
\[
 \frac{B_{\alpha}}{{\alpha-\alpha_0}}
= \frac{1}{{\alpha-\alpha_0}}\int_\Omega \psi_{\alpha}\, d\big(\eta_{{\alpha}}-\eta_{\alpha_0}\big)= 
 \frac{1}{{\alpha-\alpha_0}}\int_\Omega \psi_{\alpha} \big(\rho_{{\alpha}}-\rho_{\alpha_0}\big)dm,
\]
and so
\begin{align*}
 \frac{B_{\alpha}}{{\alpha-\alpha_0}}- { \int_\Omega \psi_{\alpha_0}\, d\dot\eta_{\alpha_0}}
&=  
 \int_\Omega \psi_{\alpha} \frac{\rho_{{\alpha}}-\rho_{\alpha_0}}{\alpha-\alpha_0}dm- { \int_\Omega \psi_{\alpha_0} \dot\rho_{\alpha_0}} dm\\
 &=  
 \int_\Omega \psi_{\alpha} \frac{\rho_{{\alpha}}-\rho_{\alpha_0}}{\alpha-\alpha_0}dm- { \int_\Omega \psi_{\alpha } \dot\rho_{\alpha_0}} dm+ { \int_\Omega \psi_{\alpha } \dot\rho_{\alpha_0}} dm- { \int_\Omega \psi_{\alpha_0} \dot\rho_{\alpha_0}} dm\\
 &=  
 \int_\Omega \psi_{\alpha}\left( \frac{\rho_{{\alpha}}-\rho_{\alpha_0}}{\alpha-\alpha_0}-\dot\rho_{\alpha_0}\right)dm +  \int_\Omega\left(\psi_{\alpha }-\psi_{\alpha_0} \right) d\dot\eta_{\alpha_0}\\
 &\le  
\| \psi_{\alpha}\|_{C^0}\int_\Omega \Big|\frac{\rho_\alpha-\rho_{\alpha_0}}{\alpha-\alpha_0}-\dot\rho_{\alpha_0}\Big| dm +  \|\psi_{\alpha }-\psi_{\alpha_0} \|_{C^0} |\dot\eta_{\alpha_0}|(\Omega),
\end{align*}
where $|\dot\eta_{\alpha_0}|$ denotes the  variation of the signed measure $\dot\eta_{\alpha_0}$. Using~\eqref{eq.convB} and Lemma~\ref{le.varphicont}, we obtain~\eqref{eq.limB}.


\section{Base maps with full  branches}\label{se.fulbranch}

In this section we prove Theorem~\ref{th.mainD} and Corollary~\ref{th.piecewise2}.  
As observed above, to prove Theorem~\ref{th.mainD} it suffices to establish~\ref{as.space}, and to deduce Corollary~\ref{th.piecewise2} it suffices to establish~\ref{as.continuous}, both for the case $k=3$.  
These conditions  are verified  in the next four subsections. More precisely:

\ref{as.spacea} $\longrightarrow$ Lemma~\ref{le.contre0};

\ref{as.spaceb}  $\longrightarrow$ Corollary~\ref{co.KE0};

\ref{as.spacec}  $\longrightarrow$ Corollary~\ref{co.contK};

\ref{as.spaced}  $\longrightarrow$ Corollary~\ref{co.diffK};

\ref{as.continuous} $\longrightarrow$  Corollary~\ref{co.KP0} and Corollary~\ref{co.contfp}.

 \subsection{Contraction}\label{sub.contraction}
Here we establish the contraction property~\ref{as.spacea} for the operator $\mathcal K_\alpha$ acting on the space $\mathbf{E}_3^0$. Incidentally, this is the only place where using the space $\mathbf{E}_3^0$ yields a better conclusion than using $\mathbf{E}_3$. We begin with a standard  lemma, whose second part will only be used in the next subsection.

\begin{lemma} 
\label{lem:push-ck}
Assume that $h: X \to X$ is a $C^k$ function, for some $k\ge 1$, and define $L_j = \|D^j h\|_{C^0}$ for $j=1, \dots, k$. Let 
\[
\mathcal{C}_k(h) = \max_{1 \le s \le k} \left\{ \sum_{m=s}^k {B}_{m,s}(L_1, \dots, L_{m-s+1}) \right\}
\]
where ${B}_{m,s}$ are the partial Bell polynomials. Then
\begin{enumerate}
    \item\label{it.lem_part1} for any $\nu \in C^k_0(X)^*$, we have $\|h_*\nu\|_{k^*} \le \mathcal{C}_k(h) \|\nu\|_{k^*}$;
    \item\label{it.lem_part2} for any $\nu \in C^k(X)^*$, we have $\|h_*\nu\|_{k^*} \le \max\{1, \diam(X) + \mathcal{C}_k(h)\} \|\nu\|_{k^*}$.
\end{enumerate}
\end{lemma}

\begin{proof}
Recall that for $\nu \in C^k_0(X)^*$, the dual norm satisfies $$\|\nu\|_{k^*} = \sup_{|\varphi|_{C^k} \le 1} \langle \nu, \varphi \rangle.$$ By the definition of the push-forward, $\langle h_*\nu, \varphi \rangle = \langle \nu, \varphi \circ h \rangle$. To prove the result, it is sufficient to show that  the composition $\Phi = \varphi \circ h$ satisfies the seminorm bound ${|\Phi|_{C^k} \le \mathcal{C}_k(h) |\varphi|_{C^k}}$.
The $m$-th derivative of the composition $\Phi$ is given by Faà di Bruno's formula:
\[
D^m \Phi(x) = \sum_{s=1}^m (D^s \varphi \circ h)(x) \cdot  {B}_{m,s}\big(Dh(x), D^2h(x), \dots, D^{m-s+1}h(x)\big),
\]
where ${B}_{m,s}$ are the partial Bell polynomials. These polynomials are sums of terms of the form
\[
\frac{m!}{j_1! j_2! \cdots j_{m-s+1}!} \prod_{r=1}^{m-s+1} \left( \frac{D^r h}{r!} \right)^{j_r}
\]
subject to $\sum_{r=1}^{m-s+1} j_r = s$ and $\sum_{r=1}^{m-s+1} r j_r = m$. Taking the supremum norm of $D^m \Phi$, we obtain the bound
\begin{equation}\label{eq.delta_m}
 \|D^m \Phi\|_{C^0}   \le \sum_{s=1}^m \|D^s \varphi\|_{C^0}  \cdot {B}_{m,s}(L_1, L_2, \dots, L_{m-s+1}).
\end{equation}
Summing \eqref{eq.delta_m} over $m=1, \dots, k$ to reconstruct the $C^k$ seminorm, we find
\begin{align*}
|\Phi|_{C^k} = \sum_{m=1}^k \|D^m \Phi\|_{C^0} &\le \sum_{m=1}^k \sum_{s=1}^m \|D^s \varphi\|_{C^0}  \cdot {B}_{m,s}(L_1, \dots, L_{m-s+1}) \\
&= \sum_{s=1}^k \|D^s \varphi\|_{C^0}  \left( \sum_{m=s}^k {B}_{m,s}(L_1, \dots, L_{m-s+1}) \right).
\end{align*}
Therefore, taking
\[
\mathcal{C}_k(h) = \max_{1 \le s \le k} \left\{ \sum_{m=s}^k {B}_{m,s}(L_1, \dots, L_{m-s+1}) \right\},
\]
we have
$$|\Phi|_{C^k} \le \mathcal{C}_k(h) \sum_{s=1}^k \|D^s \varphi\|_{C^0}  = \mathcal{C}_k(h) |\varphi|_{C^k}.$$ It follows that
\[
\|h_*\nu\|_{k^*} = \sup_{|\varphi|_{C^k} \le 1} \langle \nu, \varphi \circ h \rangle \le \sup_{|\varphi|_{C^k} \le 1} \|\nu\|_{k^*} |\varphi \circ h|_{C^k} \le \mathcal{C}_k(h) \|\nu\|_{k^*},
\]
thus concluding the proof of the first item.

\smallskip
\noindent
For the second item, we need the full norm $\|\Phi\|_{C^k} = |\Phi(x_0)| + |\Phi|_{C^k}$. 
By the mean value theorem
\[
|\varphi(h(x_0))| \le |\varphi(x_0)| + |h(x_0) - x_0| \cdot \|D\varphi\|_{C^0} \le |\varphi(x_0)| + \text{diam}(X) |\varphi|_{C^k}.
\]
Combining this with the seminorm bound, we get
\begin{align*}
\|\Phi\|_{C^k} &\le |\varphi(x_0)| + \text{diam}(X) |\varphi|_{C^k} + \mathcal{C}_k(h) |\varphi|_{C^k} \\
&= |\varphi(x_0)| + (\text{diam}(X) + \mathcal{C}_k(h)) |\varphi|_{C^k} \\
&\le \max\{1, \text{diam}(X) + \mathcal{C}_k(h)\} (|\varphi(x_0)| + |\varphi|_{C^k}).
\end{align*}
Consequently, for any $\nu \in C^k(X)^*$:
\[
\|h_*\nu\|_{k^*} = \sup_{\|\varphi\|_{C^k} \le 1} \langle \nu, \varphi \circ h \rangle \le \max\{1, \text{diam}(X) + \mathcal{C}_k(h)\} \|\nu\|_{k^*},
\]
which concludes the proof.
\end{proof}

\begin{remark}\label{re.bell}
The constant $\mathcal{C}_k(h)$ can be explicitly computed for low orders by evaluating the partial Bell polynomials and summing the coefficients corresponding to each derivative order $\|D^s \varphi\|_{C^0} $. Note that the arguments of ${B}_{m,s}$ are $(L_1, \dots, L_{m-s+1})$.
\begin{itemize}
    \item  {Case $k=1$:} Only $m=1, s=1$ exists and 
    ${B}_{1,1}(L_1) = L_1$. 
    
   Hence, $\mathcal{C}_1(h) = L_1$.
    
    \item  {Case $k=2$:} We collect coefficients for $s=1$ and $s=2$:
    \begin{itemize}
        \item For $s=1$: ${B}_{1,1}(L_1) + {B}_{2,1}(L_1, L_2) = L_1 + L_2$.
        \item For $s=2$: ${B}_{2,2}(L_1) = L_1^2$.
    \end{itemize}
    Hence, $\mathcal{C}_2(h) = \max\{ L_1 + L_2, \, L_1^2 \}$.

    \item  {Case $k=3$:} We evaluate the sums for $s=1, 2, 3$:
    \begin{itemize}
        \item For $s=1$: ${B}_{1,1}(L_1) + {B}_{2,1}(L_1, L_2) + {B}_{3,1}(L_1, L_2, L_3) = L_1 + L_2 + L_3$.
        \item For $s=2$: ${B}_{2,2}(L_1) + {B}_{3,2}(L_1, L_2) = L_1^2 + 3L_1L_2$.
        \item For $s=3$: ${B}_{3,3}(L_1) = L_1^3$.
    \end{itemize}
    Hence, $\mathcal{C}_3(h) = \max\{ L_1 + L_2 + L_3, \, L_1^2 + 3L_1L_2, \, L_1^3 \}$.
\end{itemize}
The value of $\mathcal C_3(h)$ motivates our choice of the constant $\lambda$ in~\ref{as.gb}. Indeed, since we assume in particular that $L_1 + L_2 + L_3 < 1$, it follows that $L_1 < 1$, and therefore the third term in the expression for $\mathcal C_3(h)$ can be omitted.
\end{remark}

 For the next result,  $0 < \lambda < 1$ is the constant specified in~\ref{as.gb}.

\begin{lemma}\label{le.contre0}
For every  $\bm\nu \in \mathbf{E}_{3}^0$ we have
$$\|\mathcal K_\alpha \bm\nu\|_{3}\le \lambda \| \bm\nu\|_{3}.$$
\end{lemma}

\begin{proof}
Let $\bm{\nu} \in \mathbf{E}_3^0$. By definition, for every $\omega \in \Omega$, the functional $\nu_\omega$ belongs to $C^3_0(X)^*$, which implies that $\langle \nu_\omega, 1 \rangle = 0$. The norm on the Banach space $\mathbf{E}_3^0$ is given by
\[
\|\bm{\nu}\|_3 = \sup_{\omega \in \Omega} \|\nu_\omega\|_{3^*}.
\]
For a fixed $\omega \in \Omega$, the transfer operator $\mathcal{K}_\alpha$ acts on the section $\bm{\nu}$ as
\[
(\mathcal{K}_\alpha \bm{\nu})_\omega = \sum_{i \in \mathbb{I}} p_{\alpha,i}(\omega) (g_{\alpha, \theta_{\alpha,i}(\omega)})_* \nu_{\theta_{\alpha,i}(\omega)}.
\]
Applying the dual norm $\|\cdot\|_{3^*}$ to both sides and using the triangle inequality, we obtain:
\begin{equation}\label{eq.triang3}
\|(\mathcal{K}_\alpha \bm{\nu})_\omega\|_{3^*} \le \sum_{i \in \mathbb{I}} p_{\alpha,i}(\omega) \|(g_{\alpha, \theta_{\alpha,i}(\omega)})_* \nu_{\theta_{\alpha,i}(\omega)}\|_{3^*}.
\end{equation}
For each $i \in \mathbb{I}$, let $h_{i,\omega} = g_{\alpha, \theta_{\alpha,i}(\omega)}$. Since $\nu_{\theta_{\alpha,i}(\omega)} \in C^3_0(X)^*$, we apply the  push-forward estimate in   Lemma~\ref{lem:push-ck} for $k=3$. This yields
\[
\|(h_{i,\omega})_* \nu_{\theta_{\alpha,i}(\omega)}\|_{3^*} \le \mathcal{C}_3(h_{i,\omega}) \|\nu_{\theta_{\alpha,i}(\omega)}\|_{3^*},
\]
where the constant $\mathcal{C}_3(h_{i,\omega})$ is determined by the derivative norms $L_j = \|D^j g_{\alpha, \theta_{\alpha,i}(\omega)}\|_{C^0}$ as 
\[
\mathcal{C}_3(h_{i,\omega}) = \max\{ L_1 + L_2 + L_3, \, L_1^2 + 3L_1L_2, \, L_1^3 \};
\]
recall Remark~\ref{re.bell}. By in~\ref{as.gb}, $\lambda$ is a uniform upper bound for $\mathcal{C}_3(h_{i,\omega})$ over all $i \in \mathbb{I}$ and $\omega \in \Omega$. Specifically, we assume $\sup_{\alpha, i, \omega} \mathcal{C}_3(g_{\alpha, \theta_{\alpha,i}(\omega)}) \le \lambda < 1$. Substituting this bound into \eqref{eq.triang3}, we get
\begin{align*}
\|(\mathcal{K}_\alpha \bm{\nu})_\omega\|_{3^*} &\le \sum_{i \in \mathbb{I}} p_{\alpha,i}(\omega) \lambda \|\nu_{\theta_{\alpha,i}(\omega)}\|_{3^*} \\
&\le \lambda \left( \sum_{i \in \mathbb{I}} p_{\alpha,i}(\omega) \right) \sup_{\omega' \in \Omega} \|\nu_{\omega'}\|_{3^*} \\
&= \lambda \|\bm{\nu}\|_3,
\end{align*}
where we used the normalization condition $\sum_{i \in \mathbb{I}} p_{\alpha,i}(\omega) = 1$ for all $\omega \in \Omega$. Taking the supremum over $\omega \in \Omega$ on the left-hand side, we conclude the proof.
\end{proof}

\subsection{Invariance}\label{sub.invariance}
With the application to the intermittent solenoid in Subsection~\ref{sub.solenoid} in mind, we emphasise that the arguments used to establish~\ref{as.spaceb} and~\ref{as.continuous} rely solely on assumptions~\ref{as.full}--\ref{as.theta}, together with the following weakened form of assumption~\ref{as.p}.

\begin{assumpBprime}
\item \label{as.p'} \em For each $i\in\mathbb I$,  the map $  \alpha \mapsto   p_{\alpha,i}\in C^0(  \Omega,\mathbb R)$ is   continuous,   and there exists $C'_p>0$ such that 
$$ 
\sum_{i\in\mathbb I} \sup_{\alpha\in J}\|  p_{\alpha,i} \|_{C^0}\le C'_p.
$$
 \end{assumpBprime}
 
 Define for each section $\bm\nu$ and $\omega\in\Omega$
 \begin{equation}\label{eq.termsk}
\mathcal (\mathcal K_{\alpha,i}\bm\nu)_\omega=p_{\alpha,i}(\omega) (g_{\alpha, \theta_{\alpha,i}(\omega)})_* \nu_{\theta_{\alpha,i}(\omega)},
\end{equation}
so that by~\eqref{eq.kalpha2} we have
\begin{equation}\label{eq.kalpha3}
\mathcal K_\alpha 
=
\sum_{i \in \mathbb I}\mathcal K_{\alpha,i}.
\end{equation}

 \begin{lemma}\label{le.K_F0}
Assume that ~\ref{as.full}--\ref{as.theta} and \ref{as.p'} hold. Then,  the operator $\mathcal K_\alpha$ maps $\mathbf{E}_{\mathrm L}$ into itself.
\end{lemma}

\begin{proof}
Considering the sum in~\eqref{eq.kalpha3}, we first show that each $\mathcal K_{\alpha,i}$ maps $\mathbf{E}_{\mathrm L}$ into itself. 
Take an arbitrary $\bm\nu \in \mathbf{E}_{\mathrm L}$. 
We first show that $(\mathcal K_{\alpha,i}\bm\nu)_\omega \in\mathrm{Lip}(X)^*$, for all $\omega\in\Omega$. 
By assumption~\ref{as.g}, each $g_{\alpha,\theta_{\alpha,i}(\omega)}$ is a $C^3$ map of $X$, so in particular it is a Lipschitz map. By definition, we have for each $\phi\in\lip(X)$
\[
\bigl\langle (g_{\alpha,\theta_{\alpha,i}(\omega)})_* \nu_{\theta_{\alpha,i}(\omega)}, \phi \bigr\rangle
= \langle  \nu_{\theta_{\alpha,i}(\omega)}, \phi \circ g_{\alpha,\theta_{\alpha,i}(\omega)} \rangle,
\]
which clearly shows that $(g_{\alpha,\theta_{\alpha,i}(\omega)})_* \nu_{\theta_{\alpha,i}(\omega)}$ is an element of $\mathrm{Lip}(X)^*$. 
Since multiplication by the scalar $p_{\alpha,i}(\omega) $ does not affect this fact, we have  $(\mathcal K_{\alpha,i}\bm\nu)_\omega\in\mathrm{Lip}(X)^*$ for all $\omega\in\Omega$.  

Now we show the continuity of the map $\omega \mapsto (\mathcal K_\alpha \bm\nu)_\omega$ in $\mathrm{Lip}(X)^*$.
Given  $\omega, \omega' \in \Omega$, we have
\[
\| (\mathcal K_\alpha \bm\nu)_\omega - (\mathcal K_\alpha \bm\nu)_{\omega'} \|_{\mathrm{Lip}^*}
\le \sum_{i \in \mathbb I} 
\bigl\| p_{\alpha,i}(\omega) \, (g_{\alpha,\theta_{\alpha,i}(\omega)})_* \nu_{\theta_{\alpha,i}(\omega)} 
- p_{\alpha,i}(\omega') \, (g_{\alpha,\theta_{\alpha,i}(\omega')})_* \nu_{\theta_{\alpha,i}(\omega')} \bigr\|_{\mathrm{Lip}^*}.
\]
We can split each term as
\begin{align*}
&\bigl\| p_{\alpha,i}(\omega) (g_{\alpha,\theta_{\alpha,i}(\omega)})_* \nu_{\theta_{\alpha,i}(\omega)}
- p_{\alpha,i}(\omega') (g_{\alpha,\theta_{\alpha,i}(\omega')})_* \nu_{\theta_{\alpha,i}(\omega')} \bigr\|_{\mathrm{Lip}^*} \\
&\quad \le |p_{\alpha,i}(\omega) - p_{\alpha,i}(\omega')| \cdot \| (g_{\alpha,\theta_{\alpha,i}(\omega)})_* \nu_{\theta_{\alpha,i}(\omega)} \|_{\mathrm{Lip}^*} \\
&\quad \quad + |p_{\alpha,i}(\omega')| \cdot \| (g_{\alpha,\theta_{\alpha,i}(\omega)})_* \nu_{\theta_{\alpha,i}(\omega)} - (g_{\alpha,\theta_{\alpha,i}(\omega')})_* \nu_{\theta_{\alpha,i}(\omega')} \|_{\mathrm{Lip}^*}.
\end{align*}
For the first term, we have by~\ref{as.gb} a uniform constant $L>0$ such that for any $\phi\in\lip(X)$ with $\|\phi\|_{\mathrm L}\le1$ we have
\begin{align}\label{eq.lbound}
\| \phi\circ g_{\alpha,\theta_{\alpha,i}(\omega)} \|_{\mathrm L}
&
=|\phi\circ g_{\alpha,\theta_{\alpha,i}(\omega)}(x_0)|+\mathrm{Lip}(\phi \circ g_{\alpha,\theta_{\alpha,i}(\omega)})\nonumber\\
& \le 1+ \sup_{x \in X} \| D g_{\alpha,\theta_{\alpha,i}(\omega)}(x) \| \nonumber\\
&\le 1 + L.
\end{align}
Hence
\begin{align}
\bigl\| (g_{\alpha,\theta_{\alpha,i}(\omega)})_* \nu_{\theta_{\alpha,i}(\omega)} \bigr\|_{\mathrm{Lip}^*} 
&= \sup_{  \|\phi\|_{\mathrm L} \le 1} 
\bigl| \langle \nu_{\theta_{\alpha,i}(\omega)}, \phi \circ g_{\alpha,\theta_{\alpha,i}(\omega)} \rangle \bigr|
 \nonumber \\
 &\le \sup_{  \|\phi\|_{\mathrm L} \le 1} \| \nu_{\theta_{\alpha,i}(\omega)} \|_{\mathrm{Lip}^*}  \| \phi\circ g_{\alpha,\theta_{\alpha,i}(\omega)} \|_{\mathrm L} \nonumber\\
&\le  \| \bm\nu \|_{\mathrm L}(1+L). \label{eq.unifomry}
\end{align}
Hence
\begin{equation*}
|p_{\alpha,i}(\omega) - p_{\alpha,i}(\omega')| \cdot \| (g_{\alpha,\theta_{\alpha,i}(\omega)})_* \nu_{\theta_{\alpha,i}(\omega)} \|_{\mathrm{Lip}^*}\le L\|\bm\nu\|_{\mathrm L} |p_{\alpha,i}(\omega) - p_{\alpha,i}(\omega')| .
\end{equation*}
On the other hand, using~\eqref{eq.lbound} 
\begin{align*}
\bigl\| (g_{\alpha,\theta_{\alpha,i}(\omega)})_* &\nu_{\theta_{\alpha,i}(\omega)} - (g_{\alpha,\theta_{\alpha,i}(\omega')})_* \nu_{\theta_{\alpha,i}(\omega')} \bigr\|_{\mathrm{Lip}^*} \\
&\quad= \sup_{  \|\phi\|_{\mathrm L} \le 1}  
\Bigl| \langle \nu_{\theta_{\alpha,i}(\omega)}, \phi \circ g_{\alpha,\theta_{\alpha,i}(\omega)} \rangle
- \langle \nu_{\theta_{\alpha,i}(\omega')}, \phi \circ g_{\alpha,\theta_{\alpha,i}(\omega')} \rangle \Bigr| \\
&\quad\le\sup_{  \|\phi\|_{\mathrm L} \le 1} 
\Bigl| \langle \nu_{\theta_{\alpha,i}(\omega)} - \nu_{\theta_{\alpha,i}(\omega')}, \phi \circ g_{\alpha,\theta_{\alpha,i}(\omega)} \rangle \Bigr|\\
&\quad\quad\quad\quad+ \sup_{  \|\phi\|_{\mathrm L} \le 1}  
\Bigl| \langle \nu_{\theta_{\alpha,i}(\omega')}, \phi \circ g_{\alpha,\theta_{\alpha,i}(\omega)} - \phi \circ g_{\alpha,\theta_{\alpha,i}(\omega')} \rangle \Bigr| \\
&\quad\le \|\nu_{\theta_{\alpha,i}(\omega)} - \nu_{\theta_{\alpha,i}(\omega')}\|_{\mathrm{Lip}^*} 
\, 
  \sup_{  \|\phi\|_{\mathrm L} \le 1} 
  \| \phi\circ g_{\alpha,\theta_{\alpha,i}(\omega)} \|_{\mathrm L} \\
&\quad\quad\quad\quad + \|\nu_{\theta_{\alpha,i}(\omega')}\|_{\mathrm{Lip}^*} \, 
 { \sup_{  \|\phi\|_{\mathrm L} \le 1}  \|\phi \circ g_{\alpha,\theta_{\alpha,i}(\omega)} - \phi \circ g_{\alpha,\theta_{\alpha,i}(\omega')} \|_{C^0}} \\
&\quad\le (1+L) \|\nu_{\theta_{\alpha,i}(\omega)} - \nu_{\theta_{\alpha,i}(\omega')}\|_{\mathrm{Lip}^*} 
+ \|\bm\nu \|_{\mathrm L} \, \| g_{\alpha,\theta_{\alpha,i}(\omega)} - g_{\alpha,\theta_{\alpha,i}(\omega')} \|_{C^0}.
\end{align*}
Therefore
\begin{align*}
\|  (\mathcal K_\alpha \bm\nu)_\omega  - (\mathcal K_\alpha \bm\nu)_{\omega'} \|_{\mathrm{Lip}^*}
 \le & (1+L) \|\bm\nu \|_{\mathrm L} |p_{\alpha,i}(\omega) - p_{\alpha,i}(\omega')|\\ &+
(1+L) \|\nu_{\theta_{\alpha,i}(\omega)} - \nu_{\theta_{\alpha,i}(\omega')}\|_{\mathrm{Lip}^*} \\
&+ \|\bm\nu \|_{\mathrm L} \, \| g_{\alpha,\theta_{\alpha,i}(\omega)} - g_{\alpha,\theta_{\alpha,i}(\omega')} \|_{C^0}.
\end{align*}
Using assumptions~\ref{as.g},~\ref{as.theta}, and~\ref{as.p'}, together with the uniform continuity of $\bm\nu$, 
we deduce that the three differences above can be made arbitrarily small whenever $\omega$ and $\omega'$ are sufficiently close. 
Hence, the map $\omega \mapsto (\mathcal K_\alpha \bm\nu)_\omega$ is continuous.

Finally, using assumption~\ref{as.p'} together with the uniform bound~\eqref{eq.unifomry}, we deduce that the series~\eqref{eq.kalpha3} defining $\mathcal K_\alpha \bm\nu$ converges uniformly in $\mathbf{E}_{\mathrm L}$. 
By standard results on series in Banach spaces, it then follows that $\mathcal K_\alpha \bm\nu \in \mathbf{E}_{\mathrm L}$.
\end{proof}

In the next result we establish~\ref{as.spaceb}.

  \begin{corollary}\label{co.KE0}
Assume that ~\ref{as.full}--\ref{as.theta} and \ref{as.p'} hold. Then,  $\mathcal K_\alpha$ maps $\mathbf{E}_{\mathrm L}^0$ into itself.
\end{corollary}

\begin{proof}
This is a straightforward consequence of Lemma~\ref{le.K_F0}, together with the observation that push-forwards preserve functionals in $\mathrm{Lip}_0(X)^*$, 
that is, those vanishing on constant functions.
\end{proof}

In the net result we obtain  the first part
of~\ref{as.continuous}. 

  \begin{corollary}\label{co.KP0}
Assume that ~\ref{as.full}--\ref{as.theta} and \ref{as.p'} hold. Then,  the fixed  point of $\mathcal K_\alpha$ in~$\mathbf P$ belongs in  $\mathbf P_0$.
\end{corollary}

\begin{proof} 
Since $\mathbf{P}_0$ is a closed subset of $\mathbf{P}$, it suffices to show that it is invariant under $\mathcal K_\alpha$. 
Let $\bm\nu \in \mathbf{P}_0$. 
By definition, $\mathbf{P}_0 \subset \mathbf{E}_{\mathrm L}$, so the continuity of $\mathcal K_\alpha \bm\nu$ follows from Lemma~\ref{le.K_F0}. 
It remains to check that each $(\mathcal K_\alpha \bm\nu)_\omega$ is a probability measure. 
This follows immediately from the properties of push-forwards and the weights: since $(g_{\alpha,\theta_{\alpha,i}(\omega)})_*$ preserves total mass, and $p_{\alpha,i}(\omega) \ge 0$ with $\sum_{i\in \mathbb I} p_{\alpha,i}(\omega) = 1$, we conclude that $(\mathcal K_\alpha \bm\nu)_\omega \in \mathcal M_1(X)$. 
\end{proof}

 \subsection{Continuity}\label{sub.continuity}
In this subsection, we establish the continuity property~\ref{as.spacec} and the second part of~\ref{as.continuous}, both on the space~$\mathbf{E}_3^0$. These results will follow as consequences of Proposition~\ref{pr.contK} below. We emphasise that, also in this part, our  conclusions rely solely on assumptions~\ref{as.full}--\ref{as.theta}, together with the weaker version~\ref{as.p'} of assumption~\ref{as.p} introduced in Subsection~\ref{sub.invariance}.

\begin{lemma}\label{le.salvation}
 Given   $\nu \in \mathrm{Lip}(X)^*$ and   $C^1$ maps $h_0, h_1: X \to X$,   there exists   $\mathcal C(h_1,h_2) > 0$ depending only on the $C^1$ norms of $h_0$ and $h_1$ such that, for any $\phi\in C^2(X)$, 
 \[
\big|\langle (h_1)_*\nu-(h_0)_*\nu,\phi\rangle\big| \le \mathcal C(h_1,h_2) \, \|\nu\|_{\mathrm{Lip}^*} \, \|\phi\|_{C^2} \, \|h_1 - h_0\|_{C^1}.
\]
\end{lemma}

\begin{proof}
By definition of the push-forward,
\[
\langle (h_1)_*\nu - (h_0)_*\nu, \phi \rangle = \langle \nu, \phi \circ h_1 - \phi \circ h_0 \rangle.
\]
Since $X$ is convex, define $h_t = h_0 + t(h_1 - h_0)$ for $t \in [0,1]$. Then
\[
\phi \circ h_1 - \phi \circ h_0 = \int_0^1 \frac{d}{dt} \phi \circ h_t \, dt = \int_0^1 (\nabla \phi \circ h_t) \bm\cdot (h_1 - h_0) \, dt . 
\]
Set
$$\psi_t = (\nabla \phi \circ h_t) \bm\cdot (h_1 - h_0).$$
Since $\phi \in C^2(X)$, each $\psi_t$ is Lipschitz and
 \begin{align}
\mathrm{Lip}(\psi_t) 
&\le \|\nabla \phi \circ h_t\|_{C^0} \mathrm{Lip}(h_1 - h_0) + \mathrm{Lip}(\nabla \phi \circ h_t) \|h_1 - h_0\|_{C^0}\nonumber\\
&\le \| \phi  \|_{C^1} \|h_1 - h_0\|_{C^1} + \|\phi\|_{C^2} \|h_t\|_{C^1} \|h_1 - h_0\|_{C^0}\nonumber\\
& \le \|\phi\|_{C^2} \|h_1 - h_0\|_{C^1} (1 + \|h_t\|_{C^1}).\label{eq.lipsi}
\end{align}
This  gives in particular that $t \mapsto \psi_t \in \mathrm{Lip}(X)$ is Bochner integrable in the Banach space~$\lip(X)$.
By the properties of this integral,
\[
\left| \langle \nu, \int_0^1 \psi_t \, dt \rangle \right| = \left| \int_0^1 \langle \nu, \psi_t \rangle \, dt \right| \le \int_0^1 |\langle \nu, \psi_t \rangle| \, dt.
\]
On the other hand, by the definition of the dual norm,
\[
|\langle \nu, \psi_t \rangle| \le \|\nu\|_{\mathrm{Lip}^*} \|\psi_t\|_{\lip} \quad \text{for all } t \in [0,1].
\]
Hence
\[
|\langle (h_1)_*\nu - (h_0)_*\nu, \phi \rangle| \le \int_0^1 |\langle \nu, \psi_t \rangle| \, dt \le \|\nu\|_{\mathrm{Lip}^*} \int_0^1 \|\psi_t\|_{\lip} \, dt.
\]
Finally, using the uniform bound for $\|\psi_t\|_{\lip}$ obtained in~\eqref{eq.lipsi} and integrating over $t$ we obtain  the desired conclusion.
\end{proof}

\begin{example}\label{ex.impossible} 
With this example we show that Lemma~\ref{le.salvation} cannot be established in the space $C^1(X)$.
Indeed, let $X=[0,1]$ and define for $\epsilon>0$ the dipoles $$\nu_\epsilon=\frac1\epsilon(\delta_{1} - \delta_{1-\epsilon}),$$ which satisfy $\|\nu_\epsilon\|_{\mathrm{Lip}^*}=1$.
Consider the family of contractions $h_t(z)=(1-t)z$ for $t \in (0,1)$. Note that $\|h_t-h_s\|_{C^1}=|t-s|$.
For any $\phi \in C^1(X)$, we have
\[
\langle (h_t)_*\nu_\epsilon - (h_s)_*\nu_\epsilon, \phi \rangle = \frac{\phi(1-t) - \phi((1-t)(1-\epsilon))}{\epsilon} - \frac{\phi(1-s) - \phi((1-s)(1-\epsilon))}{\epsilon}.
\]
As $\epsilon \to 0$, this converges to
\[
\Delta(t,s) := (1-t)\phi'(1-t) - (1-s)\phi'(1-s).
\]
Now, choose $\phi \in C^1(X)$ such that $\phi'(y) = |1-y|^\beta$, with $\beta \in (0,1)$. Then
\[
\Delta(t,s) = (1-t)t^\beta - (1-s)s^\beta.
\]
Evaluating the ratio for $s = 2t$ with small $t$, we obtain
\[
\frac{|\Delta(t, 2t)|}{\|h_t - h_{2t}\|_{C^0}} = \frac{|(1-t)t^\beta - (1-2t)(2t)^\beta|}{t} = t^{\beta-1} |(1-t) - 2^\beta(1-2t)|.
\]
As $t \to 0$, the term in the absolute value converges to $|1 - 2^\beta| \neq 0$. Thus, the ratio behaves as $c t^{\beta-1}$, which diverges because $\beta < 1$. This confirms that $C^1$ regularity is insufficient to provide a uniform Lipschitz constant for the push-forward operator.
\end{example}

The next result  is formulated for elements   in the space $\mathbf{E}_{\mathrm L}$, which encompasses both spaces    $\mathbf{E}_{\mathrm L}^0$ and $\mathbf P_0$ on which we have  to verify properties~\ref{as.spacec} and~\ref{as.continuous}, respectively. 

\begin{proposition}\label{pr.contK}
 Assume that~\ref{as.full}--\ref{as.theta} and \ref{as.p'} hold and let $\boldsymbol{\nu}\in\mathbf{E}_{\mathrm L}$. Then,  
\begin{enumerate}
\item the map
\(\alpha\mapsto \mathcal K_{\alpha,i} \bm\nu\)
is continuous in $ \mathbf{E}_2 $, for each $i\in\mathbb I$;
\item the series
\(
\sum_{i\in\mathbb I}  {\mathcal K}_{\alpha,i} \bm\nu
\) 
converges  in $ \mathbf{E}_2$ uniformly in $\alpha$.
\end{enumerate}
\end{proposition}

\begin{proof}
We first show that, for fixed  $i \in \mathbb{I}$ and $\bm{\nu} \in \mathbf{E}_{\mathrm L}$,  the map $\alpha \mapsto \mathcal{K}_{\alpha,i} \bm{\nu}$ is continuous in $\mathbf{E}_2$.
The action of the operator $\mathcal{K}_{\alpha,i}$ on the section $\boldsymbol{\nu}$ is given pointwise by
\[ 
[(\mathcal{K}_{\alpha,i} \boldsymbol{\nu})]_\omega = p_{\alpha,i}(\omega) (g_{\alpha, \theta_{\alpha,i}(\omega)})_* \nu_{\theta_{\alpha,i}(\omega)}. 
\]
Hence,
\begin{align}
\|[(\mathcal{K}_{\alpha,i} \boldsymbol{\nu})]_\omega - [(\mathcal{K}_{\alpha_0,i} \boldsymbol{\nu})]_\omega\|_{2^*}
&= \| p_{\alpha,i}(\omega) (g_{\alpha, \theta_{\alpha,i}(\omega)})_* \nu_{\theta_{\alpha,i}(\omega)} - p_{\alpha_0,i}(\omega) (g_{\alpha_0, \theta_{\alpha_0,i}(\omega)})_* \nu_{\theta_{\alpha_0,i}(\omega)} \|_{2^*} \nonumber \\
&\le A(\omega) + B(\omega) + C(\omega), \label{eq.triple3}
\end{align}
with
\begin{align*}
A(\omega) &= |p_{\alpha,i}(\omega) - p_{\alpha_0,i}(\omega)| \cdot \| (g_{\alpha, \theta_{\alpha,i}(\omega)})_* \nu_{\theta_{\alpha,i}(\omega)} \|_{2^*} \\
B(\omega) &= |p_{\alpha_0,i}(\omega)| \cdot \| (g_{\alpha, \theta_{\alpha,i}(\omega)})_* \nu_{\theta_{\alpha,i}(\omega)} - (g_{\alpha_0, \theta_{\alpha_0,i}(\omega)})_* \nu_{\theta_{\alpha,i}(\omega)} \|_{2^*} \\
C(\omega) &= |p_{\alpha_0,i}(\omega)| \cdot \| (g_{\alpha_0, \theta_{\alpha_0,i}(\omega)})_* (\nu_{\theta_{\alpha,i}(\omega)} - \nu_{\theta_{\alpha_0,i}(\omega)}) \|_{2^*}.
\end{align*}
We estimate these terms separately. 
Using the second item of Lemma~\ref{lem:push-ck} and Remark~\ref{re.bell} and since $\bm\nu\in\mathbf{E}_{\mathrm L}\subset \mathbf{E}_2$, we have a constant $L>0$ depending only on the constants $L_1,L_2$ in \ref{as.gb} such that
$$
\big\|(g_{\alpha,\theta_{\alpha,i}(\omega)})_* \nu_{\theta_{\alpha,i}(\omega)}\big\|_{2^*}
\le L \, \|\nu_{\theta_{\alpha,i}(\omega)}\|_{2^*}\le L\,  \, \|\boldsymbol{\nu}\|_2.
$$
Thus, 
\begin{equation}\label{eq.Aomega3}
A(\omega) \le L\, \|\boldsymbol{\nu}\|_2 \cdot \| p_{\alpha,i} - p_{\alpha_0,i} \|_{C^0}.
\end{equation}
For the second term, we apply Lemma~\ref{le.salvation}
to the maps $h_1 = g_{\alpha,\theta_{\alpha,i}(\omega)}$, $h_2 = g_{\alpha_0,\theta_{\alpha_0,i}(\omega)}$
and  
$\nu = \nu_{\theta_{\alpha,i}(\omega)}\in\lip(X)^*$, which together with~\ref{as.g} yields a uniform constant $C>0$ such that for any $\phi\in C^2(X)$
\[
\big|\langle (g_{\alpha, \theta_{\alpha,i}(\omega)})_* \nu_{\theta_{\alpha,i}(\omega)} - (g_{\alpha_0, \theta_{\alpha_0,i}(\omega)})_* \nu_{\theta_{\alpha,i}(\omega)} ,\phi\rangle\big| \le C\|\nu_{\theta_{\alpha,i}(\omega)}\|_{\lip^*}\,\|\phi\|_{C^2}\,\| g_{\alpha,\theta_{\alpha,i}(\omega)} - g_{\alpha_0,\theta_{\alpha_0,i}(\omega)} \|_{C^0} .
\]
Taking the supremum over $\|\phi\|_{C^2} \le 1$ and since $\bm\nu\in \mathbf{E}_{\mathrm L}$, we get 
\begin{equation*} 
 \| (g_{\alpha, \theta_{\alpha,i}(\omega)})_* \nu_{\theta_{\alpha,i}(\omega)} - (g_{\alpha_0, \theta_{\alpha_0,i}(\omega)})_* \nu_{\theta_{\alpha,i}(\omega)} \|_{2^*} \le   \|\bm\nu\|_{\mathrm L} \| g_{\alpha,\theta_{\alpha,i}(\omega)} - g_{\alpha_0,\theta_{\alpha_0,i}(\omega)} \|_{C^0}.
\end{equation*}
Hence,
\begin{equation}\label{eq.Bomega3}
B(\omega) \le \| p_{\alpha_0,i} \|_{C^0} \|\bm\nu\|_{\mathrm L} \| g_{\alpha,\theta_{\alpha,i}(\omega)} - g_{\alpha_0,\theta_{\alpha_0,i}(\omega)} \|_{C^0}.
\end{equation}
Using again the second part of Lemma~\ref{lem:push-ck} and Remark~\ref{re.bell}  together with \ref{as.g} for the third term, we obtain  
\begin{equation} \label{eq.Comega3}
C(\omega) \le L \| p_{\alpha_0,i} \|_{C^0} \| \nu_{\theta_{\alpha,i}(\omega)} - \nu_{\theta_{\alpha_0,i}(\omega)} \|_{2^*}.
\end{equation}
Putting \eqref{eq.triple3} and \eqref{eq.Aomega3}--\eqref{eq.Comega3} together, we get
\begin{align}
\|[(\mathcal{K}_{\alpha,i} \boldsymbol{\nu})]_\omega - [(\mathcal{K}_{\alpha_0,i} \boldsymbol{\nu})]_\omega\|_{2^*}
&\le L \|\boldsymbol{\nu}\|_3 \cdot \| p_{\alpha,i} - p_{\alpha_0,i} \|_{C^0} \nonumber\\
&\quad + C \| p_{\alpha_0,i} \|_{C^0} \|\bm{\nu}\|_2 \| g_{\alpha,\theta_{\alpha,i}(\omega)} - g_{\alpha_0,\theta_{\alpha_0,i}(\omega)} \|_{C^1}  \nonumber\\
&\quad + L \| p_{\alpha_0,i} \|_{C^0} \| \nu_{\theta_{\alpha,i}(\omega)} - \nu_{\theta_{\alpha_0,i}(\omega)} \|_{2^*}.\label{eq.finalest}
\end{align}
Taking the supremum over $\omega \in \Omega$, and using continuity assumptions \ref{as.full}--\ref{as.p} and the uniform continuity of the section $\bm{\nu} \in \mathbf{E}_{\mathrm L} \subset \mathbf{E}_2$, we conclude the proof of the first part. 
The second part follows from the estimate obtained in \eqref{eq.finalest} and the summability conditions in \ref{as.full}--\ref{as.p}.
\end{proof}

\begin{remark}\label{re.independent1}
Even when the weights and the fibre maps do not depend on the base point, the regularity of the space $\mathbf E_2$ in which Proposition~\ref{pr.contK} is proved cannot be improved. Indeed, $C^2$ regularity is essential in the proof of Lemma~\ref{le.salvation}, as Example~\ref{ex.impossible}  illustrates. Moreover, even in the base-independent case, the continuity of the term $B(\omega)$ in the proof of Proposition~\ref{pr.contK} requires comparing push-forwards associated with two nearby but distinct maps, due to the dependence on the parameter.
\end{remark}

The next result  provides   \ref{as.spacec} in the space $\mathbf{E}_3^0$. 

\begin{corollary}\label{co.contK}
 Assume that~\ref{as.full}--\ref{as.theta} and \ref{as.p'} hold and let $\boldsymbol{\nu}\in\mathbf{E}_{\mathrm L}$. 
 If  $\boldsymbol{\nu}\in\mathbf{E}_{\mathrm L}^0$, then 
  the map
\(\alpha\mapsto \mathcal K_{\alpha} \bm\nu\)
is continuous in $ \mathbf{E}_3^0 $.
\end{corollary}

\begin{proof}
First, observe that by basic calculus in Banach spaces, Proposition~\ref{pr.contK} implies that for any $\bm\nu \in \mathbf{E}_{\mathrm L}^0 \subset \mathbf{E}_{\mathrm L}$, the map $\alpha \mapsto \mathcal K_\alpha \bm\nu$ is continuous as a map into $\mathbf{E}_2$. Since we have continuous inclusions $\mathbf{E}_3^0 \subset \mathbf{E}_3 \subset \mathbf{E}_2$, convergence in $\mathbf{E}_2$ therefore implies convergence in~$\mathbf{E}_3^0$. 
 \end{proof}

The next result is a consequence of Lemma~\ref{le.contre0} and Proposition~\ref{pr.contK}, and provides the second part of \ref{as.continuous}.

\begin{corollary}\label{co.contfp}
 Assume that~\ref{as.full}--\ref{as.theta} and \ref{as.p'} hold and let $\boldsymbol{\nu}\in\mathbf{E}_{\mathrm L}$. 
If $\bm\nu_\alpha,\bm\nu_{\alpha_0}$ are respectively the unique fixed points of $\mathcal{K}_\alpha,\mathcal{K}_{\alpha_0}$ in~$\mathbf{P}_0$, then 
$$\lim_{\alpha \to \alpha_0} \|\bm\nu_\alpha - \bm\nu_{\alpha_0}\|_3 = 0.$$
\end{corollary}
\begin{proof}
Recall that the difference $\bm\nu_\alpha - \bm\nu_{\alpha_0}$ belongs to the  space $\mathbf{E}_{\mathrm L}^0\subset\mathbf{E}_3^0$. Using the fixed point property,
\[
\bm\nu_\alpha - \bm\nu_{\alpha_0} = \mathcal{K}_\alpha \bm\nu_\alpha - \mathcal{K}_{\alpha_0} \bm\nu_{\alpha_0}.
\]
We decompose the right-hand side as 
\begin{align*}
\bm\nu_\alpha - \bm\nu_{\alpha_0} &= (\mathcal{K}_\alpha \bm\nu_\alpha - \mathcal{K}_\alpha \bm\nu_{\alpha_0}) + (\mathcal{K}_\alpha \bm\nu_{\alpha_0} - \mathcal{K}_{\alpha_0} \bm\nu_{\alpha_0}) \\
&= \mathcal{K}_\alpha (\bm\nu_\alpha - \bm\nu_{\alpha_0}) + (\mathcal{K}_\alpha \bm\nu_{\alpha_0} - \mathcal{K}_{\alpha_0} \bm\nu_{\alpha_0}) .
\end{align*}
Using the contraction given by Lemma~\ref{le.contre0},  we get
\[
\|\bm\nu_\alpha - \bm\nu_{\alpha_0}\|_3 \le \lambda \|\bm\nu_\alpha - \bm\nu_{\alpha_0}\|_3 + \|(\mathcal{K}_\alpha - \mathcal{K}_{\alpha_0}) \bm\nu_{\alpha_0}\|_3.
\]
Rearranging the terms, we obtain
\[
(1-\lambda) \|\bm\nu_\alpha - \bm\nu_{\alpha_0}\|_3 \le \|(\mathcal{K}_\alpha - \mathcal{K}_{\alpha_0}) \bm\nu_{\alpha_0}\|_3.
\]
Recalling that $\mathbf P_0\subset \mathbf{E}_{\mathrm L}$, taking $\bm\nu=\bm\nu_{\alpha_0}$ in Corollary~\ref{co.contK}, we obtain
\[
\lim_{\alpha \to \alpha_0} \|\mathcal{K}_\alpha\bm\nu_{\alpha_0} - \mathcal{K}_{\alpha_0}\bm\nu_{\alpha_0} \|_3 = 0.
\]
Since $1-\lambda > 0$, it follows that $\|\bm\nu_\alpha - \bm\nu_{\alpha_0}\|_3 \to 0$ as $\alpha \to \alpha_0$.
\end{proof}

\subsection{Differentiability}\label{sub.regularity}
Here we establish the validity of~\ref{as.spaced} for any fixed $\alpha_0 \in J$,
with the argument divided into two parts. We first establish the $C^1$ regularity
of the invariant section $\bm\nu_{\alpha}$, viewed as a map
from the base space $\Omega$ into the  space $C^2(X)^*$. We then use this
spatial regularity to deduce the parameter differentiability required by~\ref{as.spaced}.

Regarding the $C^1$ regularity of the invariant section
$\omega \mapsto \nu_{\alpha,\omega}$, we remark that one could derive a Lasota--Yorke inequality for the operator
$\mathcal K_\alpha$, taking $\mathbf{E}_2'$ as the strong space and $\mathbf{E}_2$ as
the weak one. However, such an estimate would be of no use in the present
context for establishing the $C^1$ differentiability of the fixed point
$\bm\nu_\alpha$. Indeed, due to the infinite dimensionality of  $C^2(X)^*$, the space $\mathbf{E}_2'$ is not compactly embedded
in $\mathbf{E}_2$. Instead, we adopt a graph transform–type approach.

\subsubsection{Differentiability of the invariant section}
Our objective is to establish that the invariant section $\bm\nu_\alpha$, defined as the unique fixed point of the sectional transfer operator
\[
(\mathcal{K}_\alpha\boldsymbol{\nu})_\omega = \sum_{i\in\mathbb I} p_{\alpha,i}(\omega) \, (g_{\alpha, \theta_{\alpha,i}(\omega)})_*  {\nu}_{\theta_{\alpha,i}(\omega)} 
\]
is of class $C^1$ as a mapping from the base space $\Omega$ into  $C^2(X)^*$. By formally differentiating the fixed-point identity $\mathcal{K}_\alpha\boldsymbol{\nu}_\alpha  = \boldsymbol{\nu}_{\alpha}$ with respect to $\omega$, we observe that the derivative $\boldsymbol{\nu}_\alpha'$ must satisfy a linear functional equation. This equation naturally leads to the definition of a \emph{tangent operator}, which describes the first-order variation of the invariant section by accounting for the simultaneous infinitesimal changes in the weights, the fiber dynamics, and the base map shifts.

We define the \emph{tangent operator} $\mathcal{T}: \mathbf{E}_{2}^0 \to \mathbf{E}_{2}^0$  as the affine mapping given for any $\bm\xi\in\mathbf E_2^0$ and $\omega\in\Omega$ by
\[
(\mathcal{T} \boldsymbol{\xi})_\omega
=  {\sum_{i \in \mathbb I} p_{\alpha,i}(\omega) \, \theta_{\alpha,i}'(\omega) \, (g_{\alpha, \theta_{\alpha,i}(\omega)})_* \xi_{\theta_{\alpha,i}(\omega)}} 
+  {(\mathcal{R} \boldsymbol{\nu}_\alpha)_\omega} .
\]
The term $(\mathcal{R}\boldsymbol{\nu}_\alpha)$ is the image of the invariant section $\boldsymbol{\nu}_\alpha \in \mathbf{P}$ under the \emph{remainder operator}~$\mathcal{R}$. 
It will be convenient to have   this operator defined more generally for an arbitrary section $\boldsymbol{\nu} \in \mathbf{P}$, by prescribing, for each $\omega \in \Omega$ and each test function $\phi \in C^2(X)$,
\[
\langle (\mathcal{R} \boldsymbol{\nu})_\omega, \phi \rangle
= \sum_{i \in \mathbb I} p'_{\alpha,i}(\omega) \, \langle \nu_{\theta_{\alpha,i}(\omega)}, \phi \circ g_{\alpha, \theta_{\alpha,i}(\omega)} \rangle
+ \sum_{i \in \mathbb I} p_{\alpha,i}(\omega) \, \langle \nu_{\theta_{\alpha,i}(\omega)}, \partial_\omega (\phi \circ g_{\alpha, \theta_{\alpha,i}(\omega)}) \rangle.
\]
Note that since $\sum p_{\alpha,i}(\omega) = 1$ and $\langle \nu_\omega, 1\rangle = 1$ for all $\omega$, it follows that   $\mathcal{R} \boldsymbol{\nu} \in \mathbf{E}_{2}^0$.
The linear part of the operator $\mathcal{T}$ describes how derivatives propagate through the base dynamics, whereas the remainder term $\mathcal{R}\boldsymbol{\nu}_\alpha$ accounts for the variations arising from the $\omega$–dependence of the weights and of the fiber maps.  
The key observation is that, if the invariant section $\boldsymbol{\nu}_\alpha$ is differentiable with respect to $\omega$, then its derivative $\boldsymbol{\nu}_\alpha'$ necessarily satisfies the fixed-point identity
\[
\mathcal{T}\boldsymbol{\nu}_\alpha' = \boldsymbol{\nu}_\alpha'.
\]
Thus, the existence and uniqueness of a fixed point of $\mathcal{T}$ in $\mathbf{E}_{2}^0$ provides the unique candidate for the derivative of the invariant section. The $C^1$ regularity of $\boldsymbol{\nu}_\alpha$ then follows from the uniform convergence of $\partial_\omega(\mathcal{K}_\alpha^n\boldsymbol{\nu})$ to this fixed point, starting from any smooth $\boldsymbol{\nu}\in\mathbf{P}$, together with the fundamental theorem of calculus in Banach spaces.

\begin{lemma} 
\label{le.remainder}
For any $\boldsymbol{\nu}\in \mathbf{P}$, we have
\[
\lim_{n\to\infty} \|\mathcal{R}(\mathcal{K}_\alpha^n \boldsymbol{\nu} - \boldsymbol{\nu}_\alpha)\|_{2} 
= 0.
\]
\end{lemma}

\begin{proof}
Fix $\omega \in \Omega$ and $\phi \in C^2(X)$ with $\|\phi\|_{C^2} \le 1$. Set $\boldsymbol{\nu}^{(n)}=\mathcal{K}_\alpha^n \boldsymbol{\nu}$ and split
\[
\mathcal{R}(\boldsymbol{\nu}^{(n)} - \boldsymbol{\nu}_\alpha)  = 
\mathcal{R}_1(\boldsymbol{\nu}^{(n)} - \boldsymbol{\nu}_\alpha)  + 
\mathcal{R}_2(\boldsymbol{\nu}^{(n)} - \boldsymbol{\nu}_\alpha) ,
\]
with
\[
\langle (\mathcal{R}_1(\boldsymbol{\nu}^{(n)} - \boldsymbol{\nu}_\alpha))_\omega, \phi \rangle = \sum_{i \in \mathbb{I}}  p_{\alpha,i}' (\omega) \langle {\nu}^{(n)}_{\theta_{\alpha,i}(\omega)} -  {\nu}_{\alpha,\theta_{\alpha,i}(\omega)}, \phi \circ g_{\alpha, \theta_{\alpha,i}(\omega)} \rangle
\]
and
\[
\langle (\mathcal{R}_2(\boldsymbol{\nu}^{(n)} - \boldsymbol{\nu}_\alpha))_\omega, \phi \rangle = \sum_{i \in \mathbb{I}} p_{\alpha,i}(\omega) \langle {\nu}^{(n)}_{\theta_{\alpha,i}(\omega)} -  {\nu}_{\alpha,\theta_{\alpha,i}(\omega)}, \partial_\omega (\phi \circ g_{\alpha, \theta_{\alpha,i}(\omega)}) \rangle.
\]
By the chain rule we have \(\|\phi \circ g_{\alpha, \theta_{\alpha,i}(\omega)}\|_{C^1} \le   \|\phi\|_{C^1} \|g_{\alpha, \theta_{\alpha,i}(\omega)}\|_{C^1}\), and the definition of the dual norm gives
 \begin{align*}
|\langle (\mathcal{R}_1(\boldsymbol{\nu}^{(n)} - \boldsymbol{\nu}_\alpha))_\omega, \phi \rangle|
&\le \sum_{i \in \mathbb{I}} |  p_{\alpha,i}' (\omega)| \, \| {\nu}^{(n)}_{\theta_{\alpha,i}(\omega)} -  {\nu}_{\alpha,\theta_{\alpha,i}(\omega)}\|_{1^*} \, \|\phi \circ g_{\alpha, \theta_{\alpha,i}(\omega)}\|_{C^1}\\
& \le  \sum_{i \in \mathbb{I}} |  p_{\alpha,i}' (\omega)| \, \|{\nu}^{(n)}_{\theta_{\alpha,i}(\omega)} -  {\nu}_{\alpha,\theta_{\alpha,i}(\omega)}\|_{1^*} \,  \|\phi\|_{C^1} \|g_{\alpha, \theta_{\alpha,i}(\omega)}\|_{C^1}.
\end{align*}
Taking supremum over $\phi\in C^2(X)$ with   $\|\phi\|_{C^2} \le 1$ and noting that   $\|\phi\|_{C^1}\le \|\phi\|_{C^2} \le 1$, we get
\[
\|(\mathcal{R}_1(\boldsymbol{\nu}^{(n)} - \boldsymbol{\nu}_\alpha))_\omega\|_{2^*} 
\le  \sum_{i \in \mathbb{I}} |  p_{\alpha,i}' (\omega)| \, \|{\nu}^{(n)}_{\theta_{\alpha,i}(\omega)} -  {\nu}_{\alpha,\theta_{\alpha,i}(\omega)}\|_{1^*} \, \|g_{\alpha, \theta_{\alpha,i}(\omega)}\|_{C^1}.
\]
Similarly, using the chain rule to get the bound \(\|\partial_\omega (\phi \circ g_{\alpha, \theta_{\alpha,i}(\omega)})\|_{C^1} \le   \|\phi\|_{C^1} \|\partial_\omega g_{\alpha, \theta_{\alpha,i}(\omega)}\|_{C^1}\), it follows that
\[
\|(\mathcal{R}_2(\boldsymbol{\nu}^{(n)} - \boldsymbol{\nu}_\alpha))_\omega\|_{2^*} 
\le  \sum_{i \in \mathbb{I}} p_{\alpha,i}(\omega) \, \|{\nu}^{(n)}_{\theta_{\alpha,i}(\omega)} -  {\nu}_{\alpha,\theta_{\alpha,i}(\omega)}\|_{1^*} \, \|\partial_\omega g_{\alpha, \theta_{\alpha,i}(\omega)}\|_{C^1}.
\]
Since  $\sum_{i \in \mathbb{I}} p_{\alpha,i}(\omega) =1$, for all $\omega\in\Omega$, using \ref{as.g} we deduce 
\[
\sup_{\omega \in \Omega} \sum_{i \in \mathbb{I}} |  p_{\alpha,i}' (\omega)| \, \|g_{\alpha, \theta_{\alpha,i}(\omega)}\|_{C^2} < \infty
\qand
\sup_{\omega \in \Omega} \sum_{i \in \mathbb{I}} p_{\alpha,i}(\omega) \, \|\partial_\omega g_{\alpha, \theta_{\alpha,i}(\omega)}\|_{C^1} < \infty.
\]
Hence, there exists a uniform constant $C > 0$ such that
\begin{equation}\label{eq.ene}
\|(\mathcal{R}(\boldsymbol{\nu}^{(n)} - \boldsymbol{\nu}_\alpha))_\omega\|_{2^*}
 \le C
  \sup_{\omega\in\Omega} \|{\nu}^{(n)}_{ \omega} -  {\nu}_{\alpha, \omega}\|_{1^*}.
\end{equation}
By Lemma~\ref{le.contractinp}, the iterates $\boldsymbol{\nu}^{(n)}$ converge in Wasserstein (Lipschitz  dual norm)  to $\boldsymbol{\nu}_\alpha$, meaning that 
\[
\sup_{\omega \in \Omega} \|{\nu}^{(n)}_{ \omega} -  {\nu}_{\alpha, \omega}\|_{\mathrm{Lip}^*} \longrightarrow 0.
\]
Since the inclusion $ \lip(X)^*\hookrightarrow C^1(X)^*$ is continuous, it follows that
\[
\sup_{\omega \in \Omega} \|{\nu}^{(n)}_{ \omega} -  {\nu}_{\alpha, \omega}\|_{1^*} \longrightarrow 0.
\]
Combining this with~\eqref{eq.ene},   the proof is complete.
\end{proof}

 \begin{remark} \label{re.c2} The regularity of the observable $\phi$ is only relevant for the estimate of~$\mathcal{R}_2$, and one might hope to take $\phi \in C^1(X)$, with convergence of measures in $C^0(X)^*$, which would allow a stronger final result in the space of sections valued in $C^1(X)^*$. 
However, the crucial step in the proof uses the convergence of $\mathcal{K}_\alpha^n \boldsymbol{\nu}$ to $\boldsymbol{\nu}_\alpha$ in $\mathrm{Lip}(X)^*$ (Wasserstein distance) to deduce convergence in $C^1(X)^*$. 
If we attempted to go one unit lower, we would need to deduce convergence in $C^0(X)^*$ from convergence in $\mathrm{Lip}(X)^*$, which is obviously false.
\end{remark}


In the following result, we make use of the constant $0<\lambda<1$ provided by~\ref{as.g}. 
Observe that, by Remark~\ref{re.bell}, we have
\begin{equation}\label{eq.newels}
\mathcal{C}_2(g_{\alpha, \theta_{\alpha,i}}) \le \max\{ L_1 + L_2, \, L_1^2 \} 
\le 
\max\{ L_1 + L_2 + L_3, \, L_1^2 + 3L_1L_2, \, L_1^3 \} \le \lambda.
\end{equation}

\begin{lemma} 
\label{le.Tcontraction}
For every  $\bm\xi,\bm\zeta \in \mathbf{E}_{2}^0$ we have
$$\|\mathcal T\bm\xi -\mathcal T\bm\zeta\|_{2}\le \lambda \| \bm\xi - \bm\zeta\|_{2}.$$
\end{lemma}

\begin{proof}
For any $\bm\xi,\bm\zeta \in \mathbf E_2^0$, we have for all $\omega\in\Omega$
\[
\| (\mathcal{T} \bm\xi)_\omega - (\mathcal{T} \bm\zeta)_\omega \|_{2^*} \le \sum_{i \in \mathbb I} p_{\alpha,i}(\omega) \,  |\theta_{\alpha,i}'(\omega) | \, \| (g_{\alpha, \theta_{\alpha,i}})_* (\xi_{\theta_{\alpha,i}(\omega)} - \zeta_{\theta_{\alpha,i}(\omega)}) \|_{2^*}.
\]
Since $\sum_{i \in \mathbb{I}} p_{\alpha,i}(\omega) =1$, it follows from \ref{as.full}--\ref{as.theta}, Lemma~\ref{lem:push-ck} and~\eqref{eq.newels}    that 
\[
\| (\mathcal{T} \bm\xi)_\omega - (\mathcal{T} \bm\zeta)_\omega \|_{2^*} \le \lambda  \|  \xi_{\theta_{\alpha,i}(\omega)} - \zeta_{\theta_{\alpha,i}(\omega)} \|_{2^*}.
\]
Taking supremum in $\omega\in\Omega$ on both sides of the last inequality  we get the conclusion.
\end{proof}

\begin{proposition} \label{pr.differfixed}
$ \boldsymbol{\nu}_\alpha:\Omega \to C^2(X)^*$ is a  $C^1$ map.
\end{proposition}
\begin{proof}
By Lemma~\ref{le.Tcontraction} and the Banach fixed point theorem, the operator $\mathcal{T}$ admits a unique fixed point $\bm\xi^* \in \mathbf{E}_{2}^0$, that is,
\[
\mathcal{T}\bm\xi^* = \bm\xi^*.
\]
Fix a smooth section $\boldsymbol{\nu}\in \mathbf P$ and consider  $\boldsymbol{\nu}^{(n)}= \mathcal{K}_\alpha^n \boldsymbol{\nu} $. Define $\bm\xi^{(n)}\in\mathbf{E}_{2}^0$ by
\begin{equation}\label{eq.defxi}
 \xi^{(n)}_\omega = \partial_\omega  {\nu}^{(n)}_\omega.
\end{equation}
%
We define the error at step $n\ge 0$ by
\[
E_n = \|\bm\xi^{(n)} - \bm\xi^*\|_{2}.
\]
From the iterative relation
\[
\xi^{(n+1)}_\omega  = \sum_{i \in \mathbb I} p_{\alpha,i}(\omega) \, \theta_{\alpha,i}'(\omega) \, (g_{\alpha, \theta_{\alpha,i}(\omega)})_* \xi_{\theta_{\alpha,i}(\omega)}
+ \big(\mathcal{R} \boldsymbol{\nu}^{(n)}\big)_\omega
.
\]
and the fixed point equation 
$$
 \xi^*_\omega = (\mathcal{T} \bm\xi^*)_\omega = 
\sum_{i \in \mathbb I} p_{\alpha,i}(\omega) \, \theta_{\alpha,i}'(\omega) \, (g_{\alpha, \theta_{\alpha,i}(\omega)})_* \xi_{\theta_{\alpha,i}(\omega)}
+ \big(\mathcal{R} \boldsymbol{\nu}_\alpha \big)_\omega
,$$
 we obtain
\begin{align*}
 \xi^{(n+1)}_\omega -  \xi^*_\omega &= \sum_{i \in \mathbb I} p_{\alpha,i}(\omega) \, \theta_{\alpha,i}'(\omega) \, (g_{\alpha, \theta_{\alpha,i}(\omega)})_*  \big( \xi_{\theta_{\alpha,i}(\omega)}^{(n)} - \xi_{\theta_{\alpha,i}(\omega)}\big)
+ \big( \big(\mathcal{R} \boldsymbol{\nu}^{(n)}\big)_\omega - \big(\mathcal{R} \boldsymbol{\nu}_\alpha \big)_\omega \big)\\
&= \big(\mathcal{T} (\bm\xi^{(n)}  -   \bm\xi^*)\big)_\omega +   \big(\mathcal{R} (\boldsymbol{\nu}^{(n)}-\boldsymbol{\nu}_\alpha) \big)_\omega  
\end{align*}
Taking the norm in $C^2(X)^*$,  supremum over $\omega \in \Omega$ and using  Lemma~\ref{le.Tcontraction}, we get
\[
E_{n+1} \le \lambda E_n + \epsilon_n.
\]
Observing that by Lemma~\ref{le.remainder} we have $\epsilon_n\to 0$ as $n\to\infty$, standard estimates for linear recursions with vanishing forcing terms give
 \begin{equation}\label{eq.enestimate}
E_n \le \lambda^n E_0 + \sum_{k=0}^{n-1} \lambda^k \epsilon_{n-1-k} \to 0 ,\quad \text{as } n \to \infty.
\end{equation}
Now, for any $\omega_0, \omega \in \Omega$, the fundamental theorem of calculus   in the space $C^2(X)^*$ yields 
\[
 {\nu}^{(n)}_\omega =  {\nu}^{(n)}_{\omega_0} + \int_{\omega_0}^\omega \xi^{(n)}_\theta \, d\theta.
\]
Taking the limit $n \to \infty$, we have 
$$
 {\nu}^{(n)}_\omega \to  {\nu}_{\alpha,\omega}\qand  {\nu}^{(n)}_{\omega_0} \to  {\nu}_{\alpha,\omega_0},$$ with convergence in $\lip(X)^*$, 
by Lemma~\ref{le.contractinp}, hence   in $C^2(X)^*$, by the continuous inclusion of $\lip(X)^*$ in $C^2(X)^*$.  On the other hand, it follows from \eqref{eq.enestimate} that
    \[
    \left\| \int_{\omega_0}^\omega  \xi^{(n)}_\theta d\theta - \int_{\omega_0}^\omega  \xi^*_\theta d\theta \right\|_{2^*} \le |\omega - \omega_0| \sup_{\theta\in\Omega} \|\xi^{(n)}_\theta - \xi^*_\theta\|_{2^*} =|\omega - \omega_0| E_n\to 0.
    \]
Thus,
\[
 {\nu}_{\alpha,\omega} =  {\nu}_{\alpha,\omega_0} + \int_{\omega_0}^\omega \xi^*_\theta \, d\theta.
\]
Since $\bm\xi^* \in  \mathbf{E}_{2}^0$, the fundamental theorem of calculus  gives  that $  \boldsymbol{\nu}_\alpha$ is a $C^1$ map from $\Omega$ to $C^2(X)^*$, with derivative
$
 \boldsymbol{\nu}_\alpha' = \bm\xi^*.
$
\end{proof}

%

\subsubsection{Parameter differentiability}
Here we prove that  the map $\alpha \mapsto \mathcal K_\alpha \bm\nu_{\alpha_0}$ is differentiable in $\mathbf E_{3}^0$ at any $\alpha_0\in J$, thus obtaining~\ref{as.spaced}.
We denote differentiation with respect to the parameter $\alpha$ by a dot over the corresponding function. For notational convenience, we extend this convention to operators to indicate the derivative of the associated maps with respect to $\alpha$. We emphasize that this notation is purely formal; in particular, no claim is made regarding the existence of derivatives in the uniform operator norm. 
We will make use of the following lemma, whose proof follows from the chain rule
for second-order derivatives by standard arguments.

\begin{lemma} \label{le.aux}
For any $\phi \in C^3(X)$ and any two maps $h_1, h_2 \in C^2(X, X)$, we have
\begin{enumerate}
\item $\| \phi \circ h_1 - \phi \circ h_2 \|_{C^2} \le C \sum_{k=0}^2 \| D^k h_1 - D^k h_2 \|_{C^0}$, with
\[
C = \max \Big\{ \|\phi\|_{C^1}, \ \|\phi\|_{C^2} (1+L_1 + L_2), \ \|\phi\|_{C^3} (1+L_1^2) \Big\},
\]
where $L_j = \max\{\|D^j h_1\|_{C^0}, \|D^j h_2\|_{C^0}\}$;

\item $\| \phi \circ h \|_{C^2} \le C' \| \phi \|_{C^2}$, with
\[
C' = \max \big\{ 1, \ \diam(X)+\|Dh\|_{C^0} + \|D^2 h\|_{C^0}, \ \|Dh\|_{C^0}^2 \big\}.
\]
\end{enumerate}
\end{lemma}

\begin{proposition}\label{pr.diffKi}
 Assume that  $\bm\nu\in \mathbf P  \cap \mathbf E_{2} '$. Then,
 \begin{enumerate}
\item  the map
\(
\alpha \mapsto  \mathcal{K}_{\alpha,i}\bm\nu
\)
is differentiable in $\mathbf{E}_{3}^0$, for each $i\in\mathbb I$;
\item the series
\(
\sum_{i\in\mathbb I}  \dot{\mathcal K}_{\alpha,i} \bm\nu
\)
converges  in $ \mathbf{E}_3^0$ uniformly in $\alpha$.
\end{enumerate}

\end{proposition}

\begin{proof} For the first item, we are going to prove that the derivative $\dot{\mathcal{K}}_{\alpha,i}\bm\nu$ of  
\(\mathcal K_{\alpha,i} \bm\nu\) with respect to the parameter  $\alpha\in J$ 
   is given for any $\omega \in \Omega$ and $\phi \in C^3(X)$ by
\begin{align}
\langle (\dot{\mathcal{K}}_{\alpha,i}\bm\nu)_\omega, \phi \rangle
=&\; \dot p_{\alpha,i}(\omega) \, \langle \nu_{\theta_{\alpha,i}(\omega)}, \phi \circ g_{\alpha, \theta_{\alpha,i}(\omega)} \rangle \nonumber\\
&+ p_{\alpha,i}(\omega) \, \langle \nu'_{\theta_{\alpha,i}(\omega)} \dot{\theta}_{\alpha,i}(\omega), \phi \circ g_{\alpha, \theta_{\alpha,i}(\omega)} \rangle \nonumber\\
&+ p_{\alpha,i}(\omega) \, \langle \nu_{\theta_{\alpha,i}(\omega)}, D \phi \circ g_{\alpha, \theta_{\alpha,i}(\omega)} \bm\cdot \partial_\alpha g_{\alpha, \theta_{\alpha,i}(\omega)} \rangle \nonumber\\
&+ p_{\alpha,i}(\omega) \, \langle \nu_{\theta_{\alpha,i}(\omega)}, D \phi \circ g_{\alpha, \theta_{\alpha,i}(\omega)} \bm\cdot \partial_{\omega'} g_{\alpha, \theta_{\alpha,i}(\omega)} \, \dot{\theta}_{\alpha,i}(\omega) \rangle. \label{eq.finally}
\end{align}
Given any $\alpha_0 \in J$ and ${\epsilon} \neq 0$, set $\alpha = \alpha_0 + {\epsilon}$, $\omega_{\epsilon} = \theta_{\alpha,i}(\omega)$, and $\omega_0 = \theta_{\alpha_0,i}(\omega)$. For $\phi \in C^3(X)$, consider the difference quotient
\[
\langle \Delta_{\epsilon}(\omega),\phi\rangle = \frac{1}{{\epsilon}} \Big[ \langle(\mathcal{K}_{\alpha,i} \bm\nu)_\omega,\phi\rangle - \langle(\mathcal{K}_{\alpha_0,i} \bm\nu)_\omega,\phi\rangle \Big].
\]
We decompose 
$$\langle \Delta_{\epsilon}(\omega),\phi\rangle - \langle (\dot{\mathcal{K}}_{\alpha_0,i} \bm\nu)_\omega, \phi \rangle
=\langle A_{\epsilon}(\omega),\phi\rangle +\langle B_{\epsilon}(\omega),\phi\rangle +\langle C_{\epsilon}(\omega),\phi\rangle +\langle D_{\epsilon}(\omega),\phi\rangle
$$ 
with
\begin{align*}
\langle A_{\epsilon}(\omega), \phi \rangle &= \left( \frac{p_{\alpha,i}(\omega) - p_{\alpha_0,i}(\omega)}{{\epsilon}} \right) \langle \nu_{\omega_{\epsilon}}, \phi \circ g_{\alpha, \omega_{\epsilon}} \rangle - \dot{p}_{\alpha_0,i}(\omega) \langle \nu_{\omega_0}, \phi \circ g_{\alpha_0, \omega_0} \rangle, \\[2mm]
\langle B_{\epsilon}(\omega), \phi \rangle &= p_{\alpha_0,i}(\omega) \left[ \left\langle \frac{\nu_{\omega_{\epsilon}} - \nu_{\omega_0}}{{\epsilon}}, \phi \circ g_{\alpha, \omega_{\epsilon}} \right\rangle - \langle \nu'_{\omega_0} \dot{\theta}_{\alpha_0,i}(\omega), \phi \circ g_{\alpha_0, \omega_0} \rangle \right], \\[2mm]
\langle C_{\epsilon}(\omega), \phi \rangle &= p_{\alpha_0,i}(\omega) \left\langle \nu_{\omega_0},  \frac{\phi \circ g_{\alpha, \omega_{\epsilon}} - \phi \circ g_{\alpha_0, \omega_{\epsilon}}}{{\epsilon}} - D\phi(g_{\alpha_0, \omega_0}) \bm \cdot \partial_\alpha g_{\alpha_0, \omega_0} \right\rangle, \\[2mm]
\langle D_{\epsilon}(\omega), \phi \rangle &= p_{\alpha_0,i}(\omega) \left\langle \nu_{\omega_0}, \frac{\phi \circ g_{\alpha_0, \omega_{\epsilon}} - \phi \circ g_{\alpha_0, \omega_0}}{{\epsilon}} - D\phi(g_{\alpha_0, \omega_0}) \bm\cdot \partial_{\omega} g_{\alpha_0, \omega_0} \dot{\theta}_{\alpha_0,i}(\omega) \right\rangle.
\end{align*}
We need to show that each of these four terms converges to $0$ as $\epsilon\to 0$, uniformly in $\omega \in \Omega$ and $\phi \in C^3(X)$ with $\|\phi\|_{C^3} \le 1$.

\medskip
\noindent{1. Term $A_{\epsilon}$:}
  By the mean value theorem applied to   $\alpha \mapsto p_{\alpha,i}(\omega)$, there exists $\alpha' = \alpha_0 + {\epsilon}'$ with $|{\epsilon}'| < |{\epsilon}|$ such that $$\frac{1}{\epsilon}(p_{\alpha,i}(\omega) - p_{\alpha_0,i}(\omega)) = \dot{p}_{\alpha',i}(\omega).$$ Substituting this into the expression for $\langle A_{\epsilon}(\omega), \phi \rangle$, we decompose  
\begin{align*}
|\langle &A_{\epsilon}(\omega), \phi \rangle| = | \dot{p}_{\alpha',i}(\omega) \langle \nu_{\omega_\epsilon}, \phi \circ g_{\alpha, \omega_{\epsilon}} \rangle - \dot{p}_{\alpha_0,i}(\omega) \langle \nu_{\omega_0}, \phi \circ g_{\alpha_0, \omega_0} \rangle | \\
&\le | (\dot{p}_{\alpha',i}(\omega) - \dot{p}_{\alpha_0,i}(\omega)) \langle \nu_{\omega_{\epsilon}}, \phi \circ g_{\alpha, \omega_{\epsilon}} \rangle | + | \dot{p}_{\alpha_0,i}(\omega) ( \langle \nu_{\omega_{\epsilon}}, \phi \circ g_{\alpha, \omega_{\epsilon}} \rangle - \langle \nu_{\omega_0}, \phi \circ g_{\alpha_0, \omega_0} \rangle ) | \\
&\le \underbrace{|\dot{p}_{\alpha',i}(\omega) - \dot{p}_{\alpha_0,i}(\omega)|}_{\text{(I)}} 
\|\phi\|_{C^0} + \|\dot{p}_{\alpha_0,i}\|_{C^0} \underbrace{| \langle \nu_{\omega_{\epsilon}}, \phi \circ g_{\alpha, \omega_{\epsilon}} \rangle - \langle \nu_{\omega_0}, \phi \circ g_{\alpha_0, \omega_0} \rangle |}_{\text{(II)}}.
\end{align*}
Regarding term (II), we apply a   telescoping
\begin{align*}
| \langle \nu_{\omega_{\epsilon}}, \phi \circ g_{\alpha, \omega_{\epsilon}} \rangle  - \langle \nu_{\omega_0}, \phi \circ g_{\alpha_0, \omega_0} \rangle | &\le | \langle \nu_{\omega_{\epsilon}} - \nu_{\omega_0}, \phi \circ g_{\alpha, \omega_{\epsilon}} \rangle | + | \langle \nu_{\omega_0}, \phi \circ g_{\alpha, \omega_{\epsilon}} - \phi \circ g_{\alpha_0, \omega_0} \rangle | \\
&\le \underbrace{\|\nu_{\omega_{\epsilon}} - \nu_{\omega_0}\|_{2^*} \|\phi \circ g_{\alpha, \omega_{\epsilon}} \|_{C^2}}_{\text{(III)}} +   \underbrace{\|\phi \circ g_{\alpha, \omega_{\epsilon}} - \phi \circ g_{\alpha_0, \omega_0}\|_{C^0}}_{\text{(IV)}}.
\end{align*}
It follows from  Lemma~\ref{le.aux} that
\begin{equation}\label{eq.compos}
 \|\phi \circ g_{\alpha, \omega_{\epsilon}} \|_{C^2}\le  \max \big\{ 1, \diam(X)+ \|Dg_{\alpha, \omega_{\epsilon}}\|_{C^0} + \|D^2 g_{\alpha, \omega_{\epsilon}}\|_{C^0} + \|Dg_{\alpha, \omega_{\epsilon}}\|_{C^0}^2 \big\}\|\phi\|_{C^2},
\end{equation}
which together with  the mean value theorem applied to $\alpha \mapsto \nu_{\theta_{\alpha,i}(\omega)}$, yields
 $$
  \text{(III)}\le \epsilon   \|\bm\nu'\|_2 \| {\dot\theta}_{\alpha,i} \|_{C^0} \max \big\{ 1,  \diam(X)+\|Dg_{\alpha, \omega_{\epsilon}}\|_{C^0} + \|D^2 g_{\alpha, \omega_{\epsilon}}\|_{C^0} + \|Dg_{\alpha, \omega_{\epsilon}}\|_{C^0}^2 \big\}\|\phi\|_{C^2}
  $$ 
  By the mean value theorem, we also have
  $$
  \text{(IV)}
  \le
  \|\phi\|_{C^1}\|g_{\alpha, \omega_{\epsilon}} -  g_{\alpha_0, \omega_0}\|_{C_0}.
  $$
By the expressions obtained for (I)--(IV), taking the supremum over  all $\phi$ with $\|\phi\|_{C^3} \le 1$ and using~\ref{as.full}-\ref{as.p}, we conclude  that $\|A_\epsilon(\omega)\|_{3^*} \to 0$ uniformly in~$\omega$ as $\epsilon \to 0$.

\medskip
\noindent{2. Term $ B_\epsilon$:}
By the mean value theorem applied to the map $\alpha \mapsto \nu_{\theta_{\alpha,i}(\omega)}$, there exists some $\alpha'=\alpha_0+ \epsilon' $ with $| \epsilon'|<| \epsilon|$ such that $$\frac{1}{ \epsilon}(\nu_{\omega_{\epsilon}} - \nu_{\omega_0}) = \nu'_{\omega_{ \epsilon'}} \dot{\theta}_{\alpha',i}(\omega).$$
 Substituting this in $\langle B_ \epsilon(\omega),\phi\rangle$, we get 
\begin{align*}
\langle &B_\epsilon(\omega),\phi\rangle  = | \langle \nu'_{\omega_{\epsilon'}} \dot{\theta}_{\alpha',i}(\omega), \phi \circ g_{\alpha, \omega_{\epsilon}} \rangle - \langle \nu'_{\omega_0} \dot{\theta}_{\alpha_0,i}(\omega), \phi \circ g_{\alpha_0, \omega_0} \rangle | \\
&\le | \langle \nu'_{\omega_{\epsilon'}} \dot{\theta}_{\alpha',i}(\omega) - \nu'_{\omega_0} \dot{\theta}_{\alpha_0,i}(\omega), \phi \circ g_{\alpha, \omega_{\epsilon}} \rangle | + | \langle \nu'_{\omega_0} \dot{\theta}_{\alpha_0,i}(\omega), \phi \circ g_{\alpha, \omega_{\epsilon}} - \phi \circ g_{\alpha_0, \omega_0} \rangle | \\
&\le
 \underbrace{\| \nu'_{\omega_{\epsilon'}} \dot{\theta}_{\alpha',i}(\omega) - \nu'_{\omega_0} \dot{\theta}_{\alpha_0,i}(\omega) \|_{2^*}}_{\text{(I)}} \|\phi \circ g_{\alpha, \omega_{\epsilon}} \|_{C^2} 
 +
  \|\bm\nu'\|_2 \| {\dot\theta}_{\alpha_0,i} \|_{C^0}\underbrace{ \| \phi \circ g_{\alpha, \omega_{\epsilon}} - \phi \circ g_{\alpha_0, \omega_0} \|_{C^2}}_{\text{(II)}}.
\end{align*}
To properly bound term  (I), note   that
\begin{align*}
  \| \nu'_{\omega_{\epsilon'}} \dot{\theta}_{\alpha',i}(\omega) - \nu'_{\omega_0} \dot{\theta}_{\alpha_0,i}(\omega) \|_{2^*}   
&\le  \| (\nu'_{\omega_{\epsilon'}} - \nu'_{\omega_0}) \dot{\theta}_{\alpha',i}(\omega) \|_{2^*} + \| \nu'_{\omega_0} (\dot{\theta}_{\alpha',i}(\omega) - \dot{\theta}_{\alpha_0,i}(\omega)) \|_{2^*}     \\
&\le   \| \nu'_{\omega_{\epsilon'}} - \nu'_{\omega_0} \|_{2^*} |\dot{\theta}_{\alpha',i}(\omega)| + \| \nu'_{\omega_0} \|_{2^*} |\dot{\theta}_{\alpha',i}(\omega) - \dot{\theta}_{\alpha_0,i}(\omega)|   \\
&\le   \| \nu'_{\omega_{\epsilon'}} - \nu'_{\omega_0} \|_{2^*} \| {\dot\theta}_{\alpha',i} \|_{C^0} + \| \bm\nu'  \|_2 |\dot{\theta}_{\alpha',i}(\omega) - \dot{\theta}_{\alpha_0,i}(\omega)|   .
\end{align*}
Regarding term (II), Lemma~\ref{le.aux} yields the estimate
\[
\bigl\|
\phi \circ g_{\alpha,\theta_{\alpha,i}(\omega)}
-
\phi \circ g_{\alpha_0,\theta_{\alpha_0,i}(\omega)}
\bigr\|_{C^2}
\le
C
\sum_{k=0}^2
\| D^k g_{\alpha,\theta_{\alpha,i}(\omega)} - D^k g_{\alpha_0,\theta_{\alpha_0,i}(\omega)} \|_{C^0},
\]
where 
\begin{align*}
C
=
\max\Big\{
\|\phi\|_{C^1},
\; &
\|\phi\|_{C^2}
\bigl(
1
+
\|g_{\alpha,\theta_{\alpha,i}(\omega)}\|_{C^2}
+ 
\|g_{\alpha_0,\theta_{\alpha_0,i}(\omega)}\|_{C^2}
\bigr),
\\
&
\|\phi\|_{C^3}
\bigl(
1
+
\|g_{\alpha,\theta_{\alpha,i}(\omega)}\|_{C^1}^2
+
\|g_{\alpha_0,\theta_{\alpha_0,i}(\omega)}\|_{C^1}^2
\bigr)
\Big\}.
\end{align*}
By the expressions obtained for (I)--(II), taking the supremum over all $\phi$ with $\|\phi\|_{C^3} \le 1$, using~\ref{as.full}--\ref{as.p} together with~\eqref{eq.compos} and the fact that $\bm\nu \in \mathbf E_{2} '$ (and hence $\bm\nu$ is uniformly continuous), we conclude that $\|B_\epsilon(\omega)\|_{3^*} \to 0$ uniformly in~$\omega$ as $\epsilon \to 0$.

\medskip
\noindent{3. Term $C_\epsilon$:} 
By the fundamental theorem of calculus applied to the map $\alpha \mapsto \phi \circ g_{\alpha, \omega_{\epsilon}}$, for each $x \in X$ we have
\[
\frac{1}{\epsilon}(\phi(g_{\alpha_0+\epsilon, \omega_{\epsilon}}(x)) - \phi(g_{\alpha_0, \omega_{\epsilon}}(x))) = \int_0^1 \nabla \phi ( g_{\alpha_0+t\epsilon, \omega_{\epsilon}}(x) ) \cdot \partial_\alpha g_{\alpha_0+t\epsilon, \omega_{\epsilon}}(x) \, dt.
\]
Substituting this identity into the expression for $\langle C_\epsilon(\omega), \phi \rangle$ and applying Fubini's theorem to exchange the action of the measure $\nu_{\omega_0}$ with the integral over $t \in [0,1]$, we write
\begin{align*}
\langle C_\epsilon(\omega), \phi \rangle &= \int_0^1 \Big( \langle \nu_{\omega_0}, (\nabla \phi \circ g_{\alpha_0+t\epsilon, \omega_{\epsilon}}) \cdot \partial_\alpha g_{\alpha_0+t\epsilon, \omega_{\epsilon}} \rangle - \langle \nu_{\omega_0}, (\nabla \phi \circ g_{\alpha_0, \omega_0}) \cdot \partial_\alpha g_{\alpha_0, \omega_0} \rangle \Big) \, dt.
\end{align*}
We telescope the integrand by adding and subtracting the term $\langle \nu_{\omega_0}, (\nabla \phi \circ g_{\alpha_0, \omega_0}) \cdot \partial_\alpha g_{\alpha_0+t\epsilon, \omega_{\epsilon}} \rangle$. This yields:
\begin{align}
\langle C_\epsilon(\omega), \phi \rangle &= \int_0^1 \langle \nu_{\omega_0}, (\nabla \phi \circ g_{\alpha_0+t\epsilon, \omega_{\epsilon}} - \nabla \phi \circ g_{\alpha_0, \omega_0}) \cdot \partial_\alpha g_{\alpha_0+t\epsilon, \omega_{\epsilon}} \rangle \, dt \nonumber\\
&\quad + \int_0^1 \langle \nu_{\omega_0}, (\nabla \phi \circ g_{\alpha_0, \omega_0}) \cdot (\partial_\alpha g_{\alpha_0+t\epsilon, \omega_{\epsilon}} - \partial_\alpha g_{\alpha_0, \omega_0}) \rangle \, dt.\label{eq.lunch1}
\end{align}
For the first   integrand, we define the vector-valued measure $ {\mu}_{t,\epsilon} = (\partial_\alpha g_{\alpha_0+t\epsilon, \omega_{\epsilon}}) \nu_{\omega_0}$, so that it  represents the difference of the push-forwards of $ {\mu}_{t,\epsilon}$ by the maps $h_1 = g_{\alpha_0+t\epsilon, \omega_{\epsilon}}$ and $h_0 = g_{\alpha_0, \omega_0}$ acting on the gradient field $\nabla \phi$:
\[
 \langle \nu_{\omega_0}, (\nabla \phi \circ g_{\alpha_0+t\epsilon, \omega_{\epsilon}} - \nabla \phi \circ g_{\alpha_0, \omega_0}) \cdot \partial_\alpha g_{\alpha_0+t\epsilon, \omega_{\epsilon}} \rangle = \langle (g_{\alpha_0+t\epsilon, \omega_{\epsilon}})_* {\mu}_{t,\epsilon} - (g_{\alpha_0, \omega_0})_* {\mu}_{t,\epsilon}, \nabla \phi \rangle .
\]
Applying Lemma \ref{le.salvation} component-wise to the vector measure ${\mu}_{t,\epsilon}$ with the test function $\nabla \phi \in C^2(X)$, and using the fact that $\|{\mu}_{t,\epsilon}\|_{0^*} \le \|\partial_\alpha g_{\alpha_0+t\epsilon,\theta_{\alpha,i}}\|_{C^0}  $, we obtain
\begin{align}
| \langle \nu_{\omega_0}, (\nabla \phi &\circ g_{\alpha_0+t\epsilon, \omega_{\epsilon}} - \nabla \phi \circ g_{\alpha_0, \omega_0}) \cdot \partial_\alpha g_{\alpha_0+t\epsilon, \omega_{\epsilon}} \rangle |  \nonumber \\
&\le  \| {\mu}_{t,\epsilon}\|_{0^*} \|\nabla \phi\|_{C^1} \| g_{\alpha_0+t\epsilon,  \theta_{\alpha,i}} - g_{\alpha_0,  \theta_{\alpha_0,i}} \|_{C^0} \nonumber \\
&\le \|\partial_\alpha g_{\alpha_0+t\epsilon,\theta_{\alpha,i}}\|_{C^0}   \|\nabla \phi\|_{C^1} \| g_{\alpha_0+t\epsilon,  \theta_{\alpha,i}} - g_{\alpha_0,  \theta_{\alpha_0,i}} \|_{C^0} \nonumber \\
&\le \|\partial_\alpha g_{\alpha_0+t\epsilon,\theta_{\alpha,i}}\|_{C^0}    \|\phi\|_{C^2} \| g_{\alpha_0+t\epsilon,  \theta_{\alpha,i}} - g_{\alpha_0,  \theta_{\alpha_0,i}} \|_{C^0} .
\label{eq.lunch2}
\end{align}
For the second integral, recalling that $\bm\nu\in\mathbf P$ and using Cauchy-Schwarz inequality, 
\begin{align}
|\langle \nu_{\omega_0}, (\nabla \phi \circ g_{\alpha_0, \omega_0}) & \cdot (\partial_\alpha g_{\alpha_0+t\epsilon, \omega_{\epsilon}} - \partial_\alpha g_{\alpha_0, \omega_0}) \rangle| \nonumber\\
 &\le  \|\nabla \phi\|_{C^0} \|\partial_\alpha g_{\alpha_0+t\epsilon, \theta_{\alpha,i}} - \partial_\alpha g_{\alpha_0, \theta_{\alpha_0,i}}\|_{C^0} \nonumber \\
&\le   \|\phi\|_{C^1} \|\partial_\alpha g_{\alpha_0+t\epsilon, \theta_{\alpha,i}} - \partial_\alpha g_{\alpha_0, \theta_{\alpha_0,i}}\|_{C^0}.
\label{eq.lunch3}
\end{align}
Taking the supremum over   $\|\phi\|_{C^3} \le 1$ in \eqref{eq.lunch1}--\eqref{eq.lunch3}, we conclude:
\begin{align*}
\| C_\epsilon(\omega) \|_{3^*} \le    \int_0^1 \Big( &\|\dot g_{\alpha_0,\theta_{\alpha_0,i}}\|_{C^0} \| g_{\alpha_0+t\epsilon, \omega_{\epsilon}} - g_{\alpha_0, \omega_0} \|_{C^0} 
+ \| \partial_\alpha g_{\alpha_0+t\epsilon, \omega_{\epsilon}} - \partial_\alpha g_{\alpha_0, \omega_0} \|_{C^0} \Big) \, dt.
\end{align*}
Since~\ref{as.g}  holds, the integrand  above vanishes uniformly in $\omega\in\Omega$, as $\epsilon \to 0$. 

\medskip
\noindent{4. Term $D_\epsilon$:} 
By the mean value theorem applied to $\alpha \mapsto \theta_\alpha(\omega)$, for each $\epsilon\neq0$ there exists   $\alpha'=\alpha_0+ \epsilon' $ with $| \epsilon'|<| \epsilon|$ such that 
\[ 
\frac{\theta_{\alpha_0+\epsilon}(\omega) - \theta_{\alpha_0}(\omega) }\epsilon= \dot{\theta}_{\alpha'}(\omega) . 
\]
Let $\omega_{\epsilon} = \theta_{\alpha_0+\epsilon}(\omega)$. We express the difference quotient in $\langle D_\epsilon(\omega), \phi \rangle$ using the fundamental theorem of calculus on the map $\omega' \mapsto \phi \circ g_{\alpha_0, \omega'}$ along the path from $\omega_0$ to $\omega_{\epsilon}$:
\begin{align}
\langle D_\epsilon(\omega), \phi \rangle = \dot{\theta}_{\alpha'}(\omega) \int_0^1 &\langle \nu_{\omega_0}, \nabla \phi \circ g_{\alpha_0, \omega_t} \cdot \partial_{\omega'} g_{\alpha_0, \omega_t} \rangle \, dt \nonumber \\
&\quad - \dot{\theta}_{\alpha_0}(\omega) \langle \nu_{\omega_0}, \nabla \phi \circ g_{\alpha_0, \omega_0} \cdot \partial_{\omega'} g_{\alpha_0, \omega_0} \rangle, \label{eq.dinner1}
\end{align}
where $\omega_t = \omega_0 + t(\omega_{\epsilon} - \omega_0)$. We telescope the expression by adding and subtracting the term $\dot{\theta}_{\alpha'}(\omega) \langle \nu_{\omega_0}, \nabla \phi \circ g_{\alpha_0, \omega_0} \cdot \partial_{\omega'} g_{\alpha_0, \omega_t} \rangle$:
\begin{align}
\langle D_\epsilon(\omega), \phi \rangle &= \dot{\theta}_{\alpha'}(\omega) \int_0^1 \langle \nu_{\omega_0}, (\nabla \phi \circ g_{\alpha_0, \omega_t} - \nabla \phi \circ g_{\alpha_0, \omega_0}) \cdot \partial_{\omega'} g_{\alpha_0, \omega_t} \rangle \, dt \nonumber \\
&\quad + \langle \nu_{\omega_0}, \nabla \phi \circ g_{\alpha_0, \omega_0} \cdot (\dot{\theta}_{\alpha'}(\omega) \partial_{\omega'} g_{\alpha_0, \omega_t} - \dot{\theta}_{\alpha_0}(\omega) \partial_{\omega'} g_{\alpha_0, \omega_0}) \rangle. \label{eq.dinner2}
\end{align}
Defining  the vector-valued measure $\mu_{t,\epsilon} = (\partial_{\omega'} g_{\alpha_0, \omega_t}) \nu_{\omega_0}$, the integrand is the difference of push-forwards of $\mu_{t,\epsilon}$ by $h_1 = g_{\alpha_0, \omega_t}$ and $h_0 = g_{\alpha_0, \omega_0}$. Applying Lemma~\ref{le.salvation} component-wise to the vector measure $\mu_{t,\epsilon}$ with the test function $\nabla \phi \in C^2(X)$, and using   that $\|\mu_{t,\epsilon}\|_{0^*} \le \|\partial_{\omega'} g_{\alpha_0, \theta_{\alpha,i}}\|_{C^0}  $, we obtain:
\begin{align}
| \dot{\theta}_{\alpha'}(\omega) \langle (g_{\alpha_0, \omega_t})_* \mu_{t,\epsilon} &- (g_{\alpha_0, \omega_0})_* \mu_{t,\epsilon}, \nabla \phi \rangle | \nonumber \\
&\le |\dot{\theta}_{\alpha'}(\omega)| \|\mu_{t,\epsilon}\|_{0^*} \|\nabla \phi\|_{C^1} \| g_{\alpha_0, \omega_t} - g_{\alpha_0, \omega_0} \|_{C^0} \nonumber \\
&\le |\dot{\theta}_{\alpha'}(\omega)| \|\partial_{\omega'} g\|_{C^0}   \|\phi\|_{C^2} \| g_{\alpha_0, \omega_t} - g_{\alpha_0, \omega_0} \|_{C^0}. \label{eq.dinner3}
\end{align}
For the second part of \eqref{eq.dinner2}, we have:
\begin{align}
|\langle \nu_{\omega_0}, \nabla \phi \circ g_{\alpha_0, \omega_0} \cdot &(\dot{\theta}_{\alpha'}(\omega) \partial_{\omega'} g_{\alpha_0, \omega_t} - \dot{\theta}_{\alpha_0}(\omega) \partial_{\omega'} g_{\alpha_0, \omega_0}) \rangle| \nonumber \\
&\le   \|\phi\|_{C^1} \| \dot{\theta}_{\alpha'}(\omega) \partial_{\omega'} g_{\alpha_0, \omega_t} - \dot{\theta}_{\alpha_0}(\omega) \partial_{\omega'} g_{\alpha_0, \omega_0} \|_{C^0}. \label{eq.dinner4}
\end{align}
Taking the supremum over $\|\phi\|_{C^3} \le 1$ and using \ref{as.full}--\ref{as.theta},  we conclude that $\|D_\epsilon(\omega)\|_{3^*}$   vanishes  uniformly in $\omega$  as $\epsilon\to 0$. This finishes the proof of the first item.

To prove the second item,  using that  $\bm\nu\in\mathbf P \cap\mathbf E_{2} '$ and 
   the expression obtained in~\eqref{eq.finally}, we deduce   for any $\phi \in C^3(X)$ 
\begin{align*}
|\langle  (\dot{\mathcal{K}}_{\alpha,i} \bm\nu)_\omega,\phi\rangle | \le    
&\; |\dot p_{\alpha,i}(\omega)| \cdot \| \nu_{\theta_{\alpha,i}(\omega)}\|_{0^*} \|\phi   \|_{C^0} \\
&+ p_{\alpha,i}(\omega)\cdot \| \nu'_{\theta_{\alpha,i}(\omega)} \dot{\theta}_{\alpha,i}(\omega)\|_{2^*} \|\phi \circ g_{\alpha, \theta_{\alpha,i}(\omega)}\|_{C^2} \\
&+ p_{\alpha,i}(\omega) \cdot  \| \nu_{\theta_{\alpha,i}(\omega)}\|_{0^*} \|D \phi \circ g_{\alpha, \theta_{\alpha,i}(\omega)} \bm\cdot \partial_\alpha g_{\alpha, \theta_{\alpha,i}(\omega)} \|_{C^0} \\
&+ p_{\alpha,i}(\omega) \cdot \| \nu_{\theta_{\alpha,i}(\omega)}\|_{0^*} \|D \phi \circ g_{\alpha, \theta_{\alpha,i}(\omega)} \bm\cdot \partial_{\omega'} g_{\alpha, \theta_{\alpha,i}(\omega)}  \dot{\theta}_{\alpha,i}(\omega) \|_{C^0} .\end{align*}
 Using  Lemma~\ref{le.aux}, we get
\begin{equation*}
 \|\phi \circ g_{\alpha, \theta_{\alpha,i} } \|_{C^2}\le  \underbrace{\max \big\{ 1, \diam(X)+ \|Dg_{\alpha, \theta_{\alpha,i} }\|_{C^0} + \|D^2 g_{\alpha, \theta_{\alpha,i} }\|_{C^0} + \|Dg_{\alpha, \theta_{\alpha,i} }\|_{C^0}^2 \big\}}_{C(g_{\alpha, \theta_{\alpha,i}})}\|\phi\|_{C^2},
\end{equation*}
By Cauchy-Schwarz inequality, we have
\begin{equation*}
\|D \phi \circ g_{\alpha, \theta_{\alpha,i}(\omega)} \bm\cdot \partial_\alpha g_{\alpha, \theta_{\alpha,i}(\omega)} \|_{C^0} \le \|\phi\|_{C^1}\| \partial_\alpha g_{\alpha, \theta_{\alpha,i}(\omega)} \|_{C^0}
\end{equation*}
and 
\begin{equation*}
\|D \phi \circ g_{\alpha, \theta_{\alpha,i}(\omega)} \bm\cdot \partial_{\omega'} g_{\alpha, \theta_{\alpha,i}(\omega)}  \dot{\theta}_{\alpha,i}(\omega) \|_{C^0} \le \|\phi\|_{C^1}\| \partial_{\omega'} g_{\alpha, \theta_{\alpha,i}(\omega)}  \|_{C^0}|  \dot{\theta}_{\alpha,i}(\omega)|.
\end{equation*}
Taking the supremum over $\|\phi\|_{C^3} \le 1$  and  over $\omega\in\Omega$, we get
\begin{align*}
\|\dot{\mathcal{K}}_{\alpha,i} \bm\nu\|_{3^*} \le    
&\; \|\dot p_{\alpha,i}\|_{C^0}  \| \bm\nu\|_{\mathrm L}  \\
&+ \|  p_{\alpha,i}\|_{C^0} \| \bm\nu'\|_2 C(g_{\alpha, \theta_{\alpha,i}}) \\
&+  \| p_{\alpha,i}\|_{C^0}  \| \bm\nu \|_{\mathrm L} \| \partial_\alpha g_{\alpha, \theta_{\alpha,i} } \|_{C^0} \\
&+  \|p_{\alpha,i}\|_{C^0} \| \bm\nu \|_{\mathrm L} \| \partial_{\omega'} g_{\alpha, \theta_{\alpha,i} }  \|_{C^0}\|  \dot{\theta}_{\alpha,i}\|_{C^0} .\end{align*}
Using~\ref{as.full}--\ref{as.p}, we get the  convergence of the series in $ \mathbf{E}_3^0$ uniformly in $\alpha$.
\end{proof}

\begin{corollary}\label{co.diffK}
The map $\alpha \mapsto \mathcal K_\alpha \bm\nu_{\alpha_0}$ is differentiable in $\mathbf{E}_3^0$ at $\alpha_0$.
\end{corollary}
\begin{proof}
Since Proposition~\ref{pr.differfixed} gives    $\bm\nu_{\alpha_0}\in \mathbf{E}_2'$,
the conclusion  follows from Proposition~\ref{pr.diffKi} below,  by standard calculus in the Banach space $\mathbf{E}_3^0$.
\end{proof}

\section{Base maps with inducing schemes}\label{se.inducing}

In this section, we prove Theorem~\ref{thm:inducedLR}. 
Consider the skew-product map 
\[
T_\alpha: \Omega \times X \to \Omega \times X, 
\] 
as in~\eqref{eq.skewmap}, 
and the induced map 
\begin{equation*}
\bar T_\alpha : \bar\Omega \times X \to \bar\Omega \times X,
\end{equation*}
as in~\eqref{eq.induceT}. 
%
Assume that \ref{as.full}--\ref{as.p} hold for the induced system $\bar T_\alpha$, and that \ref{as.induce}--\ref{as.unfold} also hold for the unfolding operator.

\subsection{Sectional transfer  operators}
In this subsection we establish a fundamental relation between the sectional transfer operator of the original skew-product map and that of the induced skew-product map. 
Recall that the sectional transfer  operator   for the original  system acts on a section $\boldsymbol{\nu}$ defined for $\eta_\alpha$-almost every  $\omega\in \Omega$  as
\begin{equation}\label{eq:K-alpha-final}
(\mathcal{K}_\alpha \boldsymbol{\nu})_\omega
= \sum_{\theta \in f_\alpha^{-1}(\{\omega\})}
\frac{\rho_\alpha(\theta)}{\rho_\alpha(\omega) |f'_\alpha(\theta)|} \,(g_{\alpha,\theta})_* \nu_\theta,
\end{equation}
 whilst the sectional transfer operator for the induced system  acts on a section $\bar{\boldsymbol{\nu}}$ for $\bar\eta_\alpha$-almost every  $\omega\in\bar\Omega$ as
\begin{equation}\label{eq:K-hat-alpha-final}
(\bar{\mathcal{K}}_\alpha \bar{\boldsymbol{\nu}})_\omega
= \sum_{\theta \in \bar{f}_\alpha^{-1}(\{\omega\})}
\frac{\bar{\rho}_\alpha(\theta)}{\bar{\rho}_\alpha(\omega) |\bar{f}_\alpha'(\theta)|} \,(\bar{g}_{\alpha,\theta})_* \bar{\nu}_\theta.
\end{equation}
For any $k\ge 1$ and \(\theta \in f_\alpha^{-k}(\{\omega\})\), define the accumulated weight along the \(k\)-step path as
\begin{equation}\label{eq.weights}
W_\alpha(\theta,k)
= \frac{\rho_\alpha(\theta)}{\rho_\alpha(f_\alpha^k(\theta)) \, |(f_\alpha^k)'(\theta)|}.
\end{equation}
Then, the action of \(\mathcal{K}_\alpha^k\) on the  section $\bm\nu$ is given for $\eta_\alpha$-almost every  $\omega\in \Omega$ by
 \begin{equation}\label{eq.operatork}
(\mathcal{K}_\alpha^k \boldsymbol{\nu})_\omega = \sum_{\theta \in f_\alpha^{-k}(\{\omega\})} 
W_\alpha(\theta,k) \, (g_{\alpha,\theta}^{(k)})_* \nu_\theta,
\end{equation}
with the map $g_{\alpha,\theta}^{(k)}$   as in~\eqref{eq:g(k)-def}.
The following lemma provides an explicit representation of the sectional transfer operator of the induced system in terms of the original operator.

\begin{lemma}\label{lem:indid}
For any section $\bar{\bm\nu}$ defined   on $\bar\Omega$
, we have
\[
\bar{\mathcal{K}}_\alpha \bar{\bm\nu}
=
\sum_{k=1}^{\infty}
\mathcal{K}_{\alpha}^k(
\bar{\bm\nu} \cdot \mathbf{1}_{\{\tau_\alpha = k\}}).
\]
\end{lemma}

\begin{proof}
Consider an $\bar\eta_\alpha$-generic  $\omega \in \bar{\Omega}$. 
Recalling that $\bar{f}_\alpha(\theta) = f_\alpha^{\tau_\alpha(\theta)}(\theta)$, 
we partition the set of preimages in~\eqref{eq:K-hat-alpha-final} according to the first return time $n = \tau_\alpha(\theta)$,
\[
\bar{f}_\alpha^{-1}(\{\omega\}) = \bigcup_{n=1}^{\infty} A_n(\omega), 
\]
with
\[
A_n(\omega) = \{ \theta \in \bar{\Omega} \mid  f_\alpha^n(\theta) = \omega\text{ and } \tau_\alpha(\theta) = n\}.
\]
Since the density $\rho_\alpha$ in $\bar\Omega$ coincides with $\bar\rho_\alpha$, up to a normalising factor, the induced weight satisfies 
\[
\frac{\bar{\rho}_\alpha(\theta)}{\bar{\rho}_\alpha(\omega) |\bar{f}_\alpha'(\theta)|} = W_\alpha(\theta,n),
\]
for each $\theta \in A_n(\omega)$, with $W_\alpha(\theta,n)$ 
 as in~\eqref{eq.weights}. Moreover, 
 $\bar{g}_{\alpha,\theta} = g_{\alpha,\theta}^{(n)}$. Thus, we may write the action of the induced operator as
\begin{equation}\label{eq:transfinduc}
(\bar{\mathcal{K}}_\alpha \bar{\bm\nu})_\omega = \sum_{n=1}^{\infty} \sum_{\theta \in A_n(\omega)} W_\alpha(\theta,n) \big(g_{\alpha,\theta}^{(n)}\big)_* \bar{\nu}_\theta.
\end{equation}
Next, consider the right-hand side of the desired equality,
\[
\mathcal{S}_\omega = \sum_{k=1}^{\infty} \bigl(\mathcal{K}_\alpha^k (\bar{\bm\nu} \cdot \mathbf{1}_{\{\tau_\alpha = k\}})\bigr)_\omega.
\]
By the formula for  the iterated   operator in~\eqref{eq.operatork}, we have for each $k\ge 1$
\[
\bigl(\mathcal{K}_\alpha^k (\bar{\bm\nu} \cdot \mathbf{1}_{\{\tau_\alpha = k\}})\bigr)_\omega 
= \sum_{\theta \in f_\alpha^{-k}(\{\omega\})} W_\alpha(\theta,k) \, \mathbf{1}_{\{\tau_\alpha = k\}}(\theta) \, \big(g_{\alpha,\theta}^{(k)}\big)_* \bar{\nu}_\theta.
\]
Because $\bar{\bm\nu}$ is supported on $\bar{\Omega}$, the term for a fixed $k$ is non-zero only for $\theta \in \bar{\Omega}$ with $f_\alpha^k(\theta) = \omega$ and $\tau_\alpha(\theta) = k$, that is, precisely the points in $A_k(\omega)$.  
Since the sets $A_k(\omega)$ are disjoint and cover all preimages under $\bar{f}_\alpha$, summing over $k$ counts each preimage exactly once, specifically in the term where $k = \tau_\alpha(\theta)$. Thus,
\[
\mathcal{S}_\omega = \sum_{k=1}^{\infty} \sum_{\theta \in A_k(\omega)} W_\alpha(\theta,k) \big(g_{\alpha,\theta}^{(k)}\big)_* \bar{\nu}_\theta.
\]
Comparing this to \eqref{eq:transfinduc}, we see that $\mathcal{S}_\omega = (\bar{\mathcal{K}}_\alpha \bar{\bm\nu})_\omega$, which concludes the proof.
\end{proof}

In the next result, we establish a pointwise property showing that the sectional transfer operator $\mathcal{K}_\alpha$ ``commutes"  with the infinite series defining an unfolded section.

\begin{lemma}\label{lem:pointwise}
Let $\bar{\boldsymbol{\nu}}$ be a section   on $\bar\Omega$. 
For every $\eta_\alpha$-almost every $\omega \in \Omega$, we have 
\[
\Big(
\mathcal{K}_{\alpha}
\sum_{k=0}^{\infty}
\mathcal{K}_{\alpha}^k
(\bar{\boldsymbol{\nu}}\cdot\mathbf{1}_{\{\tau_\alpha>k\}})
\Big)_\omega
=
\sum_{k=0}^{\infty}
\left(
\mathcal{K}_{\alpha}^{k+1}
(\bar{\boldsymbol{\nu}}\cdot\mathbf{1}_{\{\tau_\alpha>k\}})
\right)_\omega .
\]
\end{lemma}

\begin{proof}
Define the section
\[
\boldsymbol{\xi}
=
\sum_{k=0}^{\infty}
\mathcal{K}_{\alpha}^k
(\bar{\boldsymbol{\nu}}\cdot\mathbf{1}_{\{\tau_\alpha>k\}}),
\]
where the sum is understood  pointwise.  
According to~\eqref{eq.operatork}, for any  $\theta \in \Omega$,
the term
\[
\bigl(
\mathcal{K}_{\alpha}^k
(\bar{\boldsymbol{\nu}}\cdot\mathbf{1}_{\{\tau_\alpha>k\}})
\bigr)_\theta
\]
is non-zero only if there exists $\sigma \in \bar\Omega$ such that
$f_\alpha^k(\sigma)=\theta$ and $\tau_\alpha(\sigma)>k$.
For any $\sigma \in \bar\Omega$, the inequality $\tau_\alpha(\sigma)>k$ holds
for only finitely many $k$.
Consequently, for each $\theta$, the above series contains only finitely many
non-zero terms and therefore defines a finite sum of measures.
Now, by definition of the sectional transfer operator,
\[
\bigl(\mathcal{K}_\alpha \boldsymbol{\xi}\bigr)_\omega
=
\sum_{\theta \in f_\alpha^{-1}(\{\omega\})}
W_\alpha(\theta,1)\,
(g_{\alpha,\theta})_*
\boldsymbol{\xi}_\theta.
\]
Substituting the definition of $\boldsymbol{\xi}_\theta$, we obtain
\[
\bigl(\mathcal{K}_\alpha \boldsymbol{\xi}\bigr)_\omega
=
\sum_{\theta \in f_\alpha^{-1}(\{\omega\})}
W_\alpha(\theta,1)\,
(g_{\alpha,\theta})_*
\left(
\sum_{k=0}^{\infty}
\bigl(
\mathcal{K}_{\alpha}^k
(\bar{\boldsymbol{\nu}}\cdot\mathbf{1}_{\{\tau_\alpha>k\}})
\bigr)_\theta
\right).
\]
For each fixed $\theta$, the inner sum over $k$ is finite, as seen above.
Therefore, using only the linearity of the push-forward $(g_{\alpha,\theta})_*$,
we may rewrite the above expression as
\[
\bigl(\mathcal{K}_\alpha \boldsymbol{\xi}\bigr)_\omega
=
\sum_{k=0}^{\infty}
\sum_{\theta \in f_\alpha^{-1}(\{\omega\})}
W_\alpha(\theta,1)\,
(g_{\alpha,\theta})_*
\bigl(
\mathcal{K}_{\alpha}^k
(\bar{\boldsymbol{\nu}}\cdot\mathbf{1}_{\{\tau_\alpha>k\}})
\bigr)_\theta.
\]
Finally, by the definition of $\mathcal{K}_\alpha$, the inner sum is exactly
\[
\left(
\mathcal{K}_\alpha
\bigl[
\mathcal{K}_{\alpha}^k
(\bar{\boldsymbol{\nu}}\cdot\mathbf{1}_{\{\tau_\alpha>k\}})
\bigr]
\right)_\omega
=
\left(
\mathcal{K}_{\alpha}^{k+1}
(\bar{\boldsymbol{\nu}}\cdot\mathbf{1}_{\{\tau_\alpha>k\}})
\right)_\omega.
\]
Summing over $k$ yields the desired identity.
\end{proof}

\subsection{Unfolding operator}


Recall that the unfolding operator $\mathcal{U}_\alpha$, introduced in~\eqref{eq.unfoldopera}, allows one to reconstruct a full section $\boldsymbol{\nu}$ on $\Omega$ starting from a section $\bar{\boldsymbol{\nu}}$ defined on the inducing domain $\bar\Omega$, by
\begin{equation*}
\boldsymbol{\nu} = \mathcal{U}_\alpha\bar{\boldsymbol{\nu}} = \sum_{k=0}^{\infty} \mathcal{K}_{\alpha}^k \left( \bar{\boldsymbol{\nu}} \cdot \mathbf{1}_{\{\tau_\alpha > k\}} \right).
\end{equation*}

\begin{lemma}\label{le.unfoldfixed}
If \(\bar{\boldsymbol{\nu}}\) is a fixed point of the \(\bar{\mathcal{K}}_\alpha\), then   \( \mathcal{U}_\alpha\bar{\boldsymbol{\nu}}\) is a  fixed point of   \(\mathcal{K}_\alpha\).
\end{lemma}

\begin{proof}
First we show that $\mathcal U_\alpha$ maps a fixed point of \(\bar{\mathcal{K}}_\alpha\) to a fixed point of \({\mathcal{K}}_\alpha\).
Let \(\bar{\boldsymbol{\nu}}\) is a fixed point of the \(\bar{\mathcal{K}}_\alpha\) and set $\bm\nu =\mathcal{U}_\alpha(\bar{\boldsymbol{\nu}})$.
Applying  the   sectional transfer operator \(\mathcal{K}_{\alpha}\) to both sides of the identity  defining $  \mathcal{U}_\alpha(\bar{\boldsymbol{\nu}})$  and using Lemma~\ref{lem:pointwise}, we get
\begin{align}
\mathcal{K}_{\alpha} \boldsymbol{\nu} &= \mathcal{K}_{\alpha} \left( \sum_{k=0}^{\infty} \mathcal{K}_{\alpha}^k (\bar{\boldsymbol{\nu}} \cdot \mathbf{1}_{\{\tau_\alpha > k\}}) \right)\nonumber \\
&= \sum_{k=0}^{\infty} \mathcal{K}_{\alpha}^{k+1} (\bar{\boldsymbol{\nu}} \cdot \mathbf{1}_{\{\tau_\alpha > k\}})\nonumber \\
 \label{eq:shiftedsum}
&= \sum_{k=1}^{\infty} \mathcal{K}_{\alpha}^{k} (\bar{\boldsymbol{\nu}} \cdot \mathbf{1}_{\{\tau_\alpha > k-1\}}).
\end{align}
Now, observe that
\[
\mathbf{1}_{\{\tau_\alpha > k-1\}} = \mathbf{1}_{\{\tau_\alpha = k\}} + \mathbf{1}_{\{\tau_\alpha > k\}}.
\]
Substituting this into \eqref{eq:shiftedsum}, we obtain
\begin{align*}
\mathcal{K}_{\alpha} \boldsymbol{\nu} &= \sum_{k=1}^{\infty} \mathcal{K}_{\alpha}^{k} \left( \bar{\boldsymbol{\nu}} \cdot (\mathbf{1}_{\{\tau_\alpha = k\}} + \mathbf{1}_{\{\tau_\alpha > k\}}) \right) \\
&= \underbrace{\sum_{k=1}^{\infty} \mathcal{K}_{\alpha}^{k} (\bar{\boldsymbol{\nu}} \cdot \mathbf{1}_{\{\tau_\alpha = k\}})}_{\text{(A)}} + \underbrace{\sum_{k=1}^{\infty} \mathcal{K}_{\alpha}^{k} (\bar{\boldsymbol{\nu}} \cdot \mathbf{1}_{\{\tau_\alpha > k\}})}_{\text{(B)}}.
\end{align*}
Since \(\bar{\boldsymbol{\nu}}\) is the fixed point of   \(\bar{\mathcal{K}}_\alpha\), it follows from Lemma~\ref{lem:indid} that
\[
\text{(A)} = \bar{\mathcal{K}}_\alpha \bar{\boldsymbol{\nu}}=\bar{\bm\nu}.
\]
On the other hand,
\[
\text{(B)} = \mathcal U_\alpha\bar{\bm\nu}- \mathcal{K}_\alpha^0(\bar{\boldsymbol{\nu}} \cdot \mathbf{1}_{\{\tau_\alpha > 0\}}) =
 {\bm\nu}-  \bar{\boldsymbol{\nu}} \cdot \mathbf{1}_{\{\tau_\alpha > 0\}}
= {\bm\nu}-  \bar{\boldsymbol{\nu}}  
\]
Combining the results for (A) and (B), we obtain
\[
\mathcal{K}_\alpha \boldsymbol{\nu} = \bar{\boldsymbol{\nu}} + (\boldsymbol{\nu} - \bar{\boldsymbol{\nu}}) = \boldsymbol{\nu}.
\]
This proves that $\boldsymbol{\nu}$ is a fixed point of $\mathcal{K}_\alpha$.
\end{proof}

\begin{lemma}\label{lem:generalunfold}
If $\bar{\boldsymbol{\nu}} \in \bar{\mathbf{P}}$, then $ \mathcal{U}_\alpha \bar{\boldsymbol{\nu}}\in\mathbf P$.
\end{lemma}

\begin{proof}
We prove that for $\eta_\alpha$-almost every $\omega \in \Omega$, the fiber measure $ {\nu}_\omega$ has total mass $1$.  
By the formula for  the iterated sectional transfer operator in~\eqref{eq.operatork},
\[
 {\nu}_\omega = \sum_{k=0}^{\infty} \sum_{\theta \in f_\alpha^{-k}(\{\omega\}) \cap \bar{\Omega}} 
W_\alpha(\theta,k) \, (g_{\alpha,\theta}^{(k)})_* \bigl( \bar{\nu}_\theta \mathbf{1}_{\{\tau_\alpha > k\}}(\theta) \bigr),
\]
where the weights are given by \eqref{eq.weights}. 
%
Since $\bar{\nu}_\theta$ is a probability measure and the push-forward $(g_{\alpha,\theta}^{(k)})_*$ preserves mass, we get
\begin{equation}\label{eq:fiber_mass}
 {\nu}_\omega(X) = \frac{1}{\rho_\alpha(\omega)} \sum_{k=0}^\infty \sum_{\theta \in f_\alpha^{-k}(\omega) \cap \bar{\Omega}} \frac{\rho_\alpha(\theta)}{|(f_\alpha^k)'(\theta)|} \mathbf{1}_{\{\tau_\alpha > k\}}(\theta).
\end{equation}
Now, recall that
 \begin{equation}\label{eq.tower}
\eta_\alpha = \frac{1}{ {E}} \sum_{k=0}^\infty (f_\alpha^k)_* \bigl( \bar{\eta}_\alpha \vert_{\{\tau_\alpha > k\}} \bigr), \qquad
 {E} = \int_{\bar{\Omega}} \tau_\alpha \, d\bar{\eta}_\alpha.
\end{equation}
Writing $d\eta_\alpha = \rho_\alpha \, dm$ and $d\bar{\eta}_\alpha = \bar{\rho}_\alpha \, dm$, the density on the induced base is related to the global density by $\bar{\rho}_\alpha(\theta) =  {E} \rho_\alpha(\theta)$, for $\theta \in \bar{\Omega}$.  
For any Borel set $A \subset \Omega$, the push-forward integral reads
\[
(f_\alpha^k)_* \bigl( \bar{\eta}_\alpha \vert_{\{\tau_\alpha > k\}} \bigr)(A)
= \int_{(f_\alpha^k)^{-1}(A) \cap \{\tau_\alpha > k\}} \bar{\rho}_\alpha(\theta) \, dm(\theta).
\]
Applying the change of variables formula under the local diffeomorphism $f_\alpha^k$,
\[
\int_{(f_\alpha^k)^{-1}(A) \cap \{\tau_\alpha > k\}} \bar{\rho}_\alpha(\theta) \, dm(\theta)
= \int_A \sum_{\theta \in f_\alpha^{-k}(\omega) \cap \bar{\Omega}} \frac{\bar{\rho}_\alpha(\theta)}{|(f_\alpha^k)'(\theta)|} \mathbf{1}_{\{\tau_\alpha > k\}}(\theta) \, dm(\omega).
\]
Substituting into~\eqref{eq.tower} and using $\bar{\rho}_\alpha(\theta) =  {E} \rho_\alpha(\theta)$, we obtain
\[
\eta_\alpha(A) = \frac{1}{ {E}} \sum_{k=0}^\infty \int_A \sum_{\theta \in f_\alpha^{-k}(\omega) \cap \bar{\Omega}} \frac{ {E} \rho_\alpha(\theta)}{|(f_\alpha^k)'(\theta)|} \mathbf{1}_{\{\tau_\alpha > k\}}(\theta) \, dm(\omega).
\]
Hence,
\[
\int_A \rho_\alpha(\omega) \, dm(\omega) = \int_A \sum_{k=0}^\infty \sum_{\theta \in f_\alpha^{-k}(\omega) \cap \bar{\Omega}} \frac{\rho_\alpha(\theta)}{|(f_\alpha^k)'(\theta)|} \mathbf{1}_{\{\tau_\alpha > k\}}(\theta) \, dm(\omega).
\]
By the Radon–Nikodym theorem, the integrands must coincide  $\nu_\alpha$-almost everywhere:
\[
\rho_\alpha(\omega) = \sum_{k=0}^\infty \sum_{\theta \in f_\alpha^{-k}(\omega) \cap \bar{\Omega}} \frac{\rho_\alpha(\theta)}{|(f_\alpha^k)'(\theta)|} \mathbf{1}_{\{\tau_\alpha > k\}}(\theta).
\]
Finally, substituting this  into \eqref{eq:fiber_mass} gives ${\nu}_\omega(X) =1$, for $\omega$ almost every $\omega\in\Omega$.
%
\end{proof}

We are now in conditions to conclude the proof of Theorem~\ref{thm:inducedLR}.  Since conditions~\ref{as.full}--\ref{as.p} for $\bar T_\alpha$, together with~\ref{as.induce}, 
imply that~\ref{as.base} holds for $\bar T_\alpha$, we  apply Theorem~\ref{th.mainD} to deduce:
\begin{itemize}
\item \em $\bar{\mathcal K}_\alpha$
has a  unique fixed point $\bar{\bm\nu}_\alpha$ in $\bar{\mathbf P}_0$;
\item the map 
 $\alpha\mapsto\bar{\bm\nu}_\alpha$
 is differentiable in $\bar{\mathbf E}_3^0$.
\end{itemize}
It follows from Lemma~\ref{le.unfoldfixed} and Lemma~\ref{lem:generalunfold} that
\begin{equation}\label{unfolfix}
\bm\nu_\alpha=\mathcal U_\alpha(\bar{\bm\nu}_\alpha).
\end{equation}
is a fixed point for $\mathcal K_\alpha$ in $\mathbf P$. In the next result we establish the last conclusion of Theorem~\ref{thm:inducedLR}.

\begin{proposition}\label{thm:unfoldresp}
The map
$
\alpha \mapsto \bm{\nu}_\alpha
$
is differentiable in $\mathbf{L}_3^0$ at $\alpha_0$.
\end{proposition}

\begin{proof}
 By~\ref{as.unfoldc}   the map $\alpha \mapsto \mathcal{U}_\alpha(\bar{\bm{\nu}}_{\alpha_0})$ is differentiable in $\mathbf{L}_3^0 $ at~$\alpha_0$. Let $\dot{\mathcal{U}}_{\alpha_0}(\bar{\bm{\nu}}_{\alpha_0})\in \mathbf{L}_3^0$ denote this derivative.
 We are going to show that 
\begin{equation}\label{eq.candidate}
\dot{\bm{\nu}}_{\alpha_0} = \dot{\mathcal{U}}_{\alpha_0}(\bar{\bm{\nu}}_{\alpha_0}) + \mathcal{U}_{\alpha_0}(\dot{\bar{\bm{\nu}}}_{\alpha_0})
\end{equation}
 is the derivative of $\alpha \mapsto \bm{\nu}_\alpha$ at $\alpha_0$. 
Since~\eqref{unfolfix} holds for every $\alpha$, we may write
\begin{align*}
\frac{\bm{\nu}_\alpha - \bm{\nu}_{\alpha_0}}{\alpha - \alpha_0} 
&= 
  \frac{\mathcal{U}_\alpha(\bar{\bm{\nu}}_\alpha) - \mathcal{U}_{\alpha_0}(\bar{\bm{\nu}}_{\alpha_0})}{\alpha - \alpha_0}\\
&=
\frac{\mathcal{U}_\alpha(\bar{\bm{\nu}}_\alpha) - \mathcal{U}_\alpha(\bar{\bm{\nu}}_{\alpha_0}) + \mathcal{U}_\alpha(\bar{\bm{\nu}}_{\alpha_0}) - \mathcal{U}_{\alpha_0}(\bar{\bm{\nu}}_{\alpha_0})}{\alpha - \alpha_0} \\
&= \mathcal{U}_\alpha \left( \frac{\bar{\bm{\nu}}_\alpha - \bar{\bm{\nu}}_{\alpha_0}}{\alpha - \alpha_0} \right) + \frac{\mathcal{U}_\alpha(\bar{\bm{\nu}}_{\alpha_0}) - \mathcal{U}_{\alpha_0}(\bar{\bm{\nu}}_{\alpha_0})}{\alpha - \alpha_0}.
\end{align*}
To show differentiability, we estimate the distance to the candidate derivative in~\eqref{eq.candidate}.  By the triangle inequality,
\begin{align*}
\left\| \frac{\bm{\nu}_\alpha - \bm{\nu}_{\alpha_0}}{\alpha - \alpha_0} - \dot{\bm{\nu}}_{\alpha_0} \right\|_{\mathbf{L}_3} &\le \left\| \mathcal{U}_\alpha \left( \frac{\bar{\bm{\nu}}_\alpha - \bar{\bm{\nu}}_{\alpha_0}}{\alpha - \alpha_0} \right) - \mathcal{U}_{\alpha_0}(\dot{\bar{\bm{\nu}}}_{\alpha_0}) \right\|_{\mathbf{L}_3} \\
&\quad + \left\| \frac{\mathcal{U}_\alpha(\bar{\bm{\nu}}_{\alpha_0}) - \mathcal{U}_{\alpha_0}(\bar{\bm{\nu}}_{\alpha_0})}{\alpha - \alpha_0} - \dot{\mathcal{U}}_{\alpha_0}(\bar{\bm{\nu}}_{\alpha_0}) \right\|_{\mathbf{L}_3}.
\end{align*}
We split the first term on the right-hand side
\begin{align*}
\left\| \mathcal{U}_\alpha \left( \frac{\bar{\bm{\nu}}_\alpha - \bar{\bm{\nu}}_{\alpha_0}}{\alpha - \alpha_0} \right) - \mathcal{U}_{\alpha_0}(\dot{\bar{\bm{\nu}}}_{\alpha_0}) \right\|_{\mathbf{L}_3} &\le \left\| \mathcal{U}_\alpha \left( \frac{\bar{\bm{\nu}}_\alpha - \bar{\bm{\nu}}_{\alpha_0}}{\alpha - \alpha_0} - \dot{\bar{\bm{\nu}}}_{\alpha_0} \right) \right\|_{\mathbf{L}_3} \\
&\quad + \left\| \mathcal{U}_\alpha(\dot{\bar{\bm{\nu}}}_{\alpha_0}) - \mathcal{U}_{\alpha_0}(\dot{\bar{\bm{\nu}}}_{\alpha_0}) \right\|_{\mathbf{L}_3}.
\end{align*}
Combining these and using~\ref{as.unfolda}, we get
\begin{align*}
\left\| \frac{\bm{\nu}_\alpha - \bm{\nu}_{\alpha_0}}{\alpha - \alpha_0} - \dot{\bm{\nu}}_{\alpha_0} \right\|_{\mathbf{L}_3} &\le C \left\| \frac{\bar{\bm{\nu}}_\alpha - \bar{\bm{\nu}}_{\alpha_0}}{\alpha - \alpha_0} - \dot{\bar{\bm{\nu}}}_{\alpha_0} \right\|_3 \\
&\quad + \left\| \mathcal{U}_\alpha(\dot{\bar{\bm{\nu}}}_{\alpha_0}) - \mathcal{U}_{\alpha_0}(\dot{\bar{\bm{\nu}}}_{\alpha_0}) \right\|_{\mathbf{L}_3} \\
&\quad + \left\| \frac{\mathcal{U}_\alpha(\bar{\bm{\nu}}_{\alpha_0}) - \mathcal{U}_{\alpha_0}(\bar{\bm{\nu}}_{\alpha_0})}{\alpha - \alpha_0} - \dot{\mathcal{U}}_{\alpha_0}(\bar{\bm{\nu}}_{\alpha_0}) \right\|_{\mathbf{L}_3}.
\end{align*}
Taking the limit as $\alpha \to \alpha_0$, the first term vanishes because 
$\dot{\bar{\bm{\nu}}}_{\alpha_0}$ is the derivative of the map 
$\alpha \mapsto \bar{\bm{\nu}}_\alpha$ in $\bar{\mathbf{E}}_3^0$ at $\alpha_0$; 
the second term vanishes due to the continuity of the map 
$\alpha \mapsto \mathcal{U}_\alpha(\dot{\bar{\bm{\nu}}}_{\alpha_0})$ in $\mathbf{L}_3^0$, as stated in~\ref{as.unfoldb}; 
the third term vanishes because $\dot{\mathcal{U}}_{\alpha_0}(\bar{\bm{\nu}}_{\alpha_0})$ is the derivative of the map 
$\alpha \mapsto \mathcal{U}_\alpha(\bar{\bm{\nu}}_{\alpha_0})$ at $\alpha_0$. 
Hence, the difference quotient converges to~$\dot{\bm{\nu}}_{\alpha_0}$ in $\mathbf{L}_3^0$, completing the proof.
\end{proof}

\subsection{Application: Solenoid  with intermittency}  \label{sub.solenoidcheck}

Here we show that the solenoidal maps with intermittency $\{T_\alpha\}_{\alpha\in J}$, 
introduced in Subsection~\ref{sub.solenoid}, admit inducing schemes for which the conditions of Theorem~\ref{thm:inducedLR} are satisfied.

\subsubsection{Induced map}
Here we introduce the induced maps associated with and verify assumptions~\ref{as.full}--\ref{as.p} (for the induced maps) and 
\ref{as.induce}.
Let $\Omega=\Omega_0\cup \Omega_1$ be the $m$  mod 0 partition into the smoothness domains of the base map~$f_\alpha$.
For any $\alpha\in J$, consider the two  branches of $f_\alpha$ given by
$$f_{\alpha,0}(\omega)=\omega(1+2^{\alpha}\omega^{\alpha}),\quad \omega\in \Omega_0$$ and 
$$f_{\alpha,1}(\omega)= 2\omega-1, \quad \omega\in \Omega_1.$$
Given any $n\in\mathbb N_0$, let $i$ be a word $10^n$ (a single 1 followed by $n$ 0's). Let also $\ell(i)=n+1$ be the length of such a word $i$. Set $\mathbb I$ as the (countable) set of those words $i$. Define for each  $\omega\in[0,1]$ 
 \begin{equation}\label{eq:path}
 \Theta_{\alpha,i}(\omega)=f_{\alpha, 1}^{-1}\circ\underbrace{ f_{\alpha,0}^{-1}\circ\cdots\circ f_{\alpha, 0}^{-1}}_{n\text{-times}}(\omega).
 \end{equation}
The cylinder sets $[i]_\alpha=\Theta_{\alpha,i}(\bar\Omega)$ give the $m$ mod 0 partition of $\bar\Omega=[1/2,1]$ into smoothness domains of  the full branch induced map
 $${\bar f_\alpha}:\bar\Omega\to\bar\Omega,$$
 defined as
\begin{equation}\label{eq:inducedmap}
\bar f_{\alpha}\vert_{[i]_\alpha}={f^{\ell(i)}_\alpha|}_{[i]_\alpha}.
\end{equation}
Note that $\bar\theta_{\alpha,i}=\Theta_{\alpha,i}|_{\bar\Omega}$ is the inverse branch of $\bar f_\alpha|_{[i]_\alpha}$.
Now we check \ref{as.full}--\ref{as.p} for the family of induced maps.

\medskip

\noindent
\ref{as.full}: this follows from the fact that $\bar f_{\alpha}$, as defined above, is a full-branch, uniformly expanding map.

 \medskip

\noindent\ref{as.theta}: this corresponds to $C^1$ norm of the transfer operator associated with $\bar f_{\alpha}$, evaluated at the constant function $1$. It is is well known that this $C^1$-norm is uniformly bounded independent of $\alpha$, in a small neighbourhood of $\alpha_0$. This is because in a small neighbourhood of $\alpha_0$, the expansion is also uniform in $\alpha$, $\bar f_{\alpha}$ is $C^2$ with $C^2$ extensions, and bounded distortion with a bound independent of $\alpha$. 

  \medskip

\noindent\ref{as.g}: this follows from the contraction that we impose on fibres. 

 \medskip

\noindent \ref{as.p}: 
 since we are dealing with full branch systems that $\bar\rho_\alpha$, the invariant density of $\bar f_\alpha$, is uniformly bounded above and below, then to verify 
$$
\sum_{i\in\mathbb I} \sup_{\alpha\in J}\|  \bar p_{\alpha,i} \|_{C^1}<C,
$$
it is enough to verify 
$$\sum_{i\in\mathbb I} \sup_{\alpha\in J}\|\Theta_{\alpha,i}\|_{C^1}<\infty,$$ which follows as in \ref{as.full}.
Now, to verify the rest of the assumption, we need to verify 
\begin{equation}\label{eq:check}
\sum_{i\in\mathbb I} \sup_{\alpha\in J}\| \dot{\bar p}_{\alpha,i} \|_{C^0}\le C_p.
\end{equation}
Notice that $\|\partial_{\alpha}{\bar\rho_\alpha|}_{\alpha_0}\|_{C_0}<C$ by   \cite[Theorem 3.2]{BS16} and since the density $\bar\rho_\alpha$ is uniformly  bounded above and below, verifying \eqref{eq:check}, reduces to showing
$$ 
\sum_{i\in\mathbb I} \sup_{\alpha\in J}\|\partial_\alpha \Theta_{\alpha,i} \|_{C^0}<\infty,$$
which follows by   \cite[Lemma 5.1]{BS16}.

\medskip
\noindent\ref{as.induce}:
It is well known that $\bar f_{\alpha}$ satisfies~\ref{as.induce} with the first return time function $\tau_\alpha:\bar\Omega\to\mathbb R$ satisfying
 \begin{equation*}
m\{\tau_\alpha> k\}\approx k^{-1/\alpha}.
\end{equation*}
Therefore, integrable with respect to the Lebesgue measure on $\bar\Omega$.

\subsubsection{Unfolding operator}
Here we verify the three items of assumption~\ref{as.unfold}.
Using~\eqref{eq.measures} and \cite[Equation (7)]{BS16}, we obtain the following formula for the density~$\rho_\alpha$ of the absolutely continuous invariant probability measure $ \eta_\alpha$ for $ f_\alpha$:
\begin{equation}\label{eq:undensity}
\rho_\alpha=\frac1E \sum_{i\in \mathbb I} (\bar \rho_\alpha\circ \Theta_{\alpha,i})\Theta_{\alpha,i}',
\end{equation}
where $E>0$ is the normalising constant in~\eqref{eq.measures}. Noting that   $\bar\rho_\alpha =E\rho_\alpha|_{\bar\Omega}$, it follows that
\begin{equation*}
\sum_{i\in \mathbb I}\frac{  \rho_\alpha\circ\Theta_{\alpha,i}}{\rho_\alpha } \Theta_{\alpha,i}' =1.
\end{equation*}
%
In this case, the unfolding operator can be written for all section $\bar{\bm\nu}$   on $\bar\Omega$
and all  $\omega \in \Omega$ as
\begin{equation}\label{eq.unfodlingnew}
(\mathcal{U}_\alpha \bar{\bm\nu})_\omega
= \sum_{i\in \mathbb I} 
\frac{\rho_\alpha(\Theta_{\alpha,i}(\omega))}{\rho_\alpha(\omega)} \Theta'_{\alpha,i}(\omega)
\,(g^{\ell(i)}_{\Theta_{\alpha,i}(\omega)})_*\bar{\nu}_{\Theta_{\alpha,i}(\omega)}.
\end{equation}
In the next result we establish \ref{as.unfolda}. In fact, the proof shows that the constant can be taken to be $C=1$.

\begin{lemma}\label{lem:undiff1}
$\mathcal U_\alpha$ defines a uniformly bounded linear operator from $\bar{\mathbf E}_{3}^0$ to $\bm{\mathbf L}_3^0$.
\end{lemma}
\begin{proof}
Given any $\bar{\bm\nu}\in\bar{\mathbf E}_{3}^0$, we may write
for all $\omega\in\Omega$ and $\phi\in C^3_0(X)^*$ 
\begin{equation*}
\begin{split}
\langle (\mathcal U_\alpha\bar{\bm\nu})_\omega, \phi\rangle&=  \sum_{i\in \mathbb I} 
\frac{\rho_\alpha(\Theta_{\alpha,i}(\omega))}{\rho_\alpha(\omega)} \Theta'_{\alpha,i}(\omega)
\,\langle  \bar \nu_{\Theta_{\alpha,j}(\omega)},\phi\circ g^{\ell(i)}_{\Theta_{\alpha,i}(\omega)}\rangle \\
&\le\sum_{j=1}^{\infty}\frac{\rho_\alpha(\Theta_{\alpha,i}(\omega))}{\rho_\alpha(\omega)} \Theta'_{\alpha,i}(\omega)
\, \|\bar \nu_{\Theta_{\alpha,j}(\omega)} \|_3\|\phi\|_{C^3}.
\end{split}
\end{equation*}
Taking the supremum over all test functions $\phi \in C^3_0(X)^*$ with
$\|\phi\|_{C^3}\le 1$, we obtain
\begin{equation*}
\| (\mathcal U_\alpha\bar{\bm\nu})_\omega\|_3\le 
 \|\bar \nu_{\Theta_{\alpha,j}(\omega)} \|_3,
\end{equation*}
for all $\omega\in\Omega$. Since $\bar{\bm\nu}\in\bar{\mathbf E}_{3}^0$, we have that $  
 \|\bar \nu_{\Theta_{\alpha,j}(\omega)} \|_3\le\|\bar{\bm\nu}\|_3$, and so $ 
 \|(\mathcal U_\alpha\bar{\bm\nu})_\omega \|_3\le\|\bar{\bm\nu}\|_3$ for all $\omega\in\Omega$. This implies that $\mathcal U_\alpha\bar{\bm\nu}\in \mathbf L_3^0$ and $\|\mathcal U_\alpha\bar{\bm\nu}\|_{\mathbf L_3}\le\|\bar{\bm\nu}\|_3$, thus showing that $ \mathcal U_\alpha$ defines a bounded linear operator from $\bar{\mathbf E}_{3}^0$ to $\bm{\mathbf L}_3^0$ with $ \|\mathcal U_\alpha\|\le 1$.
\end{proof}

Before verifying the second and third conditions in~\ref{as.unfold}, we present   some preparatory material, essentially collecting a few facts from~\cite{BS16}.
Given $\alpha_0 \in (0,1)$, let $r>1$ be such that $r<1/\alpha_0$. It follows from~\eqref{eq:behave} that
   \begin{equation}\label{eq.density}
\rho_{\alpha_0}\in L^r(m).
\end{equation}
Let $s>1$ be such that 
  \begin{equation}\label{eq.conjugate}
  \frac1r+\frac1s=1.
  \end{equation}
Fix $\varepsilon>0$ so that $\varepsilon s<1$ and set
 $$J_{\alpha_0}=(\alpha_0-\epsilon, \alpha_0+\epsilon).$$ 
Define  $\mathcal B$ as the set of continuous functions  $\phi: (0,1]\to\mathbb R$ such that $|\omega ^\varepsilon  \phi(\omega)|$ is bounded, and define
\begin{equation}\label{eq.norm}
\|\phi\|_{\mathcal B}=\sup_{\omega\in (0,1]}|\omega^\varepsilon  \phi(\omega)|.
\end{equation}
It is straightforward to check that $ \|\cdot\|_{\varepsilon}$ is a norm on the vector space  $\mathcal B$ and   $\mathcal B$ endowed with this norm is   a Banach  space.

\begin{lemma}\label{lem:sB}
There exists $C>0$ such that for all $\phi\in \mathcal B$,  
$$\int |\phi| d\eta_{\alpha_0}\le C\|\phi\|_{\mathcal B}.$$
\end{lemma}
\begin{proof}
Consider $s>1$ as in~\eqref{eq.conjugate}. 
Given $\phi\in\mathcal B$, we may write
\begin{equation}\label{eq.ls}
\int_{[0,1]} |\phi(\omega)|^s dm(\omega)=\int_{[0,1]} \frac{|\omega^\epsilon\phi(\omega)|^s}{\omega^{ \epsilon s}}dm(\omega)\le \|\phi\|_{\mathcal B}^s\int_{[0,1]}\frac{1}{\omega^{ \epsilon s}}dm(\omega).
\end{equation}
 Since we take  $\epsilon s<1$, this last integral is finite. On the other hand, using~\eqref{eq.density}, \eqref{eq.conjugate} and Hölder inequality, we get
 $$\int |\phi| d\eta_{\alpha_0} =\int |\phi| \rho_{\alpha_0} dm\le 
\left( \int |\phi|^s dm\right)^{1/s}\left( \int \rho_{\alpha_0}^r dm\right)^{1/r},
 $$
which together with~\eqref{eq.ls}, yields the conclusion.
\end{proof}

Set 
$$\alpha_*=\alpha_0-\frac12 \varepsilon\qand \gamma=\alpha_0+\frac12\epsilon.$$
%
%
For each $i\in\mathbb I$, define $h_{\alpha,i}=\rho_\alpha\circ\Theta_{\alpha,i}$. Notice that by \cite[Theorem 3.1]{BS16} and the chain rule, the derivatives $\dot{\rho}_{\alpha}$ and $\dot{h}_{\alpha,i}$ exist in the following sense:
for $\alpha\in(\alpha_{0}-\varepsilon/4,\alpha_{0}+\varepsilon/4)$ 

\begin{equation}\label{eq:der1}
\lim_{\alpha\to\alpha_0}\sup_{\omega\in(0,1]}\omega^{\alpha_{*}+\varepsilon}\Big|\frac{1}{\alpha-\alpha_0}(h_{\alpha,j}-h_{\alpha_0,j})-\dot{h}_{\alpha_0,j}\Big| =0
\end{equation}
and 
\begin{equation}\label{eq:der2}
\lim_{\alpha\to\alpha_0}\sup_{\omega\in(0,1]}\omega^{\alpha_{*}+\varepsilon}\Big|\frac{1}{\alpha-\alpha_0}(\rho_\alpha(\omega)-\rho_{\alpha_0}(\omega))-\dot{\rho}_{\alpha_0}(\omega)\Big| =0.
\end{equation}
\begin{lemma}\label{lem:qoutient}
For all $\alpha\in(\alpha_{0}-\varepsilon/4,\alpha_{0}+\varepsilon/4)$ and $i\in \mathbb I$, we have
$$\alpha\longmapsto \frac{{\rho}_\alpha ( \Theta_{\alpha,i}(\omega))}{ \rho_\alpha(\omega)} \Theta'_{\alpha, i}(\omega)$$
 is differentiable  in $\mathcal B$ at $\alpha_0$. 
\end{lemma}
\begin{proof}
We first notice that for all $\phi\in\mathcal B$, we have
\begin{equation}\label{eq:normsplit}
\|\phi\|_{\mathcal B}=\sup_{\omega\in(0,1]}|\omega^{\alpha_{*}+\varepsilon}{\omega^{-\alpha_{*}}}\phi(\omega)|.
\end{equation}
Thus, recalling \eqref{eq:behave}, we have
$$\Big|\frac{1}{\omega^{\alpha_*}} \left( \frac{1}{\rho_\alpha(\omega)}-\frac{1}{\rho_{\alpha_0}(\omega)}\right)\Big|\le C\omega^{\alpha+\alpha_0-\alpha_*}| \rho_\alpha(\omega)- \rho_{\alpha_0}(\omega)|\le C\omega^{\alpha_0+\frac{\varepsilon}{4}}| \rho_\alpha(\omega)- \rho_{\alpha_0}(\omega)|.$$
Therefore, by the result of \cite{BS16} the right hand side of the above inequality also converges to $0$ as $\alpha\to\alpha_0$. Using this, together with \eqref{eq:der1}, \eqref{eq:der2} and  \eqref{eq:normsplit} shows that
$$\lim_{\alpha\to\alpha_0}\Big\|\frac{1}{\alpha-\alpha_0}\left(\frac{h_{\alpha,i}}{\rho_\alpha}- \frac{h_{\alpha_0,i}}{\rho_{\alpha_0}}\right)
-\Bigg(\frac{\dot{h}_{\alpha_0,i}\rho_{\alpha_0} -\dot{\rho}_{\alpha_0}h_{\alpha_0,i}}{\rho_{\alpha_0}^2}\Bigg)\Big\|_{\mathcal B}=0.$$
 Noting that 
 \begin{equation*}\label{eq.thetaconv}
\|\Theta_{\alpha,i}- \Theta_{\alpha_0,i}\|_{C^0}\to 0\qand \|\dot\Theta_{\alpha,i}- \dot\Theta_{\alpha_0,i}\|_{C^0}\to 0,\quad\text{as $\alpha\to\alpha_0$}
\end{equation*}
and
\begin{equation*}\label{eq.supteta}
\sup_{i\in\mathbb I} \|\Theta_{\alpha,i}\|_{C^0}<\infty.
\end{equation*}
completes the proof of the second part of the lemma.
\end{proof}

The next result establishes~\ref{as.unfoldc}.

\begin{lemma}\label{lema:undiff3}
The map $\alpha\mapsto \mathcal U_\alpha({\bar{\bm \nu}}_{\alpha_0})$ is differentiable  in $\mathbf L_3^0$ at $\alpha_0$.
\end{lemma}
\begin{proof}
We proceed in two steps. First notice that
 $$\alpha\longmapsto \frac{\rho_\alpha \circ \Theta_{\alpha,i}}{\rho_\alpha}
\Theta'_{\alpha,i} \cdot  (g^j_{\Theta_{\alpha,i}})_*\bar \nu_{\Theta_{\alpha,i}}$$
 is differentiable  in $\mathbf L_3^0$ at $\alpha_0$. This follows from Lemma \ref{lem:qoutient}, the fact that $\alpha\mapsto\bar{\bm\nu}_\alpha$ is differentiable at $\alpha_0$ in the the $\bar{\mathbf{E}}_3^0$  and Lemma \ref{lem:sB}. It remains to show that the norm of the derivatives are summable. 
 We have 
 \begin{equation*}
 \begin{split}
 \sum_{i\in\mathbb I}\int_\Omega \partial_\alpha&\left(\frac{\rho_\alpha(\Theta_{\alpha,i})}{\rho_\alpha(\omega) |(f_\alpha^j)'(\Theta_{\alpha,i})|}\langle  \bar \nu_{\Theta_{\alpha,i}},\varphi\circ g^j_{\Theta_{\alpha,i}}\rangle\right) d\eta_{\alpha_0}(\omega)\\ 
 &\le \underbrace{\sum_{i\in\mathbb I}\int_\Omega \partial_\alpha\left(\frac{\rho_\alpha(\Theta_{\alpha,i})}{\rho_\alpha(\omega) |(f_\alpha^j)'(\Theta_{\alpha,i})|}\right)\cdot\langle  \bar \nu_{\Theta_{\alpha,i}},\varphi\circ g^j_{\Theta_{\alpha,i}}\rangle d\eta_{\alpha_0}(\omega)}_{\text{(I)}}\\
 &\quad+ \underbrace{\sum_{i\in\mathbb I}\int_\Omega\frac{\rho_\alpha(\Theta_{\alpha,i})}{\rho_\alpha(\omega) |(f_\alpha^j)'(\Theta_{\alpha,i})|} \partial_\alpha\left(\langle  \bar \nu_{\Theta_{\alpha,i}},\varphi\circ g^j_{\Theta_{\alpha,i}}\rangle\right) d\eta_{\alpha_0}(\omega)}_{\text{(II)}}.
 \end{split}
 \end{equation*}
 We have
 \begin{equation*}
 \begin{split}
\text{(I)}&\le C\|\bar{\bm\nu}_{\alpha}\|_3\|\varphi\|_{C^3}\sum_{i\in\mathbb I}\int_\Omega \partial_\alpha\left(\frac{\rho_\alpha(\Theta_{\alpha,i})}{\rho_\alpha(\omega) |(f_\alpha^j)'(\Theta_{\alpha,i})|}\right) d\eta_{\alpha_0}(\omega)\\
& \le \underbrace{C\|\bar{\bm\nu}_{\alpha}\|_3\|\varphi\|_{C^3}\sum_{i\in\mathbb I}\int_\Omega \partial_\alpha\left(\frac{\rho_\alpha(\Theta_{\alpha,i})}{\rho_\alpha(\omega)}\right)\frac{1}{|(f_\alpha^j)'(\Theta_{\alpha,i})|} d\eta_{\alpha_0}(\omega)}_{\text{(Ia)}}\\
&\quad+\underbrace{C\|\bar{\bm\nu}_{\alpha}\|_3\|\varphi\|_{C^3}\sum_{i\in\mathbb I}\int_\Omega \frac{\rho_\alpha(\Theta_{\alpha,i})}{\rho_\alpha(\omega)}\partial_\alpha\left(\frac{1}{|(f_\alpha^j)'(\Theta_{\alpha,i})|}\right) d\eta_{\alpha_0}(\omega)}_{\text{(Ib)}}\\
 \end{split}
 \end{equation*}
 For  (Ia), we first consider
 $$\left(\frac{\dot{h}_{\alpha,i}\rho_{\alpha} -\dot{\rho}_{\alpha}h_{\alpha,i}}{\rho_{\alpha}^2}\right).$$
 Since $\Theta_{\alpha,i}\in[\frac12,1]$, $|\dot{h}_{\alpha,i}|<C$ and ${h}_{\alpha,i}<C$,
 \begin{equation}\label{eq:fun}
 \left(\frac{\dot{h}_{\alpha,i}\rho_{\alpha} -\dot{\rho}_{\alpha}h_{\alpha,i}}{\rho_{\alpha}^2}\right)\le C\omega^{\alpha}+C\omega^{2\alpha}\dot{\rho}_{\alpha}<C\omega^\alpha,
 \end{equation}
 where we used in the last step the result of \cite{BS16} to obtain $\omega^{\alpha}\dot{\rho}_\alpha<C$. Plugging the estimate \eqref{eq:fun} into (Ia), and using Lemma \ref{lem:sB}, we obtain 
$$\text{(Ia)}\le C\|\bar{\bm\nu}_{\alpha}\|_3\|\varphi\|_{C^3}\sum_{i\in\mathbb I}\omega^{\alpha+\varepsilon}\frac{1}{|(f_\alpha^j)'(\Theta_{\alpha,i})|}$$
 which is uniformly bounded by    \cite[Lemma 5.2.a)]{BS16}.
 
 For (Ib) we use the fact that $\Theta_{\alpha,i}\in[\frac12,1]$. Therefore, $\rho_\alpha(\Theta_{\alpha,i})<C$. By Lemma \ref{lem:sB} and    \cite[Lemma 5.2.c]{BS16} we obtain 
 $$\sum_{i\in\mathbb I} \omega^{\alpha+\varepsilon}\partial_\alpha\left(\frac{1}{|(f_\alpha^j)'(\Theta_{\alpha,i})|}\right)<C.$$
 For the second summand we have
 \begin{equation}\label{eq:twice}
 \begin{split}
\text{(II)}&\le  C\|\dot{\bar{\bm\nu}}_{\alpha}\|_3\|\varphi\|_{C^3} \sum_{i\in\mathbb I}\int_\Omega\frac{\rho_\alpha(\Theta_{\alpha,i})}{\rho_\alpha(\omega) |(f_\alpha^j)'(\Theta_{\alpha,i})|} d\eta_{\alpha_0}(\omega)\\
&\le C\|\dot{\bar{\bm\nu}}_{\alpha}\|_3\|\varphi\|_{C^3} \sum_{i\in\mathbb I}\int_\Omega\frac{\rho_\alpha(\Theta_{\alpha,i})}{\rho_\alpha(\omega) |(f_\alpha^j)'(\Theta_{\alpha,i})|} d\eta_{\alpha_0}(\omega)\\
&\le C\|\dot{\bar{\bm\nu}}_{\alpha}\|_3\|\varphi\|_{C^3} \sum_{i\in\mathbb I}\sup_{\omega\in(0,1]}\omega^{\alpha+\varepsilon}\frac{\rho_\alpha(\Theta_{\alpha,i})}{|(f_\alpha^j)'(\Theta_{\alpha,i})|},
 \end{split}
 \end{equation}
 where in the last step we used \eqref{eq:behave} and Lemma \ref{lem:sB}.
 Recalling \eqref{eq:undensity}, we have  $\rho_\alpha(\Theta_{\alpha,i})<C$ and the remaining sum is uniformly bounded by   \cite[Lemma 5.2.a]{BS16}.
 \end{proof}
 
The next result establishes condition~\ref{as.unfoldb}.

 \begin{lemma}\label{lem:undiff2}
For every $\bar{\bm\nu}\in \bar{\mathbf E}_3^0$, the map $\alpha \mapsto \mathcal{U}_\alpha( \bar{\bm{\nu}} )$ is continuous in $\mathbf{L}_3^0 $ at $\alpha_0$.
\end{lemma}
\begin{proof}
First notice that
 $$\alpha\to \frac{\rho_\alpha(\Theta_{\alpha,i})}{\rho_\alpha(\omega) |(f_\alpha^j)'(\Theta_{\alpha,i})|} (g^j_{\Theta_{\alpha,i}})_*\bar \nu_{\Theta_{\alpha,i}}$$
 is continuous in $\mathbf L_3^0$ at $\alpha_0$.
By the same argument as \eqref{eq:twice}, for all $\omega\in\Omega$, we have
$$
 \sum_{i\in\mathbb I}\int_\Omega\frac{\rho_\alpha(\Theta_{\alpha,i})}{\rho_\alpha(\omega) |(f_\alpha^j)'(\Theta_{\alpha,i})|}\langle  \bar \nu_{\Theta_{\alpha,i}},\varphi\circ g^j_{\Theta_{\alpha,i}}\rangle d\eta_{\alpha_0}(\omega)\le C,
 $$
 thus concluding the proof.
\end{proof}

\appendix
\section{Dual spaces}\label{se.appendix}

This appendix collects the definitions and conventions for the functional spaces and norms used in the analysis. 
These include Lipschitz and smooth observables, their dual spaces, and the Wasserstein--$1$ distance on probability measures, together with the associated push-forward operators. 
We adopt the standard duality notation 
\(\langle \nu, \phi \rangle\) 
for the action of a functional \(\nu\) on a test function \(\phi\). 
\medskip

\safeparagraph{\bf Lipschitz.}
We denote by
\[
\lip(X)
=
\bigl\{\, \phi : X \to \mathbb{R} \;\big|\; \lip(\phi) < \infty \,\bigr\}
\]
the space of real-valued Lipschitz functions on $X$, where the Lipschitz seminorm is defined by
\[
\lip(\phi)
=
\sup_{x \neq y} \frac{|\phi(x)-\phi(y)|}{d_X(x,y)} .
\]
We turn this seminorm into a genuine norm by fixing a base point $x_0 \in X$ and setting
\begin{equation}\label{eq.normlip}
\|\phi\|_{\lip}
=
|\phi(x_0)| + \lip(\phi).
\end{equation}
Equipped with this norm, $\lip(X)$ is a Banach space, and the resulting topology is equivalent to the usual one defined using $\|\phi\|_\infty + \lip(\phi)$.
We denote by $\lip(X)^*$ the topological dual of $\lip(X)$, endowed with the dual norm
\[
\|\nu\|_{\lip^*}
=
\sup_{\|\phi\|_{\mathrm L}\le 1} \langle \nu,\phi\rangle,
\qquad \nu \in \lip(X)^*.
\]
We further define $\lip_0(X)^*$ as the subspace of $\lip(X)^*$ consisting of functionals that vanish on constant functions. For $\nu \in \lip_0(X)^*$, adding a constant to $\phi$ does not affect the pairing, and therefore the norm reduces to
\begin{equation}\label{eq.normlip0}
\|\nu\|_{\lip^*}
=
\sup_{\lip(\phi)\le 1} \langle \nu,\phi\rangle.
\end{equation}

\medskip
\safeparagraph{\bf Wasserstein.} 
Consider
\[
\mathcal M_1(X)
=
\bigl\{ \text{Borel probability measures on } X \bigr\}.
\]
We equip $\mathcal M_1(X)$ with the Wasserstein--$1$ distance, defined via the \emph{Kantorovich--Rubinstein duality} as
\begin{equation}\label{eq.w1}
W_1(\mu,\nu)
=
\sup_{\mathrm{Lip}(\phi)\le 1}
\Bigl\{
\int \phi \, d\mu - \int \phi \, d\nu
\Bigr\},
\qquad \mu,\nu \in \mathcal M_1(X).
\end{equation}
Since $X$ is compact, the supremum is finite and $W_1$ is well defined. 
From the dual perspective adopted here, $W_1$ coincides with the norm of the signed measure $\mu-\nu$ viewed as an element of $\mathrm{Lip}_0(X)^*$:
\[
W_1(\mu,\nu) = \|\mu-\nu\|_{\mathrm{Lip}^*}.
\]
We will not appeal directly to the classical optimal transport formulation; instead, we systematically exploit this dual characterization, which is particularly convenient for studying push-forwards and parameter-dependent families of measures.  
Moreover, under our assumption that $X$ is compact, the metric space $(\mathcal M_1(X),W_1)$ is complete and compact, and the Wasserstein--$1$ distance metrizes the weak* topology on probability measures; see~\cite{V09}.

\medskip
\safeparagraph{\bf Smooth.}
For $k \ge 0$, let $C^k(X)$ denote the space of $k$-times continuously differentiable functions $\phi : X \to \mathbb{R}$. For $\phi \in C^0(X)$ we set
\[
\|\phi\|_{C^0} = \sup_{x \in X} |\phi(x)|.
\]
For $k \ge 1$, we introduce a norm analogous to~\eqref{eq.normlip}. Define the seminorm
\[
|\phi|_{C^k}
=
\sum_{j=1}^k \|D^j \phi\|_{C^0},
\]
and, fixing a base point $x_0 \in X$, define the full norm
\begin{equation}\label{eq.normck}
\|\phi\|_{C^k}
=
|\phi(x_0)| + |\phi|_{C^k}.
\end{equation}
With this norm, $C^k(X)$ is a Banach space, and~\eqref{eq.normck} is equivalent to the standard $C^k$ norm.
We denote by $C^k(X)^*$ the dual space of $C^k(X)$, endowed with the dual norm
\begin{equation}\label{eq.normckdual}
\|\nu\|_{k^*}
=
\sup_{\|\phi\|_{C^k}\le 1} \langle \nu,\phi\rangle,
\qquad \nu \in C^k(X)^*.
\end{equation}
Let $C^k_0(X)^* \subset C^k(X)^*$ be the subspace of functionals vanishing on constant functions. For $\nu \in C^k_0(X)^*$, the value of $\langle\nu,\phi\rangle$ depends only on $|\phi|_{C^k}$, and hence
\begin{equation}\label{eq.normckdual0}
\|\nu\|_{k^*}
=
\sup_{|\phi|_{C^k}\le 1} \langle \nu,\phi\rangle.
\end{equation}

\medskip
\safeparagraph{\bf Push-forward.}
The notion of push-forward     for measures extends naturally to the classes of linear functionals introduced above, both in the Lipschitz and in the smooth settings.
Let   $\nu$ be a continuous linear functional on a space of observables $\lip(X)$ or $C^k(X)$. We define the push-forward $h_*\nu$ by duality,
\begin{equation}\label{eq.push}
\langle h_*\nu,\phi\rangle
=
\langle \nu,\phi\circ h\rangle,
\end{equation}
whenever the composition $\phi\circ h$ belongs to the same class as $\phi$.

In particular, if $h$ is Lipschitz and $\nu\in \lip(X)^*$, then $h_*\nu\in \lip(X)^*$; if $h$ is of class $C^k$ and $\nu\in C^k(X)^*$, then $h_*\nu\in C^k(X)^*$. In both cases, the definition is consistent with the usual push-forward of measures when measures are viewed as elements of the corresponding dual spaces.
Moreover, since constant functions are preserved under composition, the push-forward leaves invariant the subspaces $\lip_0(X)^*$ and $C^k_0(X)^*$ of functionals vanishing on constant functions.



\end{document}